\documentclass[a4paper,12pt,reqno]{amsart}  %%%%% When compiling with pdfLaTeX

\usepackage{amsmath}
\usepackage[all]{xy}
\usepackage{tikz-cd}

\newif\iffastcompile
\fastcompilefalse

\iffastcompile
  \RenewDocumentEnvironment{tikzcd}{O{} +b}
    {\hbox{\fbox{\strut tikz-cd omitted in fast mode}}}{}
\fi

\usepackage{amssymb}
\usepackage{amsthm}
\usepackage{amsfonts}
\usepackage{mathtools}
\usepackage{comment}
\usepackage{mathrsfs}
\usepackage{pifont}
\usepackage{cite}
\usepackage{enumerate}
\usepackage{here}
\usepackage{autobreak}
\usepackage{wrapfig}
\usepackage{graphicx}
\usepackage{lscape}
\usepackage[colorlinks]{hyperref}  %%%%%%%%%% pdfLaTeX ver
\usepackage{xcolor}
\hypersetup{
	bookmarksnumbered=true,
    colorlinks=true,
    citecolor=blue,
    linkcolor=purple,
    urlcolor=orange,
}
\usepackage{mleftright} % \mleft \mrightの使用
\usepackage{extpfeil} % use for \xtwoheadrightarrow
\usepackage{tikz}
\usetikzlibrary{arrows} % arrowの種類を増やす
\usetikzlibrary{cd, decorations.pathmorphing} %tikzcdの使用

\usepackage{pgfplots}
\pgfplotsset{compat=1.18}

\usepackage[capitalize, nameinlink]{cleveref}
\everymath{\displaystyle}
\numberwithin{equation}{section}

\newcommand{\todo}{\textcolor{red}{TODO}} %% TODO
\newcommand{\redtext}[1]{\textcolor{red}{#1}}
\newcommand{\memo}[1]{\redtext{[{#1}]}} %% \memo{#1}

\theoremstyle{plain}

\newtheorem{thm}{Theorem}[section]

\newtheorem{question}[thm]{Question}
\newtheorem{lem}[thm]{Lemma}
\crefname{lem}{Lemma}{Lemmas}
\newtheorem{prop}[thm]{Proposition}
\crefname{prop}{Proposition}{Propositions}
\newtheorem{cor}[thm]{Corollary}
\crefname{cor}{Corollary}{Corollaries}

\newtheorem*{claim*}{Claim}
\newtheorem*{thm*}{Theorem}

\newtheorem{introthm}{Theorem}[section]

\newtheorem{introprop}[introthm]{Proposition}

\theoremstyle{definition}

\newtheorem{dfn}[thm]{Definition}

\newtheorem{eg}[thm]{Example}
\crefname{eg}{Example}{Examples}
\newtheorem*{Ack}{Acknowledgement}

\newtheorem*{Out}{Outline of this paper}

\newtheorem*{AI}{AI Usage}

\theoremstyle{remark}
\newtheorem{rem}[thm]{Remark}

\usepackage{enumitem} %% [label={(M\arabic*)}] の使用

\DeclareMathOperator{\Aut}{Aut}

\DeclareMathOperator{\ord}{ord}

\DeclareMathOperator{\Ker}{Ker}

\DeclareMathOperator{\id}{id}

\DeclareMathOperator{\GL}{GL}

\DeclareMathOperator{\Image}{Im}

\DeclareMathOperator{\End}{End}

\DeclareMathOperator{\pr}{pr}
\DeclareMathOperator{\sgn}{sgn}

\DeclareMathOperator{\AGL}{AGL}
\DeclareMathOperator{\Eq}{Eq}

\newcommand{\twoheadlongrightarrow}{\relbar\joinrel\twoheadrightarrow}

\newcommand\dual{\raise0.9ex\hbox{$\scriptscriptstyle\vee$}}

\newcommand{\catname}[1]{\mathsf{#1}}
\newcommand{\Set}{\catname{Set}}

\newcommand{\Grp}{\catname{Grp}}

\newcommand{\GrpAut}{\catname{GrpAut}}

\newcommand{\op}{\mathrm{op}} % opposite symbol

\DeclareMathOperator{\Map}{Map} % set of maps
\DeclareMathOperator{\Inn}{Inn} % Inn group of quandles
\DeclareMathOperator{\Dis}{Dis} % Displacement group of quandles
\DeclareMathOperator{\Trans}{Trans} % transvection group of quandles
\DeclareMathOperator{\relInn}{Inn_{rel}} % relative inner automorphism group
\DeclareMathOperator{\relTrans}{Trans_{rel}} % relative transvection group
\DeclareMathOperator{\relAut}{Aut_{rel}} % relative automorphism group
\newcommand{\relInnmodTrans}{H} % relInn / relTrans

\newcommand{\Quandle}{\catname{Qnd}}% the cat of quandles
\DeclareMathOperator{\Conj}{Conj} % conjugacy quandle of a group
\DeclareMathOperator{\Fix}{Fix} % fixed element of a group auto
\DeclareMathOperator{\GAlex}{GAlex} % generalized Alexander quandle
\DeclareMathOperator{\Alex}{Alex} % Alexander quandle

\newcommand{\card}[1]{\lvert{#1}\rvert} % cardinal of set

\newcommand{\covfac}[1]{{#1}^{\mathrm{cov}}}

\newcommand{\FB}{\mathbb{F}}

\newcommand{\RB}{\mathbb{R}}

\newcommand{\ZZ}{\mathbb{Z}}

\newcommand{\EC}{\mathcal{E}}

\newcommand{\OO}{\mathscr{O}}

\author{Yuki Imamura}
\address[Y.Imamura]{Osaka Central Advanced Mathematical Institute, Osaka Metropolitan University, Osaka 558-8585, Japan}
\email{\href{mailto:u287972b@alumni.osaka-u.ac.jp}{u287972b@alumni.osaka-u.ac.jp}}

\author{Tomoki Yoshida}
\address[T.Yoshida]{Department~of~Mathematics, School~of~Science~and~Engineering, Waseda~University, Ohkubo~3-4-1, Shinjuku, Tokyo~169-8555, Japan}
\email{\href{mailto:tomoki_y@asagi.waseda.jp}{tomoki\_y@asagi.waseda.jp}}

\title{Relativization of Symmetries on Quandles}

\date{September 12, 2026}

\begin{document}

\begin{abstract}
This paper introduces relative versions of the inner automorphism group and the transvection group associated with surjective quandle homomorphisms. 
By using the relative inner automorphism group, we define a notion of \emph{connectedness} for surjective homomorphisms.
We characterize connected homomorphisms algebraically as quotient maps, and use the relative transvection group to establish a maximal \emph{connected--covering} factorization for arbitrary surjections. 
We also introduce \emph{strongly connected homomorphisms} via the relative transvection group, and classify them in terms of perfect normal subgroups of inner automorphism groups.

A later part of this paper studies surjective homomorphisms for which the relative inner automorphism group acts $2$-transitively on each fiber. Under this assumption, we classify the possible quandle structures of the finite fibers.
\end{abstract}

\maketitle

\setcounter{tocdepth}{2} %table of contents subsection
\tableofcontents

\setcounter{section}{-1}
\section{Introduction}
\label{section: introduction}

Quandles were introduced independently by Joyce and Matveev as algebraic structures associated with knots and their diagrams (\cite{joyce_1982_a_classifying_invariant_of_knots_the_knot_quandle,matveev_1982_distributive_groupoids_in_knot_theory}). 
Since then, they have found applications in knot theory, low-dimensional topology, and the study of set-theoretic solutions to the Yang--Baxter equation.

Fundamental examples of quandles include conjugacy quandles arising from groups, Alexander quandles, and quandles associated with Riemannian symmetric spaces.
These examples illustrate the diversity of quandles, which is substantially broader than that of groups. Consequently, even in the finite case, a complete classification of finite quandles remains out of reach.

To navigate this structural complexity, one fruitful approach is to regard quandles as algebraic counterparts of Riemannian symmetric spaces, as in the work of Tamaru and others (e.g. \cite{tamaru_2013_twopoint_homogeneous_quandles_with_prime_cardinality,ishihara_tamaru_2016_flat_connected_finite_quandles}).
From this viewpoint, one investigates quandles satisfying conditions defined in terms of their canonical symmetries $s_x$ and the associated transformation groups, most notably the inner automorphism group $\Inn(P)=\langle s_x\mid x\in P \rangle$ and the transvection group $\Trans(P)=\langle s_x s_y^{-1} \mid x,y\in P \rangle$ (or displacement group).
Examples of such properties include connectedness \cite{hulpke_stanovsky_vojtv_2016_connected_quandles_and_transitive_groups},
double transitivity \cite{tamaru_2013_twopoint_homogeneous_quandles_with_prime_cardinality,wada_2015_twopoint_homogeneous_quandles_with_cardinality_of_prime_power,vendramin_2017_doubly_transitive_groups_and_cyclic_quandles},
flatness \cite{ishihara_tamaru_2016_flat_connected_finite_quandles,saito_sugawara_2025_homogeneous_quandles_with_abelian_inner_automorphism_groups}, 
the existence of antipodal subsets \cite{kubo_nagashiki_okuda_2022_a_commutativity_condition_for_subsets_in_quandlesa_generalization_of_antipodal_subsets}, 
and homogeneity \cite{furuki_tamaru_2024_homogeneous_quandles_with_abelian_inner_automorphism_groups_and_vertextransitive_graphs}.

The present paper pursues this approach in a relative setting. 
Instead of studying a single quandle in isolation, we focus on the ``vertical'' inner symmetries associated with a surjective quandle homomorphism.

Let $f\colon P\twoheadrightarrow Q$ be a surjective quandle homomorphism. 
To develop a relative theory, we single out the subgroup of the inner automorphisms of $P$ that become trivial after passing to the inner automorphisms of $Q$.
Although the assignment $P\mapsto\Inn(P)$ is not functorial on the entire category of quandles, it becomes functorial when restricted to the category of quandles and surjective homomorphisms (see \cref{proposition: functoriality of Inn for surjections}). 
Therefore, a surjective quandle homomorphism $f\colon P\twoheadrightarrow Q$ induces a canonical epimorphism $f_*\colon \Inn(P)\twoheadrightarrow\Inn(Q)$.
We define the \emph{relative inner automorphism group} $\relInn(f)$ of $f$ by 
\[
    \relInn(f)\coloneqq \Ker(f_*\colon \Inn(P)\twoheadrightarrow \Inn(Q)).
\]
The group $\relInn(f)$ consists precisely of those inner automorphisms of $P$ that preserve every fiber $f^{-1}(q)$ for $q\in Q$.

This naturally leads to a relative notion of connectedness.
We call $f$ a \emph{connected homomorphism} if the action of $\relInn(f)$ on each fiber of $f$ is transitive.
When $Q$ is the one-point terminal quandle $\{*\}$, this definition recovers the standard notion of connectedness for the quandle $P$.

Crucially, this relative connectedness is not determined solely by the intrinsic quandle structures of the fibers. Indeed, relative inner automorphisms can act transitively on a fiber even when the fiber itself is a trivial subquandle. Thus, our approach captures the ``extrinsic'' symmetries imposed by the global structure of $P$ relative to $Q$.

This perspective is reminiscent of the study of morphisms with connected fibers in topology and algebraic geometry, where relative connectedness is closely related to factorization phenomena.
A similar factorization can be obtained in the quandle setting.
%We prove that the relative connectedness of a surjective homomorphism can be characterized in terms of a quotient of normal subgroups of the inner automorphism group.
Using a characterization of relative connectedness in terms of quandle quotients, we establish that every surjective quandle homomorphism admits a canonical \emph{connected--rigid} factorization and, more broadly, a \emph{connected--covering} factorization.

Another interesting point of relativization is, for example, unlike in the absolute setting, the relative inner automorphism group and the relative transvection group need not have the same orbits on the fibers. This distinction leads to the notion of strongly connected homomorphisms and also explains the instability of connectedness under base change.

Taken together, these results establish a relative framework for the study of quandle homomorphisms, shifting the focus from absolute properties of individual quandles to structural properties of surjective morphisms.

\subsection{Results}
\label{subsection: results in intro}

Inspired by the theory of morphisms of schemes, we introduce and investigate connected quandle homomorphisms.
This notion is genuinely relative.
For a surjective homomorphism $f\colon P \twoheadrightarrow Q$, if every fiber of $f$ is connected as a subquandle of $P$, then $f$ turns out to be connected (\cref{remark: fiberwise connected is connected}).
However, the converse does not hold: a connected homomorphism may have disconnected, or even trivial, fibers (\cref{example: connected morphism whose fiber is trivial,example: doubly trans morph and triv and connected fiber}).
This is because the relative inner automorphism group can act on a fiber more transitively than the intrinsic inner automorphism group of the fiber itself.

\begin{comment}
We begin by establishing basic properties of connected homomorphisms.
First, we see that the notion is genuinely relative. 
If every fiber of $f$, regarded as a subquandle of $P$, is connected, then $f$ is connected.
The converse, however, need not hold.
A connected homomorphism may have disconnected, or even trivial fibers (\cref{example: connected morphism whose fiber is trivial,example: doubly trans morph and triv and connected fiber}) because the relative inner automorphism group may act on a fiber more transitively than the intrinsic inner automorphism group of the fiber itself.
Note that the degeneration phenomena do not occur under the condition of connectedness of the target quandle (\cref{proposition: if Q is connected all fibers are isomorphic}).
\end{comment}
%% この段落はoutlineでもいいかも
%We first show that the class of connected homomorphisms exhibits strong categorical closure properties; specifically, it is closed under composition (\cref{proposition: connectedness composition and 2 out of 3}), pullbacks along surjections (\cref{proposition: connectedness regarding pull-backs}), and products (\cref{proposition: product of connected homs is connected}).
%In particular, the absolute connectedness of a quandle is shown to lift along connected homomorphisms.
%Furthermore, generalizing the absolute setting, we establish that the relative connectedness of surjective homomorphisms descends along surjections (\cref{proposition: connectedness composition and 2 out of 3}).

A prominent feature of connected homomorphisms is their characterization as quotient maps by normal subgroups of the inner automorphism group. Given a normal subgroup $N$ of $\Inn(P)$ for a quandle $P$, the set $P/N$ of $N$-orbits has a canonical quandle structure induced by that of $P$, and we have the natural quotient quandle homomorphism $\pi_N\colon P \to P/N$. 
By construction, the relative inner automorphism group $\relInn(f)$ is a normal subgroup of $\Inn(P)$, naturally yielding such a quotient.
The following proposition establishes that connected homomorphisms are precisely characterized as the surjective homomorphisms arising in this manner.

\begin{introprop}[\cref{proposition:connected_iff_quotient_by_normal_subgroup_of_Inn}]
\label{introprop:connected_iff_quotient_by_normal_subgroup_of_Inn}
    A surjective quandle homomorphism $f\colon P \to Q$ is connected if and only if there is some normal subgroup $N \subseteq \Inn(P)$ such that $f$ is isomorphic to $\pi_N\colon P \to P/N$ (as quotients of $P$).
\end{introprop}

From this \cref{introprop:connected_iff_quotient_by_normal_subgroup_of_Inn}, we find that our class of connected homomorphisms coincides with the class $\EC_{1}$ introduced by Even and Gran in \cite[Definition 1.5]{even_gran_2014_on_factorization_systems_for_surjective_quandle_homomorphisms}. 
Hence, from the present relative viewpoint, the factorization theorem given in \cite[Proposition 3.1]{even_gran_2014_on_factorization_systems_for_surjective_quandle_homomorphisms}, or essentially \cite[Theorem 8.1]{bunch_lofgren_rapp_yetter_2010_on_quotients_of_quandles}, can be restated as follows:

\begin{introthm}[quandle Stein factorization, \cref{theorem: quandle stein factorization}, {\cite[Theorem 8.1]{bunch_lofgren_rapp_yetter_2010_on_quotients_of_quandles}}]
%{\cite[Proposition 3.1]{even_gran_2014_on_factorization_systems_for_surjective_quandle_homomorphisms}},
%{\cite[Theorem 8.1]{bunch_lofgren_rapp_yetter_2010_on_quotients_of_quandles}}
\label{theorem in intro: quandle stein factorization}
    Let $f \colon P\to Q$ be a surjective quandle homomorphism.
    Then, $f$ has a factorization as $f=h\circ g$, where $g$ is connected and $h$ is rigid (i.e. $\relInn(h)$ is trivial); especially, $h$ is covering.
\end{introthm}

We also define the notion of a homogeneous homomorphism for a quandle surjection, using the relative automorphism group. We observe analogous quotient characterization and factorization results for homogeneous homomorphisms; see \cref{proposition:homogeneous_iff_quotient_by_subgroup_normalized_by_Inn,proposition: homogeneous-rigid_factorization} for the precise statements. From this relative perspective, the well-known characterization of homogeneous quandles as the quandles associated with quandle triplets can be reformulated as the assertion that every homogeneous quandle is a homogeneous quotient of a generalized Alexander quandle; see \cref{example: quandle_triplet_quandle_is_homogeneous_quotient_of_GAlex}.

We next introduce a relative version of the transvection group. 
There are two natural candidates for this relative notion: one is the kernel of the induced epimorphism $f_* \colon \Trans(P) \to \Trans(Q)$, and the other is the subgroup of $\Inn(P)$ generated by transvections associated with pairs of points in the same fiber.

Motivated by Eisermann's theory of quandle coverings (\cite{eisermann_2014_quandle_coverings_and_their_galois_correspondence}), we adopt the latter approach. Namely, for a surjective homomorphism $f\colon P\twoheadrightarrow Q$, we define the \emph{relative transvection group} of $f$ by 
\[
\relTrans(f)\coloneqq \langle\{s_xs_y^{-1} \mid x, y\in P \text{ such that } f(x)=f(y) \}\rangle.
\]
Recall that a surjective homomorphism $f\colon P\to Q$ is a \emph{covering homomorphism} if $f(x) = f(y)$ implies $s_x = s_y$ for all $x, y\in P$.
Therefore, by definition, $\relTrans(f)$ measures precisely the obstruction to $f$ being a covering homomorphism. 
Indeed, the vanishing of the relative transvection group characterizes covering homomorphisms (\cref{lemma: relTrans vanish iff covering}), while the kernel of the induced map between transvection groups does not (\cref{example: ker of trans does not characterize covering}).

We show that this relative transvection group is useful for detecting a canonical covering part of an arbitrary surjection. 
As the quotient of $P$ by $\relTrans(f)$ produces a canonical covering associated with every surjection, every surjective homomorphism $f$ admits a canonical decomposition into a connected homomorphism and a covering homomorphism that satisfies a universal property.

\begin{introthm}[\cref{proposition: universality of connected-covering decomposition,proposition: connected covering decomposition via quot of relTrans}]
\label{theorem in intro: maximal connected covering decomposition}
    Let $f\colon P\twoheadrightarrow Q$ be a surjective quandle homomorphism and $N=\relTrans(f)$. Then, $f$ decomposes as 
        \[\begin{tikzcd}[column sep=small]
        P \arrow[two heads]{rr}{f} \arrow[two heads]{rd}[swap]{\pi_N} & & Q\rlap{,} \\
        & P/N \arrow[two heads]{ru}[swap]{h_N}
    \end{tikzcd}\]
    where $\pi_N$ is connected and $h_N$ is covering,
    and this is universal among all factorizations with the second morphism a covering homomorphism.
    %Moreover, this factorization is universal among such factorizations.
    %Namely, if $f$ factors as $f=h\circ \pi$ for an arbitrary homomorphism $\pi\colon P \to R$ and a covering homomorphism $h\colon R\to Q$, then there exists a unique quandle homomorphism $u\colon P/N \to R$ such that $\pi = u\circ \pi_N$ and $h_N= h \circ u$.
\end{introthm}

%For a quandle surjection $f\colon P\twoheadrightarrow Q$, we refer to the covering homomorphism $h_N\colon P/\relTrans(f) \to Q$ established in \cref{theorem in intro: maximal connected covering decomposition} as the \emph{covering factor} of $f$, denoted by $\covfac{f}$.
%The 2-out-of-3 discussion for connected homomorphisms (\cref{proposition: connectedness composition and 2 out of 3}) shows that the connectedness of $f$ implies the connectedness of $\covfac{f}$. 

It is worth remarking that unlike in the absolute case, where the transitivity of $\Inn(P)$ is equivalent to that of $\Trans(P)$ (\cref{lemma: connectedness is also determined by trans}), this equivalence fails in the relative setting: $\relInn(f)$ may be transitive even when $\relTrans(f)$ is not (\cref{example:connected_but_not_strongly_connected}).
This leads to the notion of strong connectedness: a surjective quandle homomorphism $f$ is said to be \emph{strongly connected} if $\relTrans(f)$ acts transitively on every fiber.
Because $\relTrans(f)$ is a normal subgroup of $\relInn(f)$, strongly connected homomorphisms are connected.

Strongly connected homomorphisms also admit a quotient-theoretic characterization in terms of \emph{\(\Inn(P)\)-perfect} normal subgroups.

\begin{introthm}[\cref{thm:characterize_strongly_connected,cor:strongly_connected_hom_are_in_bijection_to_perfect_subgroup}]
\label{theorem in intro: quotient characterization of strongly connected hom}
    A surjective quandle homomorphism $f\colon P\twoheadrightarrow Q$ is strongly connected if and only if it is isomorphic to the quotient $\pi_N$ by a normal subgroup $N\subseteq \Inn(P)$ that is $\Inn(P)$-perfect (i.e.\ $[\Inn(P),N]=N$).
    Moreover, there is a one-to-one correspondence between strong connected quotients of $P$ and $\Inn(P)$-perfect normal subgroups of $\Inn(P)$.
\end{introthm}

The quotient 
\[
    H_f\coloneqq \relInn(f)/\relTrans(f)
\]
measures the residual relative symmetries that are not generated by relative transvections.
We show that $H_f$ is a central subgroup of $\Inn(P)/\relTrans(f)$; in particular, $H_f$ is always abelian.
More precisely, we prove the following.

\begin{introprop}[\cref{cor:central_extension_associated_with_surjective_homs}]
\label{proposition in intro: central extension sequence of H_f}
    For a surjective quandle homomorphism $f\colon P\to Q$, there is a short exact sequence 
    \[ 1 \longrightarrow H_f \longrightarrow \Inn(P)/\relTrans(f) \longrightarrow \Inn(Q) \longrightarrow 1, \]
    which is a central group extension.
\end{introprop}

Moreover, when $f$ is connected, $H_f$ is canonically isomorphic to the $\Inn(P)$-relative abelianization of $\relInn(f)$: more precisely, $H_f\cong \relInn(f)/[\Inn(P), \relInn(f)]$ (\cref{proposition: Hf and its vanishing for connected homs}).

We further determine the obstruction to preserving connectedness under surjective base change.
Strongly connected homomorphisms are preserved and reflected by surjective base change and are closed under products, whereas connected homomorphisms need not be preserved under either operation; see \cref{subsection:pullbacks_for_connectd_and_strongly_connected_hom}.

In a later part of this paper, we explore a substantially stronger symmetry condition.
In the absolute setting, the study of \emph{doubly transitive} quandles (quandles whose inner automorphism group acts $2$-transitively) was initiated by Tamaru \cite{tamaru_2013_twopoint_homogeneous_quandles_with_prime_cardinality}.
This line of research is motivated by the classification of two-point homogeneous Riemannian manifolds, viewing quandles as algebraic counterparts of Riemannian symmetric spaces.
Given that a connected Riemannian manifold is two-point homogeneous if and only if it is isometric to $\RB^n$ or a rank-one symmetric space, doubly transitive quandles serve as natural candidates for quandle analogues of such spaces.
Subsequent works of Wada and Vendramin have established classification results for finite doubly transitive quandles \cite{vendramin_2017_doubly_transitive_groups_and_cyclic_quandles,
wada_2015_twopoint_homogeneous_quandles_with_cardinality_of_prime_power}.%; see \cref{theorem: classification of doubly transitive quandles} for the form used in this paper.

To extend this viewpoint to the relative setting, we say that a surjective quandle homomorphism $f$ is \emph{doubly transitive} if the action of $\relInn(f)$ on each fiber is $2$-transitive.
This condition is substantially stronger than connectedness and imposes severe restrictions on the intrinsic quandle structures of the fibers.
Nevertheless, a fiber of a doubly transitive homomorphism may not be a doubly transitive quandle as a subquandle.

Our classification theorem is as follows.

\begin{introthm}[\cref{theorem:classification_of_doubly_transitive_hom_with_finite_fibers}]
\label{theorem in intro: classification of fibers of doubly trans morph}
    Let $f\colon P\to Q$ be a doubly transitive homomorphism with finite fibers.
    Then each fiber $F_q$ is either a trivial quandle or a connected Alexander quandle of prime-power cardinality.
    %Moreover, for every quandle $F$ that is either a trivial quandle of prime power cardinality or a connected Alexander quandle of prime power cardinality, there exists a doubly transitive homomorphism $f\colon P\to Q$ and an element $x\in Q$ such that the fiber $f^{-1}(x)$ is isomorphic to $F$.
\end{introthm}

The proof of \cref{theorem in intro: classification of fibers of doubly trans morph} proceeds as follows.
For a point $q\in Q$, we consider the group $G\coloneqq \langle \rho_q(\relInn(f)), \Inn(F_q) \rangle$, which acts $2$-transitively on the fiber $F_q$.
Then, a group-theoretic analysis shows that the quandle $F_q$ is either trivial or connected (\cref{theorem: each fiber is trivial or connected as a subquandle}).
Next, using results from finite group theory, we show in \cref{proposition: sufficient cond for alex when doubly trans containing inns} that if a finite quandle has a $2$-transitive action of a certain group $G$ satisfying $\Inn(F_q)\subseteq G \subseteq \Aut(F_q)$, then it is an Alexander quandle of prime power order. 
Combining this proposition with the discussion surrounding Vendramin's classification result \cite{vendramin_2017_doubly_transitive_groups_and_cyclic_quandles}, we obtain the desired assertion when $\card{F_q} \geq 4$ (\cref{theorem: fiber of doubly transitive is connected Alexander quandle}).
%Note that, as trivial quandles are Alexander, both cases are Alexander.

We give an explicit example showing that both cases in this classification can actually occur as fibers of a doubly transitive homomorphism:

\begin{introprop}[\cref{example: doubly trans morph and triv and connected fiber}]
\label{proposition in intro: construct doubly trans morph of possible fibers}
    Let $\FB$ be a finite field, $\lambda\in \FB^\times$, and $L_\lambda$ the multiplication by $\lambda$.
    Then there exists a doubly transitive homomorphism having a fiber isomorphic to the Alexander quandle $\Alex((\FB,+), L_\lambda)$.
    Consequently, both trivial fibers and non-trivial connected Alexander fibers occur.
\end{introprop}

\begin{Out}

In \cref{section: preliminaries on quandles}, we recall basic facts about quandles and introduce the examples and notation used later in the paper.
%In \cref{section: preliminaries on quandles}, we recall basic facts about quandles, with particular emphasis on properties defined in terms of the canonical symmetries of a quandle, such as inner automorphisms and transvections. 
We also recall classification results on doubly transitive and doubly homogeneous quandles.
In \cref{subsection: quandle_congruence_and_quotients}, we review the notion of quandle congruences and quandle quotients by group actions.

\cref{section: relative_inner_auto_group} is devoted to relative inner automorphism groups and connected homomorphisms. We establish basic properties of connected homomorphisms, including their behavior under composition and the characterization result (\cref{introprop:connected_iff_quotient_by_normal_subgroup_of_Inn}) in \cref{subsection: connected_homomorphisms}.
We also introduce the notion of rigid homomorphisms and see the quandle Stein factorization theorem (\cref{theorem in intro: quandle stein factorization}) in \cref{subsection: rigidity quotients of quandles and factorization}.
As a notion similar to connected homomorphisms, homogeneous homomorphisms are introduced and studied in \cref{subsection:relative_homogeneous}.

Subsequently, the relative version of the transvection group will be defined in \cref{section: relative transvection group}. 
We introduce strongly connected homomorphisms and characterize strongly connected quotients in terms of $\Inn(P)$-perfect normal subgroups in \cref{subsection: strongly connected homomorphisms}.
Then, \cref{subsection: covering factor of a surjection} contains the proof of \cref{theorem in intro: maximal connected covering decomposition} and proves their orthogonality to covering homomorphisms. The following subsections study the group $H_f$ and pullbacks for connected and strongly connected homomorphisms by determining the obstruction to preserving connectedness.

In turn, in \cref{section: doubly transitive homomorphisms}, we discuss doubly transitive homomorphisms of quandles.
After preparing group-theoretic facts in \cref{subsection: group theoretic preparation}, we show \cref{theorem in intro: classification of fibers of doubly trans morph} in \cref{subsection: classification of doubly transitive action on fibers}.
%Having developed the relative theory of connectedness, rigidity, and coverings, we now turn to a substantially different question in \cref{section: doubly transitive homomorphisms}: the structure of surjective quandle homomorphisms whose relative inner automorphism groups act doubly transitively on the fibers.
%First, in \cref{subsection: group theoretic preparation}, we prepare some group theoretical facts, especially on a 2-transitive group action on some sets, including Burnside's theorem. 
%With these preparations, \cref{subsection: classification of doubly transitive action on fibers} shows that every fiber of a doubly transitive morphism is either a trivial quandle or a connected Alexander quandle of prime-power order.
%, that is \cref{theorem in intro: classification of fibers of doubly trans morph}.

The last \cref{section: examples and counterexamples} is devoted to the construction of several explicit examples and counterexamples, arising from conjugacy quandles and generalized (and linear) Alexander quandles.
Its main purpose is to illustrate the scope of the notions introduced in this paper, to exhibit homomorphisms satisfying the various relative properties considered above, and to show that several natural converses fail.

\begin{comment}    
To this end, we study three classes of typical surjective homomorphisms: homomorphisms between conjugacy quandles induced by group surjections (\cref{subsection: example cases of conjugacy quandles}), homomorphisms between generalized Alexander quandles induced by suitable group homomorphisms (\cref{subsection: generalized alexander quandles}), and homomorphisms between linear Alexander quandles (\cref{subsection: linear alexander quandles in examples section}). 
In each case, we give explicit descriptions of the relative inner automorphism group and the relative transvection group in terms of elementary number-theoretic or group-theoretic data.
\end{comment}
\end{Out}

\begin{AI}
The authors used GPT-5.6 Sol (OpenAI) and Gemini 3.1 Pro (Google) throughout the preparation of this work for mathematical discussions and for English proofreading.
All arguments and proofs were independently verified by the authors, who take full responsibility for the contents of this paper.
\end{AI}

\begin{Ack}
We are grateful to Sakumi Sugawara and Jumpei Yasuda for useful conversations.
%Y.I was supported by Grant \todo.
This work was supported by JSPS KAKENHI Grant Number JP25K23333 and JP26KJ0276.
The second author acknowledges support from OpenAI through the ChatGPT for Academic Researchers program.
\end{Ack}

%\newpage

\section{Preliminaries on Quandles}
\label{section: preliminaries on quandles}

\subsection{Quandles}
\label{subsection: Quandles}
    In this section, we review the notion of quandles.
    The reader may consult the very good monograph \cite[Part II]{preprint_valeriy_mohamed_mahender_2025_yangbaxter_equation_and_related_algebraic_structures} for much more information on this topic.
    
\begin{dfn}\label{definition: quandle}
    A \emph{quandle} $(Q, s)$ is a pair of a set $Q$ and a map $s\colon Q \to \Map(Q, Q)$ such that 
    \begin{enumerate}
        \item \label{condition:Q1} $s_x(x) = x$ for any $x\in Q$, 
        \item \label{condition:Q2} $s_x$ is bijective for any $x\in Q$, and
        \item \label{condition:Q3} $s_x\circ s_y = s_{s_x(y)}\circ s_x$ for any $x,y\in Q$.
    \end{enumerate}
\end{dfn}

\begin{dfn}\label{definition: quandle homomorphism}
    A \emph{quandle homomorphism} $f\colon (Q, s) \to (Q', s')$ between quandles $(Q, s)$ and $(Q', s')$ is a map
    $f\colon Q \to Q'$ such that the following diagram commutes for each $x\in Q$:
    \[\begin{tikzcd}
        Q \arrow{r}{f} \arrow{d}[swap]{s_x} & Q' \arrow{d}{s'_{f(x)}} \\
        Q \arrow{r}[swap]{f} & Q' \rlap{.}
    \end{tikzcd}\]
    The category of quandles and their homomorphisms will be denoted by $\Quandle$.
\end{dfn}

The condition~\eqref{condition:Q3} says each symmetry map $s_x$ is a quandle homomorphism $Q\to Q$.

\begin{eg}
    Every set $X$ has the quandle structure by setting $s_x = \id$ for $x\in X$. We call it the \emph{trivial quandle structure} on $X$.
\end{eg}

Giving trivial structures, we have a canonical fully faithful functor $\iota\colon \Set \hookrightarrow \Quandle$.

\begin{dfn}
\label{definition: inner, aut, connected, homogeneous}
    Let $Q$ be a quandle.
    \begin{enumerate}
        \item The group of bijective endomorphisms of $Q$ is denoted by $\Aut(Q)$.
        \item The \emph{inner automorphism group} $\Inn(Q)$ of $Q$ is the subgroup generated by $\{s_x\}_{x\in Q}$, that is, 
        \[
        \Inn(Q) = \langle \{s_x \mid x\in Q\} \rangle.
        \]
        \item The \emph{transvection group} $\Trans(Q)$ of $Q$ is the subgroup generated by elements of the form of $s_x\circ s_y^{-1}$ for $x, y\in Q$, that is, 
        \[
        \Trans(Q) = \langle \{s_x \circ s_y^{-1} \mid x, y\in Q\} \rangle.
        \]
        We remark that some people refer to $\Trans(Q)$ as the displacement group.
        %\item The \emph{displacement group} $\Dis(Q)$ of $Q$ is the subgroup generated by elements of the form of $s_x\circ s_y$ for $x, y\in Q$, that is, 
        %\[ \Dis(Q) =\langle\{ s_x\circ s_y \mid x, y\in Q \}\rangle. \]
        \item A quandle $Q$ is called \emph{homogeneous} if the natural action of $\Aut(Q)$ on $Q$ is transitive.
        \item A quandle $Q$ is called \emph{connected} if the natural action of $\Inn(Q)$ on $Q$ is transitive.
    \end{enumerate}
\end{dfn}

\begin{rem}
    We do not think of the empty quandle as being connected or homogeneous.
\end{rem}

For a general quandle $Q$, we have inclusions
\[ \Trans(Q) \subseteq \Inn(Q) \subseteq \Aut(Q). \]
In what follows, these groups are considered to act on $Q$ from the left via $(f,x) \mapsto f(x)$.

In fact, connectedness can equivalently be characterized in terms of the action of the transvection group.

\begin{lem}[{\cite[Proposition 2.1]{hulpke_stanovsky_vojtv_2016_connected_quandles_and_transitive_groups}}]
\label{lemma: connectedness is also determined by trans}
    A quandle $Q$ is connected if and only if the action of $\Trans(Q)$ on $Q$ is transitive. 
\end{lem}
\begin{comment}
\begin{proof}
    The ``if'' part is trivial, since $\Trans(P)\subseteq \Inn(P)$.
    To show the converse, take any $x, y\in P$. Since $P$ is connected, there exists a $f\in \Inn(P)$ such that $f(x)=y$. We can write it as $f = s_{a_1}^{\epsilon_1}\circ\cdots\circ s_{a_m}^{\epsilon_m}$, 
    where $a_i\in P$ and $\epsilon_i\in\{1, -1\}$ for $1\le i\le m$.
    Set 
    \begin{equation*}
        g_m = 
        \begin{cases}
            s_{a_m}\circ s_x^{-1}& \text{ if } \epsilon_m = 1\\
            s_{s_{a_m}^{-1}(x)}\circ s_{a_m}^{-1} & \text{ if } \epsilon_m = -1
        \end{cases}
    \end{equation*}
    and
    \[
    f_m\coloneqq s_{a_1}^{\epsilon_1}\circ\cdots\circ s_{a_{m-1}}^{\epsilon_{m-1}} \circ g_m.
    \]
    Furthermore, define
    \begin{equation*}
        g_{m-1} = 
        \begin{cases}
            s_{a_{m-1}}\circ s_{g_m(x)}^{-1}& \text{ if } \epsilon_m = 1\\
            s_{s_{a_{m-1}}^{-1}\circ g_m(x)}\circ s_{a_{m-1}}^{-1} & \text{ if } \epsilon_m = -1
        \end{cases}
    \end{equation*}
    and
    \[
    f_{m-1}\coloneqq s_{a_1}^{\epsilon_1}\circ\cdots\circ s_{a_{m-2}}^{\epsilon_{m-2}} \circ g_{m-1}\circ g_m.
    \]
    Repeating this procedure, we get $f_1\in\Trans(P)$ such that $f_1(x) = y$. 
    \memo{Make a more detailed explanation! I think the proof in the ref is easier to understand.}
\end{proof}
\end{comment}

\begin{dfn}\label{definition: connected component}
    A \emph{connected component} of a quandle $Q$ is an orbit of the action $\Inn(Q)$ on $Q$.
    For a quandle $Q$, we write $\pi_0(Q)$ for the set of all connected components.
\end{dfn}

\begin{eg}
    Let $Q$ be the trivial quandle with $n$ elements.
    Then, $\card{\pi_0(Q)} = n$.
\end{eg}

\begin{prop}
\label{prop:adjunction_between_Quandle_and_Set}
    The assignment $Q \mapsto \pi_0(Q)$ forms a left adjoint to the inclusion $\iota\colon \Set \hookrightarrow \Quandle$:
    \[
\begin{tikzcd}
    \Quandle \arrow[bend left=20]{rr}{\pi_0} \arrow[phantom]{rr}{\perp} & & 
    \Set. \arrow[bend left=20,hook]{ll}{\iota}
    %\arrow[start anchor=north west, end anchor=north east, bend right=30]{ll}[swap]{({-})^\text{disc}}{\perp}
    %\arrow[start anchor=south west, end anchor=south east, bend left=30 {ll}{({-})^\text{indisc}}[swap]{\perp}
\end{tikzcd}
\]
\end{prop}

In general, taking inner automorphism groups of quandles does not become a functor. However, $\Inn(-)$ behaves functorially for surjections.

\begin{prop}\label{proposition: functoriality of Inn for surjections}
Let $\pi\colon P \twoheadrightarrow Q$ be a surjective homomorphism of quandles.
Then, for every $\varphi\in \Inn(P)$, there exists a unique $\pi_*(\varphi) \in \Inn(Q)$ such that the square diagram
\[\begin{tikzcd}
    P \arrow{r}{\varphi} \arrow[two heads]{d}[swap]{\pi} & P \arrow[two heads]{d}{\pi} \\
    Q \arrow[dashed]{r}{\pi_*(\varphi)} & Q
\end{tikzcd}\]
commutes.
This assignment $\varphi\mapsto \pi_*(\varphi)$ gives rise to a surjective group homomorphism $\pi_*=\Inn(\pi) \colon \Inn(P) \twoheadrightarrow \Inn(Q)$ making the following diagram of maps commutative: 
\[\begin{tikzcd}
    P \arrow{r} \arrow[two heads]{d}[swap]{\pi} & \Inn(P) \arrow[dashed,two heads]{d}{\pi_\ast} \\
    Q \arrow{r} & \Inn(Q)\rlap{.}
\end{tikzcd}\]
\end{prop}

\begin{proof}
    The uniqueness follows from the surjectivity.
    The existence can be checked by the observation that $\pi_*(s_x)= s_{\pi(x)}$.
\end{proof}

Moreover, we obtain a functor $\Inn(-) \colon \Quandle^\mathrm{surj} \to \Grp^\mathrm{surj}$ from the category of quandles and surjective quandle homomorphisms to the category of groups and surjective group homomorphisms.

\begin{dfn}[{\cite[Definition 2.42]{eisermann_2014_quandle_coverings_and_their_galois_correspondence}}]
\label{definition: quandle covering}
    A quandle homomorphism $\pi\colon P\to Q$ is called \emph{covering} if $\pi$ is surjective and the following condition holds for all $x,y\in P$:
    \[
    \pi(x) = \pi(y) \text{ implies } s_x = s_y\in\Inn(P).
    \]
\end{dfn}

Note that covering homomorphisms are not closed under composition.

\begin{prop}[{\cite[Proposition 2.49]{eisermann_2014_quandle_coverings_and_their_galois_correspondence}}]
% \cite[Proposition 1.31]{darne_2026_nilpotent_quandles}
\label{proposition: covering hom is central extension}
        Let $\pi \colon P\to Q$ be a covering homomorphism. Then, $\pi_*\coloneqq \Inn(\pi) \colon \Inn(P) \to \Inn(Q)$ is a central extension.
\end{prop}

\subsection{Examples of Quandles}\label{subsection: examples of quandles}
In this subsection, we review some examples of quandles and fix notations that will be mainly used in \cref{section: examples and counterexamples}.

\begin{eg}[conjugacy quandle]
    Let $G$ be a group. For each $x\in G$, define a map $s_x\colon G\to G$ by 
    \[
    s_x(y)\coloneqq xyx^{-1} \qquad \text{ for } y\in G.
    \]
    Then, $(G, s)$ is a quandle. Indeed, the inverse of $s_x$ is given by 
    \[
    s_x^{-1}(z) = x^{-1}zx \qquad \text{ for } z\in G.
    \]
    We call this quandle the \emph{conjugacy quandle} of $G$ and denote it by $\Conj(G)$.
    If the group $G$ is abelian, the quandle structure of $\Conj(G)$ is trivial.
    Note that some authors instead use the map $s_x$ defined by $x\mapsto x^{-1}yx$.

    Moreover, we have a functor $\Conj\colon \Grp\to\Quandle$.
\end{eg}

\begin{eg}[quandle triplet, {\cite[Definition 3.1]{ishihara_tamaru_2016_flat_connected_finite_quandles}}]
\label{example: quandle triplet}
    Let $G$ be a group, $H$ a subgroup of $G$, and $\sigma$ a group automorphism of $G$.
    The triplet $(G, H, \sigma)$ is called a \emph{quandle triplet} if $H\subseteq \Fix(\sigma) \coloneqq \{g\in G \mid \sigma(g) = g\}$.

    For a quandle triplet $(G, H, \sigma)$, the set $G/H$ of cosets has a quandle structure given by 
    \[
    s_{[x]}([y]) = [x\cdot \sigma(x^{-1}y)] \qquad ([x],[y] \in G/H).
    \]
    For this quandle, we write $Q(G, H, \sigma)=(G/H,s)$.
\end{eg}

\begin{eg}[(generalized) Alexander quandle]
\label{example: generalized alex and alex}
    The quandle constructed from a quandle triplet of the form $(G, \{e_G\}, \sigma)$ is called a \emph{generalized Alexander quandle}. More precisely, for a group $G$ and a group automorphism $\sigma\in \Aut(G)$, the generalized Alexander quandle $\GAlex(G, \sigma)$ is $G$ as a set with the quandle structure
    \[
    s_g(h) \coloneqq g\cdot\sigma(g^{-1}h)\qquad (g, h\in G).
    \]
    If a quandle $Q$ is isomorphic to a quandle $\GAlex(A,\sigma)$ for some abelian group $A$, then $Q$ is called an \emph{Alexander quandle} and denoted by $\Alex(A, \sigma)$.
\end{eg}

A morphism $\Phi\colon (G,\sigma)\to(H,\tau)$ of groups equipped with automorphisms is a group homomorphism $\Phi\colon G\to H$ satisfying $\Phi\circ\sigma=\tau\circ\Phi$. Observe that we have a functor $\GAlex(-) \colon \GrpAut \to \Quandle$ from the category of groups with automorphisms to the category of quandles.

\begin{eg}[linear Alexander quandle]
\label{example: linear Alexander quandle}
    Let $n$ be a positive integer and $a$ be an integer coprime to $n$.
    The Alexander quandle associated with the abelian group $\ZZ/n\ZZ$ together with the automorphism $x\mapsto ax$ is called a \emph{linear Alexander quandle}.
    In other words, a linear Alexander quandle (of type $(n,a)$) is the set $\ZZ/n\ZZ$ equipped with the quandle structure given by
    \[ s_x(y) = (1-a)x + ay \]
    for $x,y \in \ZZ/n\ZZ$.
    Such a quandle is denoted by $\Lambda_{n,a}$.

    For example, $\Lambda_{n, 1}$ is just the trivial quandle with $n$ elements; $\Lambda_{n, -1}$ is exactly what is called the \emph{dihedral quandle} $R_n$ of order $n$.
\end{eg}

We consider many such examples in \cref{section: examples and counterexamples}.

\subsection{Doubly Transitive Quandles}
\label{subsection: Doubly transitive quandles}

We recollect the classification result of \emph{doubly transitive quandles}
to compare the results that we will show in \cref{section: doubly transitive homomorphisms}.

\begin{dfn}[{\cite[Definition 3.1]{tamaru_2013_twopoint_homogeneous_quandles_with_prime_cardinality}}]
    A quandle $Q$ is said to be \emph{doubly transitive} (or \emph{two-point homogeneous}) if for all $(x_1, x_2), (y_1, y_2)\in Q\times Q$ satisfying $x_1\neq x_2$ and $y_1 \neq y_2$, there exists $f\in\Inn(Q)$ such that $(f(x_1), f(x_2)) = (y_1, y_2)$.
    This means precisely that $\Inn(Q)$ acts $2$-transitively on $Q$.
\end{dfn}

\begin{dfn}[{\cite[p.694]{bonatto_2020_principal_and_doubly_homogeneous_quandles}}]
\label{definition: doubly homogeneous quandle}
    A quandle $Q$ is said to be \emph{doubly homogeneous} if $\Aut(Q)$ acts $2$-transitively on $Q$.
\end{dfn}

The classification of finite doubly transitive quandles was given in \cite{wada_2015_twopoint_homogeneous_quandles_with_cardinality_of_prime_power} and \cite{vendramin_2017_doubly_transitive_groups_and_cyclic_quandles}.
For a ring $R$ and an element $a\in R$, let $L_a$ denote the left multiplication operator $x\mapsto ax$.

\begin{thm}[{\cite{wada_2015_twopoint_homogeneous_quandles_with_cardinality_of_prime_power,vendramin_2017_doubly_transitive_groups_and_cyclic_quandles}}]
\label{theorem: classification of doubly transitive quandles}
    A finite quandle $Q$ is doubly transitive if and only if 
    $Q$ is isomorphic to the Alexander quandle $\Alex(\FB_{p^k}, L_\omega)$, where $\FB_{p^k}$ is (the additive group of) the finite field with $p^k$ elements and $\omega$ is a primitive element of $\FB_{p^k}$.
\end{thm}

\begin{thm}[{\cite{bonatto_2020_principal_and_doubly_homogeneous_quandles}}]
\label{theorem: bonatto characterization of doubly homogeneous quandles}
    Let $Q$ be a finite connected Alexander quandle.
    Then, the following are equivalent:
    \begin{itemize}
        \item $Q$ is a doubly homogeneous quandle.
        \item $Q\cong \Alex(\FB_{p^m}^n, L_\lambda)$, where $p$ is a prime number and $\lambda\in\FB_{p^m}^{\times}\setminus\{1\}$.
    \end{itemize}
\end{thm}

\begin{proof}
    It follows by combining \cite[Corollary 3.6 and Theorem 4.12]{bonatto_2020_principal_and_doubly_homogeneous_quandles}.
\end{proof}

\subsection{Quandle Congruences and Quotients}
\label{subsection: quandle_congruence_and_quotients} 

Here we recall the notion of quandle congruence.

\begin{dfn}
    An equivalence relation $\sim$ on a quandle $Q$ is called a \emph{quandle congruence} (or simply a \emph{congruence}) if for any $x,x',y,y'\in Q$ the following conditions hold:
\begin{enumerate}[label=(C\arabic*)]
\item If $x\sim x'$ and $y\sim y'$, then $s_x(y)\sim s_{x'}(y')$,
\item If $x\sim x'$ and $y\sim y'$, then $s_x^{-1}(y)\sim s_{x'}^{-1}(y')$.
\end{enumerate}
\end{dfn}

\begin{prop}
    Let $\sim$ be a quandle congruence on a quandle $Q$. Then the quotient set $Q/{\sim}$ admits a quandle structure defined by
    $s_{[x]}([y])\coloneqq [s_x(y)]$, and the natural surjective map $\pi \colon Q \twoheadrightarrow Q/{\sim}$ is a quandle homomorphism.
    Moreover, if a quandle homomorphism $f\colon Q\to Q'$ satisfies $x\sim y \Rightarrow f(x)=f(y)$, then there exists a unique quandle homomorphism $\overline{f}\colon Q/{\sim}\to Q'$ making the following diagram commutative:
    \[\begin{tikzcd}
        Q \arrow[two heads]{r}{\pi} \arrow{rd}[swap]{f} & Q/{\sim} \arrow[dashed]{d}{\overline{f}} \\
        & Q'\rlap{.}
    \end{tikzcd}\]
\end{prop}

Surjective homomorphisms of quandles are sometimes called \emph{quandle quotients}.
The following is well-known.

\begin{prop}
    For any surjective homomorphism $f\colon Q\twoheadrightarrow R$, define a relation $\sim$ on $Q$ by
    \[ x \sim y \iff f(x)=f(y). \]
    Then $\sim$ is a quandle congruence, and there is an isomorphism $R\cong Q/{\sim}$ (as quandles quotients). This quandle congruence is called the \emph{kernel congruence} of $f$ and is denoted by $\Eq(f)$.

    This assignment gives a one-to-one correspondence (up to isomorphism) between quotient quandles of $Q$ and quandle congruences on $Q$.
\end{prop}

\begin{eg}
    The unit $Q\to \pi_0(Q)$ of the adjunction $\pi_0 \dashv \iota$ in \cref{prop:adjunction_between_Quandle_and_Set} is surjective. It is the quotient quandle corresponding to the quandle congruence induced by the orbit decomposition of the action of $\Inn(Q)$.
\end{eg}

For a quandle, one can also take a quotient by a group by constructing a quandle congruence from a group action.

\begin{prop}
\label{proposition:quandle_quotient_by_group}
    Let $Q$ be a quandle and let $G\subseteq \Aut(Q)$ be a subgroup. Define a relation $\sim_G$ on $Q$ by
    \[ x\sim_G y \iff \text{there exists } g\in G \text{ such that } y=g(x). \]
    Then the following are equivalent.
    \begin{enumerate}[label={(\roman*)}]
        \item \label{item:sim_N_is_congruence} $\sim_G$ is a quandle congruence.
        \item \label{item:sim_N_is_Inn(Q)-invariant} $\sim_G$ is \emph{$\Inn(Q)$-invariant}: namely, for all $\varphi\in \Inn(Q)$, $y\sim_G z$ implies $\varphi(y) \sim_G \varphi(z)$.
        \item \label{item:pointwise_normalized_by_Inn(Q)} For all $\varphi \in \Inn(Q)$ and $x\in Q$, we have $\varphi G \varphi^{-1}\cdot x \subseteq G\cdot x$.
    \end{enumerate}
    In this case, the corresponding quotient quandle will be denoted by $Q/G\coloneqq Q/{\sim_G}$.
\end{prop}

\begin{proof}
    \ref{item:sim_N_is_congruence} $\Rightarrow$ \ref{item:sim_N_is_Inn(Q)-invariant}: trivial.

    \ref{item:sim_N_is_Inn(Q)-invariant} $\Rightarrow$ \ref{item:pointwise_normalized_by_Inn(Q)}: Let $\varphi \in\Inn(Q)$, $x\in Q$, and $g\in G$. By definition, we have $g \varphi^{-1}(x) \sim_G \varphi^{-1}(x)$. Since $\sim_G$ is $\Inn(Q)$-invariant, it follows that $\varphi g \varphi^{-1}(x) \sim_G \varphi\varphi^{-1}(x)=x$. This means $\varphi g \varphi^{-1}(x) \in G\cdot x$.

    \ref{item:pointwise_normalized_by_Inn(Q)} $\Rightarrow$ \ref{item:sim_N_is_congruence}: Let $x,x',y,y' \in Q$ satisfy $x \sim_G x'$ and $y \sim_G y'$. Take $g,h \in G$ such that $g(x)=x'$ and $h(y)=y'$. Then, since $g$ is a quandle automorphism, 
    \[ s_{x'}(y') = s_{g(x)}(h(y)) = g s_x g^{-1} h(y) = g s_xg^{-1}h s_x^{-1}(s_x(y)) \sim_G s_xg^{-1}h s_x^{-1}(s_x(y)). \]
    By the assumption, we have $s_xg^{-1}h s_x^{-1}(s_x(y)) \in G \cdot s_x(y)$. Hence $s_{x'}(y') \sim_G s_x(y)$. Replacing $s_x$ by $s_x^{-1}$ in the above argument gives $s_{x'}^{-1}(y') \sim_G s_x^{-1}(y)$, and thus $\sim_G$ is a congruence.
\end{proof}

\begin{rem}
\label{remark_of_proposition:quandle_quotient_by_group}
    In particular, if a subgroup $G\subseteq \Aut(Q)$ is \emph{normalized by $\Inn(Q)$} (i.e.\ $\varphi G \varphi^{-1} \subseteq G$ for all $\varphi \in \Inn(Q)$), then the orbit equivalence relation $\sim_G$ is a quandle congruence (as observed in {\cite[Lemma 1.20]{darne_2026_nilpotent_quandles}}).
    As a special case, if $N\trianglelefteq \Inn(Q)$ is a normal subgroup, $\sim_N$ is a congruence (cf.~{\cite[Theorem 6.1]{bunch_lofgren_rapp_yetter_2010_on_quotients_of_quandles}}).
\end{rem}

\begin{eg}
    For a quandle $Q$, we have $\pi_0(Q)=Q/\Inn(Q)$.
\end{eg}

\begin{rem} 
\label{remark: reformulation of the necessity direction} 
The preceding theorem provides a reformulation of the necessity direction in \cite[Theorem 6.1] {bunch_lofgren_rapp_yetter_2010_on_quotients_of_quandles}. Such a reformulation is needed because the necessity statement, as stated there, requires further qualification. The following example illustrates this point. 
\end{rem}

\begin{eg}
    There is a subgroup $N \subseteq \Inn(Q)$ such that the orbit relation $\sim_N$ is a quandle congruence but $N$ is not a normal subgroup.
    
    Let $R_3$ be the dihedral quandle with three elements.
    Consider two copies of $R_3$, say $A=R_3=\{a_0,a_1,a_2\}$ and $B=R_3=\{b_0,b_1,b_2\}$, and let $Q=A \sqcup B$ be their disjoint union, where the symmetry map is the usual dihedral quandle operation on each component and is trivial between distinct components.
    
    The symmetry maps of $R_3$ are precisely the three transpositions $(01)$, $(02)$, and $(12)$. Since these transpositions generate the full symmetric group on $\{0,1,2\}$, we obtain $\Inn(R_3)=\langle s_0,s_1,s_2 \rangle \cong  S_3$.
    As the inner automorphisms of $A$ and $B$ act independently, we have
    \[ \Inn(Q)\cong \Inn(R_3)\times\Inn(R_3)\cong S_3\times S_3. \]
    
    Now consider the diagonal subgroup
    \[ N= \{ (g,g) \mid g \in S_3 \} \subseteq \Inn(Q), \]
    which is a subgroup of $S_3\times S_3$.
    Then the relation $\sim_N$ induced by $N$-orbits is a quandle congruence. To see this, let $x,x',y,y'\in Q$ satisfy $x\sim_N x'$ and $y\sim_N y'$. We show that $s_x(y)\sim_N s_{x'}(y')$ (the verification for $s_x^{-1}$ is analogous).
    First suppose that $x\sim_N y$. Then we also have $x'\sim_N y'$.
    Let $g,g'\in S_3$ be the permutations induced by the symmetry maps $s_x$ and $s_{x'}$ on the respective copies of $R_3$. Since $(g,g), (g',g') \in N$, we have
    \[ s_x(y)=(g,g)(y)\sim_N y\sim_N y'\sim_N (g',g')(y')=s_{x'}(y'). \]
    Next, suppose that $x\not\sim_N y$. Then also $x'\not\sim_N y'$. Since the action between different components is trivial, we have
    \[ s_x(y)=y\sim_N y'=s_{x'}(y'). \]
    Thus, we see that $\sim_N$ is a quandle congruence.

    However, $N$ is not a normal subgroup of $\Inn(Q)$; indeed, for $\varphi= ((012),\id) \in \Inn(Q)$ and $n=((01),(01)) \in N$, we have
    \[ \varphi n \varphi^{-1} = ((012),\id)\circ ((01),(01)) \circ ((021),\id) = ((12),(01)), \]
    which is not in $N$.
    %This example shows that the normality of a subgroup of $\Inn(Q)$ is not necessary for its orbit equivalence relation to be a quandle congruence.
\end{eg}

\begin{comment}
%% 古いやつ
\begin{lem}[{\cite[Lemma 1.20]{darne_2026_nilpotent_quandles}}]
\label{lemma: quotient of quandle by subgroup of Aut(Q)}
    Let $Q$ be a quandle and $G$ be a subgroup of $\Aut(Q)$ such that $\Inn(Q)\subset N_{\Aut(A)}(G)$.
    Then, the $G$-orbits decomposition of $Q$ is compatible with the quandle law. 
    The quandle of $G$-orbits will be denoted by $Q/G$.
\end{lem}

\begin{prop}[{\cite[Theorem 6.1]{bunch_lofgren_rapp_yetter_2010_on_quotients_of_quandles}}]
    Let $Q$ be a quandle and let $N\subseteq \Inn(Q)$ be a subgroup. Define a relation $\sim_N$ on $Q$ by
    \[ x\sim_N y \iff \text{there exists } g\in N \text{ such that } y=g(x). \]
    Then $\sim_N$ is a quandle congruence if and only if $N\subseteq \Inn(Q)$ is a normal subgroup. In this case, the corresponding quotient quandle is denoted by $Q/N\coloneqq Q/{\sim_N}$.
\end{prop}
\end{comment}

\section{Relative Inner Automorphism Groups}
\label{section: relative_inner_auto_group}

We introduce a relativization of inner automorphism groups of quandles.

\subsection{Definition and First Properties}
\label{subsection: definition and first properties of relative inner}

\begin{dfn}
    Let $f\colon P \to Q$ be a homomorphism of quandles.
    The \emph{relative inner automorphism group} of $f$, denoted by $\relInn(f)$, is a subgroup of $\Inn(P)$ defined by
    \[
    \relInn(f) \coloneqq \Ker(\Inn(P) \stackrel{f_{*}}{\longrightarrow} \Inn(\Image(f))).
    \]
\end{dfn}

\begin{eg}
    \begin{enumerate}
        \item For an injective homomorphism $f$, we have $\relInn(f) = \{1\}$.
        \item When $f\colon P \to \{\ast\}$ is the unique homomorphism into the one-point quandle, we have $\relInn(f)=\Inn(P)$.
    \end{enumerate}
\end{eg}

We will investigate more examples of relative inner automorphism groups in \cref{section: examples and counterexamples}.

\begin{dfn}
\label{definition:relative_automorphism_group}
For a quandle homomorphism $f\colon P \to Q$, we define the \emph{relative automorphism group} of $f$ as
    \[
    \relAut(f) \coloneqq \{\varphi\in\Aut(P) \mid f\circ\varphi = f \}.
    \]
\end{dfn}

\begin{prop}
\label{proposition: relinn preserves fibers}
    Let $f\colon P \to Q$ be a quandle homomorphism and $f'\colon P \to Q'\coloneqq \Image(f)$ be the codomain restriction of $f$ to its image.
    Then, the following hold.
    \begin{enumerate}
        \item\label{item: rel inn depends only on the image} $\relInn(f) = \relInn(f')$ holds. 
        \item\label{item:relInn_is_normal_in_Inn}  $\relInn(f)$ is a normal subgroup of $\Inn(P)$.
        \item\label{item:relInn_is_Inn_cap_relAut} We have $\relInn(f) = \Inn(P)\cap \relAut(f)$. Namely, an inner automorphism $\varphi\in \Inn(P)$ belongs to $\relInn(f)$ if and only if it preserves every non-empty fiber of $f$.
    \end{enumerate}
\end{prop}
\begin{proof}
    \eqref{item: rel inn depends only on the image}: By definition.
    \eqref{item:relInn_is_normal_in_Inn}: Since $\relInn(f)$ is defined as the kernel of a group homomorphism.
    \eqref{item:relInn_is_Inn_cap_relAut}: By the definition of $(f')_*$, an inner automorphism $\varphi\in \Inn(P)$ is in $\relInn(f)=\relInn(f')$ if and only if $(f')_*(\varphi)=\id_{Q'}$, i.e., $f'\circ \varphi=\id_{Q'} \circ f'=f'$.
    This shows the first assertion.
    The rest follows from this.
\end{proof}

\subsection{Connected Homomorphisms}
\label{subsection: connected_homomorphisms}

Let $\pi\colon P \twoheadrightarrow Q$ be a surjective quandle homomorphism and $q\in Q$ be an element. Then, by \cref{proposition: relinn preserves fibers}, every relative inner automorphism $\varphi\in \relInn(\pi)$ preserves the fiber $F_q\coloneqq \pi^{-1}(q)$. Hence, we obtain the action of $\relInn(\pi)$ on $\pi^{-1}(q)$ induced by the restriction map 
\begin{align*}
\rho_q\colon \relInn(\pi) \hookrightarrow\relAut(\pi) \to \Aut(F_q), \quad\varphi \mapsto \varphi\rvert_{F_q}.
\end{align*}
This enables us to define the notion of connected quandle homomorphisms.

\begin{dfn}
\label{definition: connectedness for morphism}
    Let $f \colon P\to Q$ be a quandle homomorphism.
    We call $f$ \emph{connected} if it is surjective and for each $q\in Q$, the action of $\relInn(f)$ on the fiber $f^{-1}(q)$ is transitive.
    In this case, we say that $P$ is \emph{$f$-connected}, and also that $P$ is \emph{connected over $Q$} when $f$ is clear.
\end{dfn}

\begin{eg}
    Every isomorphism between quandles is connected.
\end{eg}

\begin{eg}
    Consider the unique homomorphism $f\colon P\twoheadrightarrow\{\ast\}$ into the terminal quandle. Since $\Inn(\{\ast\})=\{\id_{\{\ast\}}\}$, we have $\relInn(f) = \Inn(P)$. Therefore, $f$ is connected if and only if $P$ is a connected quandle.
\end{eg}

\begin{eg}
    Let $\eta\colon Q\to \pi_0(Q)$ be the unit of the adjunction $\pi_0 \dashv \iota$ in \cref{prop:adjunction_between_Quandle_and_Set}. Then $\relInn(\eta)=\Inn(Q)$ since $\Inn(\pi_0(Q))=\{\id\}$, and the fibers of $\eta$ are the connected components of $Q$. Thus $\relInn(\eta)$ acts on each fiber transitively, and hence $\eta$ is connected.

    More precisely, we can see from \cref{prop:connected_hom_is_inverted_by_pi0} below that a surjective homomorphism $f\colon Q\to X$ to a trivial quandle $X$ is connected if and only if $X\cong \pi_0(Q)$.
\end{eg}

\begin{prop}
\label{prop:connected_hom_is_inverted_by_pi0}
    For a connected quandle homomorphism $f\colon P \to Q$, the map $\pi_0(f)\colon \pi_0(P) \to \pi_0(Q)$ is bijective.
\end{prop}

\begin{proof}
    Clearly $\pi_0(f)$ is surjective, as $f$ is so. For the injectivity, let $x,y \in P$ satisfying $[f(x)]=[f(y)]$. There is a $\varphi\in \Inn(Q)$ such that $\varphi(f(x))=f(y)$. Since $f_*\colon \Inn(P) \to \Inn(Q)$ is surjective, we can take $\psi\in \Inn(P)$ with $f_*(\psi)= \varphi$. Then $f(\psi(x))=f_*(\psi)(f(x)) = \varphi(f(x)) = f(y)$. By the connectedness of $f$, there is some $\xi\in\relInn(f)$ such that $\xi(\psi(x)) = y$. This means that $x$ and $y$ are in the same connected component, which shows that $\pi_0(f)$ is injective.
\end{proof}

Motivated by the role of varying fibers in geometric settings, we first note that surjective quandle homomorphisms admit neither variation nor degeneration of fiber type along an inner orbit of the target. 
More precisely, the fibers over points lying in the same connected component are mutually isomorphic.

\begin{prop}
\label{proposition: if Q is connected all fibers are isomorphic}
    Let $f \colon P\to Q$ be a surjective quandle homomorphism.
    Then, for any $p, q\in Q$ in the same $\Inn(Q)$-orbit, the fibers $F_p$ and $F_q$ are isomorphic to each other.
\end{prop}
\begin{proof}
    As $p,q\in Q$ is in the same orbit of the action $\Inn(Q)\curvearrowright Q$, there is $\psi\in\Inn(Q)$ such that $\psi(p) = q$.
    Since $f_*\colon \Inn(P) \to \Inn(Q)$ is surjective, one can take be an inner automorphism $\varphi\in\Inn(P)$ such that $f_{*}(\varphi) = \psi$.
    Then the image of $F_p$ under $\varphi$ is contained in the fiber $F_q$; indeed, for any $x\in F_p$, 
    \[
        f(\varphi(x)) = f_{*}(\varphi)(f(x)) = \psi(p) = q.
    \]
    Hence we obtain a homomorphism $\Phi\coloneqq \varphi\rvert_{F_{p}}\colon F_p \to F_q$. It is injective, as $\varphi\in\Aut(P)$.
    For any $y\in F_q$, putting $x\coloneqq \varphi^{-1}(y)$, we have
    \[
    f(x) = f(\varphi^{-1}(y)) = f_{*}(\varphi^{-1})(f(y)) = \psi^{-1}(q) = p,
    \]
    and $\Phi(x)=y$.
    Thus, $\Phi\colon F_p \to F_q$ is surjective, and hence we have $F_p \cong F_q$.
\end{proof}

%\todo The theorem above shows that a connected quandle homomorphism gives rise to a structure that can be viewed as a quandle fiber bundle. Therefore, we can also call connected homomorphisms \emph{quandle fiber bundles}.

Note that a surjective homomorphism to a connected quandle can be represented as a dynamical extension by \cite{andruskiewitsch_grana_2003_from_racks_to_pointed_hopf_algebras}.
%which will be referred to in \cref{subsection: prelim on Quandle extensions}.

\begin{rem}
\label{remark: fiberwise connected is connected}
The connectedness of a surjective quandle homomorphism $f\colon P\twoheadrightarrow Q$ can be verified fiberwise.
That is, if each fiber $F_q \coloneqq f^{-1}(q)$, seen as a subquandle of $P$, is connected, then $f$ is connected; see \cref{proposition:strongly_connected_can_be_verified_fiberwise} below.
However, the converse is not true in general: there is a connected homomorphism with a non-connected fiber; see \cref{example: connected morphism whose fiber is trivial}.
\end{rem}

\begin{prop}
\label{proposition: connectedness composition and 2 out of 3}
    Let $f \colon P\to Q$ and $g\colon Q \to R$ be surjective quandle homomorphisms, and put $h\coloneqq g\circ f$. Then, the following hold.
    \begin{enumerate}
        \item\label{item: inclusion of relInn} The inclusion $\relInn(f)\subseteq \relInn(h)$ holds. And we have an exact sequence
        \[\begin{tikzcd}[column sep=small]
            1 \arrow{r} & \relInn(f) \arrow{r} & \relInn(h) \arrow{r}{f_*} & \relInn(g) \arrow{r} & 1.
        \end{tikzcd}\]
        \item\label{item: composition is conn} If $f$ and $g$ are connected, then $h$ is connected.
        \item\label{item: g is conn} If $h$ is connected, then $g$ is connected.
    \end{enumerate}
\end{prop}
\begin{proof}
\eqref{item: inclusion of relInn}: The inclusion is straightforward.
The exact sequence follows by applying the third isomorphism theorem for groups.

\eqref{item: composition is conn}: %Note $h$ is surjective, as $f,g$ are so.
Let $r\in R$ and $x, y\in h^{-1}(r)$.
Since $f(x), f(y)\in g^{-1}(r)$ and $g$ is connected, there exists $\varphi\in\relInn(g)$ such that $\varphi(f(x)) = f(y)$. 
Then, by the surjectivity of $f$, there exists $\psi\in\Inn(P)$ such that $f_{*}(\psi) = \varphi$.
In this situation, we have 
\[
g_{*}\circ f_{*}(\psi) = g_{*}(\varphi) = \id_{R},
\]
so $\psi$ is actually in $\relInn(h)$.
Moreover, as $f(\psi(x)) = \varphi(f(x)) = f(y)$, $\psi(x)$ and $y$ are in the same fiber of $f$.
The connectedness of $f$ implies that there exists $\Phi\in\relInn(f)\subseteq\relInn(h)$ such that $\Phi(\psi(x)) = y$. This shows that $h$ is connected.

\eqref{item: g is conn}: Here we do not need the assumption of $g$ being surjective, since it follows from the surjectivity of $h=g\circ f$.
Let $r\in R$ and $x, y\in g^{-1}(r)$. 
As $f$ is surjective, there exist $a, b\in P$ such that $f(a) = x$ and $f(b) = y$.
Since $a, b$ are in the same fiber of $h$ and $h$ is connected, we can take $\varphi\in\relInn(h)$ such that $\varphi(a) = b$.

Then, $\psi\coloneqq f_{*}(\varphi)$ belongs to $\relInn(g)$ as
\[
g_{*}(\psi) = g_{*}\circ f_{*}(\varphi) = h_*(\varphi) = \id_{R}.
\]
Additionally, we have 
\[ \psi(x)= \psi(f(a)) = f (\varphi(a)) = f(b) = y. \]
This completes the proof.
\end{proof}

\begin{rem}
    Informally, \cref{proposition: connectedness composition and 2 out of 3}~\eqref{item: g is conn} says that connectedness descends along surjective homomorphisms: namely, given surjections $h,g$ over a quandle $R$ and a surjective homomorphism $f$ \textit{over} $R$ from $h$ to $g$, if $h$ is connected, then so is $g$.
    As a special case with $R=\{\ast\}$, we deduce that the image of a surjective homomorphism from a connected quandle is connected.
\end{rem}

\begin{cor}
    Let $f\colon P\to Q$ be a surjective quandle homomorphism. If $f$ is connected and $Q$ is connected, then $P$ is connected.
\end{cor}

\begin{proof}
The fact that $Q$ is a connected quandle is equivalent to saying that $Q\to \{\ast\}$ is connected.
So the statement follows from \eqref{item: composition is conn} of \cref{proposition: connectedness composition and 2 out of 3}.
\end{proof}

\begin{comment} %% なんかいらないなと思って消しちゃった
\begin{eg}
\begin{enumerate}
    \item Let $Q$ be a connected quandle, $P$ be a non-connected subquandle of $Q$, and $R=\{*\}$ the quandle of one point. 
    Let $\iota\colon P\to Q$ be the inclusion and $f \colon Q\to R$ the quandle homomorphism.
    Then, $f$ is connected as $Q$ is connected. However, the composition $f\circ \iota$ is not connected.
    Indeed, because $\relInn(f\circ\iota) = \Inn(P)$ and $P$ is not connected by the assumption, we have the claim.
    \item Consider the morphisms
    \begin{align*}
        \{*\} \stackrel{f}{\longrightarrow} \{a, b\} \stackrel{g}{\longrightarrow} \{*\}.
    \end{align*}
    Then, $g\circ f$ is connected, and $g$ is not connected.
\end{enumerate}
\end{eg}
\end{comment}

Finally, we characterize connected homomorphisms in terms of quandle quotients. 
%More precisely, we show that a surjective homomorphism is connected if and only if it arises as the quotient by a quandle congruence induced by a normal subgroup of the inner automorphism group.
Recall from \cref{proposition:quandle_quotient_by_group,remark_of_proposition:quandle_quotient_by_group} that we can take a quotient of a quandle $Q$ by a normal subgroup $N \subseteq \Inn(Q)$, summarizing as:

\begin{prop}
\label{prop:basic_property_of_quotient_by_normal_subgroup}
Let $Q$ be a quandle and let $N\subseteq \Inn(Q)$ be a normal subgroup.
\begin{enumerate}
    \item\label{item:N_is_contained_in_relInn_of_projection}
    For the quotient map $\pi_N\colon Q\twoheadrightarrow Q/N$, we have $N \subseteq \relInn(\pi_N)=\Ker(\Inn(\pi_N))$.
    \item\label{item:factoring_condition_through_quotient_by_normal_subgroup}
    A quandle homomorphism $f\colon Q\to R$ factors through $Q/N$ as in the diagram below if and only if $f\circ \varphi = f$ holds for all $\varphi\in N$, or equivalently if and only if $N\subseteq \relInn(f)$.
    \[\begin{tikzcd}[column sep=small]
        Q \arrow{rr}{f} \arrow[two heads]{rd}[swap]{\pi_N} & & R \\
        & Q/N \arrow[dashed]{ru}[swap]{\overline{f}} &
    \end{tikzcd}\]
\end{enumerate}
\end{prop}

\begin{proof}
    \eqref{item:N_is_contained_in_relInn_of_projection}: For $\varphi\in N$, let $\psi=(\pi_N)_*(\varphi)$. Then for any $[x]\in Q/N$, we have $\psi([x])=[\varphi(x)]$. Since $\varphi\in N$ implies $x \sim_N \varphi(x)$, it follows that $\psi([x])=[x]$, and hence $\psi=\id$. Therefore $N \subseteq \Ker((\pi_N)_*)$.
    \eqref{item:factoring_condition_through_quotient_by_normal_subgroup}: The fact that $f$ factors through $Q/N$ is equivalent to the condition that $f$ equalizes the quandle congruence $\sim_N$. This is equivalent to $f(N\cdot x)=\{f(x)\}$ for all $x\in Q$, which in turn is equivalent to $f(\varphi(x)) = f(x)$ for all $x\in Q$ and $\varphi\in N$. The other equivalent condition follows from \cref{proposition: relinn preserves fibers}.
\end{proof}

Then we can prove that connected homomorphisms are precisely those arising as quotients by normal subgroups of $\Inn(Q)$.

\begin{prop}
\label{proposition:connected_iff_quotient_by_normal_subgroup_of_Inn}
    A surjective quandle homomorphism $f\colon P \to Q$ is connected if and only if there is some normal subgroup $N \subseteq \Inn(P)$ such that $f$ is isomorphic to $\pi_N\colon P \to P/N$ (as quotients of $P$).
\end{prop}

\begin{proof}
    ($\Rightarrow$): Any homomorphism $f\colon P\to Q$ factors through the quotient $P/{\relInn(f)}$ by \cref{prop:basic_property_of_quotient_by_normal_subgroup}~\eqref{item:factoring_condition_through_quotient_by_normal_subgroup}.
    If $f$ is connected, the induced homomorphism $\overline{f}\colon P/{\relInn(f)} \to Q$ is not only surjective, but also injective by the connectedness. Thus it is an isomorphism, and $f$ is isomorphic to $\pi_{\relInn(f)}$.
    
    ($\Leftarrow$): We prove that the projection $\pi_N\colon P \to P/N$ is connected. Take any $x,y \in P$ with $\pi_N(x)=\pi_N(y)$, which means that $x$ and $y$ are in the same $N$-orbit. Then there is a $g \in N\subseteq \Inn(P)$ such that $g(x)=y$. By \cref{prop:basic_property_of_quotient_by_normal_subgroup}~\eqref{item:N_is_contained_in_relInn_of_projection} we have $g\in \relInn(\pi_N)$, and hence $\pi_N$ is connected.
\end{proof}

\begin{rem}
\label{rem:after:proposition:connected_iff_quotient_by_normal_subgroup_of_Inn}
    From \cref{proposition:connected_iff_quotient_by_normal_subgroup_of_Inn}, we see that the class of connected homomorphisms is equal to the class $\mathcal{E}_1$ introduced in {\cite[Definition 1.5]{even_gran_2014_on_factorization_systems_for_surjective_quandle_homomorphisms}}.
\end{rem}

For a normal subgroup $N\trianglelefteq\Inn(P)$, the group $\Inn(P/N)$ need not be isomorphic to $\Inn(P)/N$ in general. However, when $N=\relInn(f)$, no additional inner automorphisms become trivial in the quotient (see \cref{example: Inner of P/relInn is InnP/relInn}).

\subsection{Rigidity and Factorization}
\label{subsection: rigidity quotients of quandles and factorization}

Complementary to connected homomorphisms, \cite{bunch_lofgren_rapp_yetter_2010_on_quotients_of_quandles} and \cite{even_gran_2014_on_factorization_systems_for_surjective_quandle_homomorphisms} consider the following notion of a rigid homomorphism.

\begin{dfn}[{\cite[Definition 4.1]{bunch_lofgren_rapp_yetter_2010_on_quotients_of_quandles}}]
\label{definition: rigidity of morphism}
    A surjective quandle homomorphism $f \colon P\to Q$ is called \emph{rigid} if the induced group homomorphism $f_*\colon \Inn(P)\to\Inn(Q)$ is an isomorphism, or equivalently if
    $\relInn(f)$ is trivial.
\end{dfn}

\begin{lem}
\label{lemma: rigid is covering}
    Any rigid homomorphism is covering.
\end{lem}

\begin{proof}
    Let $f\colon P\to Q$ be a rigid homomorphism. 
    For $x,y\in P$ satisfying $f(x) = f(y)$, $f_*(s_xs_y^{-1}) = s_{f(x)}s_{f(y)}^{-1} =\id_Q$.
    Thus, the assumption implies that $s_x = s_y$, which completes the proof.
    Alternatively, it follows from the fact that $\relTrans(f)\subset \relInn(f) = \{1\}$ (see \cref{section: relative transvection group} for the definition of the former group) and \cref{lemma: relTrans vanish iff covering}.
    %If $f$ is a rigid homomorphism, then we have $\relTrans(f)\subset \relInn(f) = \{1\}$. Thus, the statement follows from \cref{lemma: relTrans vanish iff covering}.
    %\memo{future citation}
\end{proof}

\begin{lem}
    Any connected and rigid homomorphism is an isomorphism.
\end{lem}

\begin{proof}
    Suppose that a homomorphism $f\colon P \to Q$ is connected and rigid. Then it is surjective and, for each $q\in Q$, the group $\relInn(f)$ acts on $f^{-1}(q)$ transitively. Since $\relInn(f)=\{\id_P\}$, the set $f^{-1}(q)$ must be the one-point set, which implies that $f$ is injective.
\end{proof}

\begin{prop}
\label{proposition: properties of rigid hom}
    Let $f\colon P\to Q$, $g\colon Q\to R$ be surjective quandle homomorphisms and put $h\coloneqq g\circ f$. Then, the following hold.
    \begin{enumerate}
        \item\label{item:rigid_satisfy_two-out-of-three} If two out of $f$, $g$, and $h$ are rigid, then the other is. In particular, rigid morphisms are closed under composition.
        \item\label{item:rigid_is_closed_under_surjections} Rigid homomorphisms are closed under pullback along surjective homomorphisms. That is, for any surjective quandle homomorphism $k\colon Q'\to Q$, if $f$ is rigid, then the pull-back $P\times_{Q}Q' \to Q'$ of $f$ is also rigid.
        \item\label{item:rigid_is_closed_under_products} If $f$ and $f'\colon P' \to Q'$ are rigid, then the product $f\times f'\colon P\times P' \to Q\times Q'$ is also rigid.
    \end{enumerate}
\end{prop}
\begin{proof}
    \eqref{item:rigid_satisfy_two-out-of-three}: By the functoriality of $\Inn(-)$, we have $h_*= g_* \circ f_*$. The statement follows from the fact that if two out of $f_*$, $g_*$, and $h_*$ are isomorphisms, then the other is.

    \eqref{item:rigid_is_closed_under_surjections}: We write $\Tilde{f}\colon R\coloneqq P\times_{Q}Q'\to Q'$ for the pull-back and $p\colon R\to P$ for the first projection. 
    Since $p$ and $\Tilde{f}$ is surjective by the surjectivity of $k$ and $f$, the functoriality of $\Inn(-)$ shows that $f_*\circ p_* = k_*\circ\Tilde{f}_*$.
    Let $\alpha\in\relInn(\Tilde{f})$. Then,
    \[
        f_*\circ p_*(\alpha) = k_*\circ\Tilde{f}_*(\alpha)=\id_Q.
    \]
    The rigidity of $f$ implies that $p_*(\alpha) = \id_P$. 
    Thus, we have $p\circ\alpha = p_*(\alpha)\circ p =p$. 
    Also, $\Tilde{f}\circ\alpha = \Tilde{f}_*(\alpha)\circ\Tilde{f} = \Tilde{f}$ holds by the assumption. 
    Thus, by the universality of pullback, we conclude $\alpha=\id_R$, which completes the proof.

    \eqref{item:rigid_is_closed_under_products}: Put $F\coloneqq f\times f'$, and let $p_i$ be the $i$-th projection from $P\times P'$ for $i=1,2$ and $q_j$ be the $j$-th projection from $Q\times Q'$ for $j=1,2$. 
    Let $\beta\in\relInn(F)$. 
    As $f_*((p_1)_*(\beta)) = (q_1)_*(F_*(\beta)) = \id_Q$ and $f$ is rigid, it follows that $(p_1)_*(\beta)=\id_P$. 
    Similarly, we have $(p_2)_*(\beta)=\id_{P'}$. 
    These observations mean $p_i\circ \beta=p_i$ for $i=1,2$. Thus, by the universality of product, we obtain $\beta=\id_{P\times P'}$. Thus $F=f\times f'$ is rigid.
\end{proof}

From the viewpoint of morphism properties, we can restate \cite[Proposition 3.1]{even_gran_2014_on_factorization_systems_for_surjective_quandle_homomorphisms}, or essentially \cite[Theorem 8.1]{bunch_lofgren_rapp_yetter_2010_on_quotients_of_quandles}, as follows:

\begin{thm}[Quandle Stein Factorization {\cite[Theorem 8.1]{bunch_lofgren_rapp_yetter_2010_on_quotients_of_quandles}}]
% also cited in {\cite[Proposition 3.1]{even_gran_2014_on_factorization_systems_for_surjective_quandle_homomorphisms}}
\label{theorem: quandle stein factorization}
    Let $f \colon P\to Q$ be a surjective quandle homomorphism.
    Then, $f$ has a factorization as $f=h\circ g$, where $g$ is connected and $h$ is rigid; especially, $h$ is covering.
\end{thm}

\begin{proof}
    Taking $N=\relInn(f)$, we obtain a factorization of $f$ as
    \[\begin{tikzcd}[column sep=small]
        P \arrow{rr}{f} \arrow{rd}[swap]{g} & & Q\rlap{.} \\
        & P/N \arrow{ru}[swap]{h} &
    \end{tikzcd}\]
    By \cref{proposition:connected_iff_quotient_by_normal_subgroup_of_Inn}, the projection $g$ is connected.
    It is clear from the construction that for every $\varphi \in N$, $g\circ \varphi=g$: that is, we have $N \subseteq \Ker(g_*)$.
    Since $h_*\circ g_* = f_*$ by functoriality, we also have $\Ker(g_*) \subseteq \Ker(f_*)=N$. Thus $N=\Ker(g_*)$, and from this we conclude that $h_*=\Inn(h)$ is an isomorphism.
\end{proof}

\begin{cor}
\label{example: Inner of P/relInn is InnP/relInn}
    Let $f\colon P\to Q$ be a surjective quandle homomorphism.
    Then, we have the following isomorphism:
    \[
        \Inn(P/\relInn(f)) \cong \Inn(P)/\relInn(f).
    \]
\end{cor}
\begin{comment}Proof:
    As the relative inner automorphism group is a normal subgroup of $\Inn(P)$, the surjection $\pi\colon P \to P/\relInn(f)$ is well-defined.
    Thus, we have the induced group surjective homomorphism $\pi_*\colon \Inn(P) \twoheadrightarrow \Inn(P/\relInn(f))$.
    We need to show that the kernel $\Ker(\pi_*)$ coincides with the relative inner automorphism group $\relInn(f)$.
    The inclusion $\relInn(f)\subset \Ker(\pi_*)$ is easy because every $\varphi\in\relInn(f)$ preserves $N$-orbits. 
    For the converse, let $g\in\Ker(\pi_*)$. 
    For any $p\in P$, there exists a relative inner automorphism $\varphi_p\in \relInn(f)$ such that $g(p) = \varphi_p(p)$ since $[g(p)] = [p]$ in $P/\relInn(f)$. 
    On the other hand, we have
    \[
        f(p) = f(\varphi_p(p)) = f\circ g(p) = f_*(g)(f(p)).
    \]
    Therefore, we have the equality $f_*(g) = \id_Q$ by the surjectivity of $f$.
\end{comment}

In fact, the connected--rigid factorization enjoys the stronger property of forming an orthogonal factorization system, which was observed by Even and Gran~\cite{even_gran_2014_on_factorization_systems_for_surjective_quandle_homomorphisms}. As noted in \cref{rem:after:proposition:connected_iff_quotient_by_normal_subgroup_of_Inn}, the class of connected homomorphisms coincides with the class $\mathcal{E}_1$ in the sense of \cite[Definition 1.5]{even_gran_2014_on_factorization_systems_for_surjective_quandle_homomorphisms}.

\begin{prop}[{\cite[Proposition 3.2]{even_gran_2014_on_factorization_systems_for_surjective_quandle_homomorphisms}}]
\label{prop:connected-rigid_forms_OFS_for_surjection}
    The pair $(\{\mathrm{connected}\}, \{\mathrm{rigid}\})$ of classes of surjections is an orthogonal factorization system for the surjections in $\Quandle$.
\end{prop}

\begin{rem}
    By \cref{cor:connected_is_stable_under_rigid_base_change} below, this factorization system in fact satisfies the \textit{Frobenius condition} (see \cite[\href{https://ncatlab.org/nlab/show/Frobenius+reciprocity\#InWeakFactorizationSystems}{Frobenius reciprocity}]{nLab}).
\end{rem}

\begin{cor}
    Connected homomorphisms are closed under pushout along surjective homomorphisms.
\end{cor}

\begin{proof}
    This follows formally from \cref{prop:connected-rigid_forms_OFS_for_surjection}.
\end{proof}

\subsection{Homogeneous Homomorphisms}
\label{subsection:relative_homogeneous}

Similarly to connectedness, we introduce the relative notion of homogeneity, using relative automorphism groups (\cref{definition:relative_automorphism_group}).

\begin{dfn}
\label{definition: homogeneity for morphism}
    Let $f \colon P\to Q$ be a quandle homomorphism.
    We say that $f$ is \emph{homogeneous} if it is surjective and for each $q\in Q$, the action of $\relAut(f)$ on the fiber $f^{-1}(q)$ is transitive.
\end{dfn}

Since $\relInn(f)\subseteq \relAut(f)$, if $f$ connected, then it is homogeneous.

\begin{eg}
    Consider the unique homomorphism $f\colon P\twoheadrightarrow\{\ast\}$ into the terminal quandle. Since $\Aut(\{\ast\})=\{\id_{\{\ast\}}\}$ and hence $\relAut(f) = \Aut(P)$, $f$ is homogeneous if and only if $P$ is a homogeneous quandle.
\end{eg}

Analogously to \cref{proposition:connected_iff_quotient_by_normal_subgroup_of_Inn}, we can characterize homogeneous homomorphisms in view of group quotients.

\begin{lem}
\label{lemma:relAut_is_normalized_by_Inn}
    For a surjective homomorphism $f\colon P \twoheadrightarrow Q$, the relative automorphism group $\relAut(f)\subseteq \Aut(P)$ is normalized by $\Inn(P)$
\end{lem}

\begin{proof}
    For any $x\in P$ and $\varphi\in \relAut(f)$, we have
    \begin{align*}
        fs_x\varphi s_x^{-1} &= s_{f(x)} f \varphi s_x^{-1} = s_{f(x)} f s_x^{-1} = fs_xs_x^{-1}= f, \text{ and}\\
        fs_x^{-1}\varphi s_x &= s_{f(x)}^{-1} f \varphi s_x = s_{f(x)}^{-1} f s_x = fs_x^{-1}s_x= f.
    \end{align*}
    Hence, $s_x\varphi s_x^{-1}$ and $s_x^{-1}\varphi s_x$ is in $\relAut(f)$, which shows that $\relAut(f)$ is normalized by $\Inn(P)$.
\end{proof}

\begin{prop}
\label{proposition:homogeneous_iff_quotient_by_subgroup_normalized_by_Inn}
    A surjective quandle homomorphism $f\colon P \to Q$ is homogeneous if and only if there is some subgroup $G \subseteq \Aut(P)$ normalized by $\Inn(P)$ such that $f$ is isomorphic to $\pi_G\colon P \to P/G$ (as quotients of $P$).
\end{prop}

\begin{proof}
    ($\Rightarrow$): By \cref{proposition:quandle_quotient_by_group} and \cref{lemma:relAut_is_normalized_by_Inn}, any homomorphism $f\colon P\to Q$ factors through the quotient $P/{\relAut(f)}$.
    If $f$ is homogeneous, the induced homomorphism $\overline{f}\colon P/{\relAut(f)} \to Q$ is not only surjective, but also injective by the homogeneity. Thus it is an isomorphism, and $f$ is isomorphic to $\pi_{\relAut(f)}$.
    
    ($\Leftarrow$): We prove that the quotient map $\pi_G\colon P \to P/G$ is homogeneous. 
    Note that for all $g\in G$, we have $\pi_G \circ g= \pi_G$, and so $G\subseteq \relAut(\pi_G)$.
    Take any $x,y \in P$ with $\pi_G(x)=\pi_G(y)$, which means that $x$ and $y$ are in the same $G$-orbit. Then there is $g \in G$ such that $g(x)=y$. Since $G \subseteq \relAut(\pi_G)$, this shows that $\pi_G$ is homogeneous.
\end{proof}

\begin{eg}
\label{example: quandle_triplet_quandle_is_homogeneous_quotient_of_GAlex}
    Let $(G,H,\sigma)$ be a quandle triplet (see \cref{example: quandle triplet}), and consider the associated quandle $Q(G,H,\sigma)$. It is easy to verify that the canonical surjection
    \[ f\colon \GAlex(G,\sigma) \twoheadrightarrow Q(G,H,\sigma),\quad x \mapsto [x] \]
    is a quandle homomorphism.
    For each element $h \in H$, let $R_h\colon G \to G$ denote the right multiplication map $x\mapsto xh$ of $h$. Then it becomes a quandle automorphism $R_h\colon \GAlex(G,\sigma)\to \GAlex(G,\sigma)$; indeed, since $H\subseteq \Fix(\sigma)$, we calculate
    \begin{align*}
        s_{R_h(x)}(R_h(y)) &= s_{xh}(yh)= xh \cdot\sigma((xh)^{-1}yh) = xh\cdot h^{-1}\sigma(x^{-1}y)h \\
        &= x\cdot \sigma(x^{-1}y)h=R_h(s_x(y)) 
    \end{align*}
    for all $x,y \in G$. Moreover, for all $x\in G$,
    \[ f\circ R_h(x) = f(xh) = [xh] = [x] = f(x), \]
    so $R_h \in \relAut(f)$.

    Let $x,y\in G$ such that $f(x)=f(y)$. By the definition of $Q(G,H,\sigma)$, there is an element $h \in H$ such that $x=yh = R_h(y)$. Since $R_h \in \relAut(f)$, this shows that $\relAut(f)$ acts transitively on the fibers of $f$, and hence $f$ is homogeneous.

    Since $R_{hh'}=R_{h'}\circ R_h$, we have an injective group homomorphism $H^\op \hookrightarrow \Aut(\GAlex(G,\sigma))$ from the opposite group of $H$. Through it, we regard $H^\op$ as a subgroup of $\Aut(\GAlex(G,\sigma))$. 
    In view of \cref{proposition:homogeneous_iff_quotient_by_subgroup_normalized_by_Inn}, we can identify $Q(G,H,\sigma)$ with the quotient by the subgroup $H^\op\subseteq \Aut(\GAlex(G,\sigma))$. For this, we first show that $H^\op$ is normalized by $\Inn(\GAlex(G,\sigma))$; indeed, since $s_x=s_{xh}$ for $h \in H$, we have
    \[ s_x R_h s_x^{-1} = s_xs_{xh}^{-1} R_h = R_h \text{ and } s_x^{-1}R_h s_x = s_x^{-1}s_{xh} R_h=R_h \]
    for all $x\in G$. Consequently, we can take the quotient by $H^\op\subseteq \Aut(\GAlex(G,\sigma))$, and the canonical surjection $\GAlex(G,\sigma)/H^\op \to Q(G,H,\sigma)$ is readily seen to be an isomorphism.
\end{eg}

It is known that every homogeneous quandle is isomorphic to the quandle associated with a quandle triplet (see \cite[Theorem 7.1]{joyce_1982_a_classifying_invariant_of_knots_the_knot_quandle}).
The above example shows that every homogeneous quandle is a homogeneous quotient of a generalized Alexander quandle.
Note that, as the following example shows, not every homogeneous quotient of a generalized Alexander quandle arises from a quandle triplet.

\begin{eg}
\label{example: homogeneous_quotient_of_GAlex_which_does_not_arise_from_quandle_triplet}
    Consider the linear Alexander quandle $P=\Lambda_{4,-1}=\Alex(\ZZ/4\ZZ, -\id)$ of order $4$, or the dihedral quandle $R_4$. The symmetry map $s_x$ is given by $s_x(y)=2x-y$.
    Let $K=\langle \sigma \rangle= \{\id,-\id\} \subseteq \Aut(P)$, where $\sigma=-\id$ is regarded as an automorphism of $\Lambda_{4,-1}$.

    For every $x,y\in \ZZ/4\ZZ$, we have
    \[ s_x\sigma s_x^{-1}(y) = s_x(-(2x-y)) = 4x-y = -y = \sigma(y), \]
    and similarly $s_x^{-1}\sigma s_x=\sigma$.
    Hence $K$ is normalized by $\Inn(P)$.
    By \cref{proposition:homogeneous_iff_quotient_by_subgroup_normalized_by_Inn}, the canonical quotient map
    \[ f\colon P\twoheadrightarrow Q=P/K \]
    is homogeneous.
    Note that the quotient $Q=P/K$ consists of three equivalence classes: $[0]=\{0\}$, $[1]=\{1,3\}$, and $[2]=\{2\}$.

    However, $Q$ is not isomorphic to a quandle associated with any quandle triplet.
    Indeed, if $(G,H,\sigma)$ is a quandle triplet, then every fiber of the canonical projection
    \[ \GAlex(G,\sigma)\twoheadrightarrow Q(G,H,\sigma) \]
    has the same cardinality $\lvert H\rvert$, since the fibers are precisely the right cosets of $H$.
    On the other hand, the fibers of $f$ do not all have the same cardinality: $\vert f^{-1}([0])\rvert = \lvert f^{-1}([2])\rvert=1$ and $\lvert f^{-1}([1])\rvert =2$.
    Therefore, $Q$ cannot arise from a quandle triplet.

    We can show directly that $Q$ is not homogeneous.
    The quandle structure on $Q$ is $s_{[0]}=\id=s_{[2]}$ and $s_{[1]}=(02)$.
    If $\psi\in \Aut(Q)$ is a quandle automorphism, then $s_{\psi([0])}= \psi\circ s_{[0]}\circ \psi^{-1}=\id$ forces $\psi([0])\in\{[0],[2]\}$. Similarly, $\psi([2])\in\{[0],[2]\}$, so necessarily $\psi([1])=[1]$. Since $\psi$ is bijective, it follows that $\psi$ is either the identity or the transposition $(02)$. Therefore, we have
    \[ \Aut(Q)=\Inn(Q)=\{\id,(02)\}. \]
    Since $\Aut(Q)$ does not act transitively on $Q$, the quandle $Q$ is not homogeneous.
\end{eg}

\begin{rem}
    This example shows that, unlike connectedness,homogeneity does not satisfy the right cancellation property. Homogeneous homomorphisms are not even closed under composition; see \cref{example: composition of homogeneous is not homogeneous}.
    On the other hand, they are stable under products and pullbacks.
\end{rem}

\begin{prop}
\label{proposition:product_of_homogeneous_homs_is_homogeneous}
    Let $f \colon P\to Q$ and $f'\colon P' \to Q'$ be homogeneous surjective quandle homomorphisms.
    Then the product $f\times f'\colon P\times P'\to Q\times Q'$ is also homogeneous.
\end{prop}
\begin{proof}
    First, as $f$ is surjective, $P$ is an empty quandle if and only if $Q$ is an empty quandle. In this case, the product $f\times f'$ is the identity $\id_{\empty}$ of the empty quandle, which is clearly homogeneous.

    Thus, we may assume that all quandles are non-empty.
    Let $(x,x'), (y,y')\in P\times P'$ satisfy $(f(x),f'(x'))= (f(y),f'(y'))$. As $f(x)=f(y)$, the homogeneity of $f$ implies that there is $k \in \relAut(f)$ such that $k(x)=y$. Similarly, there is $k'\in \relAut(f')$ with $k'(x')=y'$. Then, we have $k\times k' \in \relAut(f\times f')$, and $(k\times k')(x,x') = (k(x),k'(x')) = (y,y')$. This completes the proof.
\end{proof}

\begin{prop}
\label{proposition: homogeneousness_is_closed_under_any_pullback}
    Let $f\colon P\twoheadrightarrow Q$ be surjective quandle homomorphism and $g\colon Q' \to Q$ be an arbitrary homomorphism.
    If $f$ is homogeneous, then the pull-back of $f$ along $g$ is also homogeneous.
\end{prop}
\begin{proof}
    Let
    \[ \begin{tikzcd}
        P' \arrow{r}{g'} \arrow{d}[swap]{f'} & P \arrow{d}{f} \\
        Q' \arrow{r}[swap]{g} & Q \arrow[phantom]{lu}[very near end]{\lrcorner}
    \end{tikzcd} \]
    be a pullback diagram of $f$ along $g$. Remark that $f'$ is surjective since $f$ is so. 
    Take $x,y \in P'$ such that $f'(x)=f'(y)$. Since $f(g'(x))=g(f'(x))=g(f'(y))= f(g'(y))$ and $f$ is homogeneous, there is $k \in \relAut(f)$ such that $k(g'(x))=g'(y)$.
    From the universality of pullback, we get a homomorphism $k'\colon P' \to P'$ satisfying
    \[\begin{tikzcd}
        P' \arrow{r}{g'} \arrow[dashed]{d}[swap]{k'} \arrow[bend right=40]{dd}[swap]{f'} & P \arrow{d}{k} \arrow[bend left=40]{dd}{f} \\
        P' \arrow{r}{g'} \arrow{d}[swap]{f'} & P \arrow{d}{f} \arrow[phantom]{lu}[very near end]{\lrcorner} \\
        Q' \arrow{r}[swap]{g} & Q \arrow[phantom]{lu}[very near end]{\lrcorner}
    \end{tikzcd}
    \quad = \quad
    \begin{tikzcd}
        P' \arrow{r}{g'} \arrow{d}[swap]{f'} & P \arrow{d}{f} \\
        Q' \arrow{r}[swap]{g} & Q\rlap{.} \arrow[phantom]{lu}[very near end]{\lrcorner}
    \end{tikzcd}\]
    Since $k$ is an isomorphism, so is $k'$, and hence $k' \in \relAut(f')$.
    As $f'(k'(x)) = f'(x)=f'(y)$ and $g'(k'(x)) = k(g'(x)) = g'(y)$, we conclude $k'(x)=y$, which completes the proof.
\end{proof}

\begin{cor}
\label{cor:fibers_of_homogeneous_hom_are_homogeneous}
    For a homogeneous homomorphism $f\colon P \twoheadrightarrow Q$, every fiber $f^{-1}(q)$ is homogeneous.
\end{cor}

The converse of \cref{cor:fibers_of_homogeneous_hom_are_homogeneous} does not hold in general; see \cref{example: fiber-wise homogeneous but not homogeneous}.

\begin{prop}
\label{proposition: homogeneous-rigid_factorization}
    Let $f\colon P\twoheadrightarrow Q$ be a surjective quandle homomorphism and $G=\relAut(f)$. Then, $f$ factors as 
    \[\begin{tikzcd}[column sep=small]
        P \arrow[two heads]{rr}{f} \arrow[two heads]{rd}[swap]{\pi_G} & & Q\rlap{,} \\
        & P/G \arrow[two heads]{ru}[swap]{h_G}
    \end{tikzcd}\]
    where $\pi_G$ is homogeneous and $h_G$ is rigid.
\end{prop}
\begin{proof}
    By \cref{lemma:relAut_is_normalized_by_Inn}, $G$ is normalized by $\Inn(P)$.
    Therefore, we can take the quotient $P/G$ and the homomorphism $f$ factors through it.
    The quotient map $\pi_G$ is homogeneous by \cref{proposition:homogeneous_iff_quotient_by_subgroup_normalized_by_Inn}.
    It remains to show that the induced homomorphism $h_G\colon P/G \to Q$ is rigid.
    Let $\varphi\in \Inn(P/G)$ with $(h_G)_*(\varphi)=\id$. By the surjectivity of $(\pi_G)_*$, we can take $\psi\in \Inn(P)$ such that $(\pi_G)_*(\psi)=\varphi$. Then we have $\psi\in \relInn(f)\subseteq \relAut(f)=G$, and so $\varphi=(\pi_G)_*(\psi)=\id$. Thus, it follows that $\relInn(h_G)=\{\id\}$, and $h_G$ is rigid.
\end{proof}

\begin{prop}
    The homogeneous--rigid factorization of a surjective homomorphism $f$ in \cref{proposition: homogeneous-rigid_factorization} is universal among all factorizations through a homogeneous homomorphism:
    that is, for any factorization $f=h\circ \pi$ with $\pi\colon P \to R$ homogeneous, there exists a unique quandle homomorphism $u\colon R \to P/G$, where $G=\relAut(f)$, such that $\pi_G = u\circ \pi$ and $h= h_G \circ u$.
    \[\begin{tikzcd}
        P \arrow[rr,"f" description,rounded corners,
		      to path={(\tikztostart)
			-- ([yshift=2ex]\tikztostart.north)
			-- ([yshift=2ex]\tikztotarget.north)\tikztonodes
			-- (\tikztotarget)}]
        \arrow{r}{\pi_G} \arrow{rd}[swap]{\pi} & P/G \arrow{r}{h_G} & Q \\
        & R \arrow{ru}[swap]{h} \arrow[dashed]{u}{u} & 
    \end{tikzcd}\]
\end{prop}

\begin{proof}
    Suppose that $f= h\circ \pi$, where $\pi\colon P\to R$ is a homogeneous homomorphism and $h\colon R\to Q$ is a homomorphism.
    By \cref{proposition:homogeneous_iff_quotient_by_subgroup_normalized_by_Inn}, there is a subgroup $K\subseteq \Aut(P)$ normalized by $\Inn(P)$ such that $R\cong P/K$ as quotients of $P$.
    Then $K\subseteq \relAut(\pi) \subseteq \relAut(f)=G$, and thus $\pi_G$ factors uniquely through $\pi$ as $\pi_G=u\circ \pi$. 
    By observing $h_G\circ u\circ \pi = h_G\circ \pi_G = f = h\circ \pi$, we see $h= h_G \circ u$ by the surjectivity of $\pi$.
\end{proof}

\section{Relative Transvection Group}
\label{section: relative transvection group}

We also introduce a relativization of transvection groups of quandles.

\subsection{Definition}
\label{subsection: definition of relative transvection group}

\begin{dfn}
\label{definition: relative transvection group}
    Let $f\colon P \to Q$ be a homomorphism of quandles.
    We define the \emph{relative transvection group} of $f$, denoted by $\relTrans(f)$, as 
    \[
    \relTrans(f) \coloneqq \langle \{ s_xs_y^{-1} \mid x,y\in P \text{ such that }f(x) = f(y) \}\rangle \subseteq \Inn(P). 
    \]
\end{dfn}

\begin{rem}
    The group $\relTrans(f)$ is also considered in \cite{bonatto_stanovsky_2021_commutator_theory_for_racks_and_quandles} as the relative displacement group $\Dis_\alpha$ where $\alpha$ is the congruence corresponding to the surjection $f$.
\end{rem}

\begin{eg}
    If $f$ is the unique homomorphism $P\to \{\ast\}$ into the terminal quandle, then $\relTrans(f)=\Trans(P)$.

    If $f$ is an injection, then $\relTrans(f) = \relInn(f) = \{\id\}$.
\end{eg}

\begin{prop}
\label{proposition: relTrans is normal in relInn}
    The relative transvection group $\relTrans(f)$ is a normal subgroup of $\Inn(P)$ and of $\relInn(f)$.
\end{prop}

\begin{proof}
    Let $x, y\in P$ such that $f(x) = f(y)$.
    Then, 
    \[ f_*(s_xs_y^{-1}) = f_*(s_x)f_*(s_y)^{-1} = s_{f(x)}s_{f(y)}^{-1} = \id_{\Image(f)}. \]
    Therefore, we have $\relTrans(f) \subseteq \relInn(f)$.

    For any $\varphi\in\Inn(P)$, as $f(\varphi(x)) = f_*(\varphi)(f(x)) = f_*(\varphi)(f(y)) = f(\varphi(y))$, 
    we have
    \[ \varphi(s_xs_y^{-1})\varphi^{-1} 
        = (\varphi s_x \varphi^{-1})(\varphi s_y^{-1}\varphi^{-1}) = s_{\varphi(x)}s_{\varphi(y)}^{-1}\in\relTrans(f). \]
    Thus, $\relTrans(f)\unlhd \Inn(P)$; in particular, $\relTrans(f)\unlhd\relInn(f)$.
\end{proof}

\begin{eg}
\label{example:connected_but_not_strongly_connected}
    It is worth noting that the connectedness of a homomorphism does \textit{not}, in general, imply the transitivity of action of its relative transvection group on the fibers, though it is true for quandles (\cref{lemma: connectedness is also determined by trans}).
    
    For example, let $f \colon R_4\to R_2$ be the quandle homomorphism between dihedral quandles, sending $x$ to $(x \mod 2)$.
    Then, $f$ is connected (see \cref{example: relinn of morph between lin alexes}), but, as the symmetries of $R_4$ are $s_0=s_2=-\id$ and $s_1=s_3=(x\mapsto 2-x)$, the relative transvection group $\relTrans(f)$ becomes trivial, and does not act transitively on fibers, which contain two elements.
\end{eg}

\subsection{Strongly Connected Homomorphisms}
\label{subsection: strongly connected homomorphisms}

Unlike in the absolute case, transitivity of the relative inner group and transitivity of the relative transvection group are no longer equivalent in the relative setting. This leads us to the following notion.

\begin{dfn}
\label{def:strongly_connected_homomorphism}
    A surjective quandle homomorphism $f \colon P \twoheadrightarrow Q$ is said to be \emph{strongly connected} if its relative transvection group $\relTrans(f)$ acts transitively on every fiber of $f$.
\end{dfn}

Since $\relTrans(f) \subseteq \relInn(f)$, strongly connected homomorphisms are connected.
\cref{example:connected_but_not_strongly_connected} above says that $R_4\to R_2$ is connected, but not strongly connected. 

\begin{eg}
    Consider the unique homomorphism $f\colon P\twoheadrightarrow\{\ast\}$ into the terminal quandle. Since $\Inn(\{\ast\})=\{\id_{\{\ast\}}\}$, we have $\relTrans(f) = \Trans(P)$. In view of \cref{lemma: connectedness is also determined by trans}, $f$ is strongly connected if and only if $f$ is connected, and if and only if $P$ is a connected quandle.
\end{eg}

\begin{lem}
\label[lem]{lemma: relative symmetry groups under a surjective factorization}
Suppose that $f=h\circ u$, where $u\colon P\twoheadrightarrow R$ and $h\colon R\twoheadrightarrow Q$ are surjective. Then
\[ u_*(\relInn(f))=\relInn(h), \qquad u_*(\relTrans(f))=\relTrans(h). \]
\end{lem}

\begin{proof}
The first equality follows from $h_*\circ u_*=f_*$ and the surjectivity of $u_*$. For the second, every generator $s_xs_y^{-1}$ of $\relTrans(f)$ maps to a generator of $\relTrans(h)$, and every generator downstairs lifts to one upstairs because $u$ is surjective.
\end{proof}

\begin{prop}
\label[prop]{lemma: composition of strongly connected homomorphisms}
    A composite of strongly connected homomorphisms is strongly connected.
\end{prop}
\begin{proof}
Let $f\colon P\twoheadrightarrow Q$ and $g\colon Q\twoheadrightarrow R$ be strongly connected, and put $h\coloneqq g\circ f$. Let $x,y\in P$ satisfy $h(x)=h(y)$. Since $g(f(x))=g(f(y))$, we can take $\psi\in\relTrans(g)$ such that $\psi(f(x))=f(y)$. By \cref{lemma: relative symmetry groups under a surjective factorization}, there is $\varphi\in\relTrans(h)$ with $f_*(\varphi)=\psi$. Since $f(\varphi(x))=f_*(\varphi)(f(x))=\psi(f(x))=f(y)$, there is $\xi\in\relTrans(f)$ such that $\xi(\varphi(x))=y$. Since $\relTrans(f)\subseteq\relTrans(h)$, the assertion follows.
\end{proof}

\begin{prop}
\label{proposition:strongly_connected_can_be_verified_fiberwise}
    Let $\pi\colon P\twoheadrightarrow Q$ be a surjective quandle homomorphism. 
    Assume that each fiber $F_q \coloneqq \pi^{-1}(q)$, seen as a subquandle of $P$, is connected.
    Then, $\pi$ is strongly connected (and hence connected).
\end{prop}
\begin{proof}
    Fix a point $q\in Q$ and consider its fiber $F_q$.
    We claim that $\Trans(F_q) \subseteq \rho_q(\relTrans(\pi))$, since $s_x s_y^{-1} \in \relTrans(\pi)$ for any $x, y\in F_q$ and thus
    \[
    \Trans(F_q) = \langle \{\rho_q( s_xs_y^{-1}) \mid x, y\in F_q\} \rangle \subset \rho_q(\relTrans(\pi)).
    \]
    
    If $F_q$ is connected, then $\Trans(F_q)$ acts transitively on $F_q$ by \cref{lemma: connectedness is also determined by trans}.
    Therefore, $\rho_q(\relTrans(\pi))$ also acts transitively on $F_q$, which means that $\pi$ is strongly connected.
\end{proof}

\begin{lem}
\label{lem:commutator_inn_rel_trans_rel}
Let $f \colon P \twoheadrightarrow Q$ be a surjective quandle homomorphism.
\begin{enumerate}
    \item\label{item:commutator_of_Inn_and_relInn_is_in_relTrans}
    We have the inclusion $[\Inn(P), \relInn(f)] \subseteq \relTrans(f)$.
    \item\label{item:the_equality_holds_for_connected_hom}
    If $f$ is connected, then the equality $[\Inn(P), \relInn(f)] = \relTrans(f)$ holds.
\end{enumerate}
\end{lem}

\begin{proof}
%% \cref{prop:H_f_is_contained_in_the_center_of_Inn/relTrans}と同じことしてる。こっちのほうがわかりやすいかも
    \eqref{item:commutator_of_Inn_and_relInn_is_in_relTrans}: Let $\varphi \in \relInn(f)$ and let $s_x \in \Inn(P)$ be a generator. The commutator of these elements is given by
    \[
    [\varphi, s_x] = \varphi s_x \varphi^{-1} s_x^{-1} = s_{\varphi(x)} s_x^{-1}.
    \]
    Since $\varphi \in \relInn(f)$ satisfies $f\circ \varphi= f$, we have $f(\varphi(x)) = f(x)$. Hence, $[\varphi, s_x]=s_{\varphi(x)} s_x^{-1}$ belongs to $\relTrans(f)$. Because $\relTrans(f)$ is a normal subgroup of $\Inn(P)$, it contains all such commutators, establishing $[\Inn(P), \relInn(f)] \subseteq \relTrans(f)$.
    
    \eqref{item:the_equality_holds_for_connected_hom}: If $f$ is connected, $\relInn(f)$ acts transitively on each fiber. Thus, for any pair $x, y \in P$ in the same fiber, there exists some $\varphi \in \relInn(f)$ such that $\varphi(x) = y$. Consequently, any generator $s_y s_x^{-1}$ of $\relTrans(f)$ can be realized as the commutator $[\phi, s_x] \in [\relInn(f), \Inn(P)] = [\Inn(P), \relInn(f)]$, proving the equality.
\end{proof}

\begin{lem}
\label{lem:relTrans_of_quotient_map}
    Let $P$ be a quandle and $N$ be a normal subgroup of $\Inn(P)$. 
    For the quotient homomorphism $\pi_N \colon P \twoheadrightarrow P/N$, we have $\relTrans(\pi_N) = [\Inn(P), N]$.
\end{lem}

\begin{proof}
    By definition, two elements $x, y \in P$ lie in the same fiber of $\pi_N$ if and only if there exists $n \in N$ such that $y = n(x)$. Therefore, $\relTrans(\pi_N)$ is generated by elements of the form
    \[
    s_y s_x^{-1} = s_{n(x)} s_x^{-1} = n s_x n^{-1} s_x^{-1} = [n, s_x].
    \]
    Since $\Inn(P)$ is generated by the elements $s_x$, the normal subgroup generated by these commutators is exactly $[\Inn(P), N]$. Hence, we obtain $\relTrans(\pi_N) = [\Inn(P), N]$.
\end{proof}

Among connected homomorphisms, strongly connected homomorphisms correspond exactly to perfect normal subgroups of relative inner automorphism groups.
We say that a normal subgroup $N$ of a group $G$ is \emph{$G$-perfect} if $[G,N]=N$.

\begin{thm}
\label{thm:characterize_strongly_connected}
    A surjective quandle homomorphism $f\colon P\twoheadrightarrow Q$ is strongly connected if and only if it is isomorphic to the quotient $\pi_N$ by an $\Inn(P)$-perfect subgroup $N\subseteq \Inn(P)$.
\end{thm}

\begin{proof}
    ($\Rightarrow$): When $f$ is strongly connected, it is isomorphic to the quotient $\pi_N$, where $N=\relTrans(f)$.
    By \cref{lem:relTrans_of_quotient_map}, we have $\relTrans(\pi_N) = [\Inn(P), N]$.
    Therefore, $N=\relTrans(f)=\relTrans(\pi_N)=[\Inn(P),N]$, which shows that $N=\relTrans(f)$ is $\Inn(P)$-perfect.

    ($\Leftarrow$): Suppose that $N$ is $\Inn(P)$-perfect, i.e., $N = [\Inn(P), N]$. Then, by \cref{lem:relTrans_of_quotient_map}, it follows that $\relTrans(\pi_N) = N$, which trivially acts transitively on the orbits of $\pi_N$. Therefore $\pi_N$ is strongly connected.
\end{proof}

\begin{cor}
\label{cor:strongly_connected_hom_are_in_bijection_to_perfect_subgroup}
    For any quandle $P$, there is a canonical one-to-one correspondence between:
    \begin{enumerate}[label={(\roman*)}]
        \item Isomorphism classes of strongly connected quotients of $P$.
        \item $\Inn(P)$-perfect normal subgroups of $\Inn(P)$.
    \end{enumerate}
    The correspondence is given by mapping a surjection $f$ to its relative transvection group $\relTrans(f)$, and mapping a perfect normal subgroup $N$ to the quotient map $\pi_N \colon P \twoheadrightarrow P/N$.
\end{cor}

\begin{proof}
    By definition, for a strongly connected homomorphism $f$, we have $f\cong \pi_{\relTrans(f)}$. Conversely, for an $\Inn(P)$-perfect normal subgroup $N\subseteq \Inn(P)$, it follows from \cref{lem:relTrans_of_quotient_map} that $\relTrans(\pi_N)=[\Inn(P),N]=N$.
\end{proof}

\subsection{Covering Factor of a Surjection}
\label{subsection: covering factor of a surjection}

We observe that the vanishing of relative transvection groups characterizes covering homomorphisms.

\begin{lem}
\label{lemma: relTrans vanish iff covering}
    A surjective quandle homomorphism $\pi\colon P\to Q$ is covering if and only if $\relTrans(\pi)=1$.
\end{lem}
\begin{proof}
    It follows immediately from the definition of covering homomorphisms.
\end{proof}

\begin{rem}
    Let $f\colon P\to Q$ be a surjective quandle homomorphism. Then $\Inn(f)$ induces a group homomorphism $\Trans(P) \twoheadrightarrow \Trans(Q)$, for which we here write $\Trans(f)$ temporarily. It is easy to see that $\relTrans(f) \subseteq \Ker(\Trans(f))$. Hence if $\Ker(\Trans(f))$ is trivial, then $f$ is covering. However, the converse does not hold in general; see \cref{example: ker of trans does not characterize covering}.
    This is one of the motivations for our definition of the relative transvection group.
\end{rem}

\begin{prop}
\label{proposition: connected covering decomposition via quot of relTrans}
    Let $f\colon P\twoheadrightarrow Q$ be a surjective quandle homomorphism and $N=\relTrans(f)$. Then, $f$ factors as 
        \[\begin{tikzcd}[column sep=small]
        P \arrow[two heads]{rr}{f} \arrow[two heads]{rd}[swap]{\pi_N} & & Q\rlap{,} \\
        & P/N \arrow[two heads]{ru}[swap]{h_N}
    \end{tikzcd}\]
    where $\pi_N$ is connected and $h_N$ is covering. 
\end{prop}
\begin{proof}
    Since $\relTrans(f) \subseteq \relInn(f)$, the homomorphism $f$ factors through the quotient $P/N$ by \cref{prop:basic_property_of_quotient_by_normal_subgroup}. 
    The quotient map $\pi_N$ is connected by \cref{proposition:connected_iff_quotient_by_normal_subgroup_of_Inn}.
    
    It remains to show that the induced homomorphism $h_N\colon P/N \to Q$ is a covering.
    Let $[x], [y]\in P/N$ satisfy $h_N([x]) = h_N([y])$. 
    Then $f(x) = f(y)$, and hence
    \[
    s_xs_y^{-1}\in\relTrans(f) =N.
    \]
    Therefore, we have 
    \[
        s_{[x]}s_{[y]}^{-1} = (\pi_N)_*(s_xs_y^{-1}) = \id_{P/N}.
    \]
    It follows that $s_{[x]} = s_{[y]}$. Hence $h_N$ is a covering homomorphism.
\end{proof}

In particular, \cref{proposition: connected covering decomposition via quot of relTrans} together with \cref{theorem: quandle stein factorization} exhibits the non-uniqueness of $(\mathrm{connected}, \mathrm{covering})$-factorizations.

Nevertheless, the following proposition shows that this factorization is universal among all factorizations through covering homomorphisms.

\begin{prop}
\label{proposition: universality of connected-covering decomposition}
    Let $f\colon P\twoheadrightarrow Q$ be a surjective quandle homomorphism and $N=\relTrans(f)$.
    Suppose that $f$ factors as $f=h\circ \pi$, where 
    $\pi\colon P\twoheadrightarrow R$ is a general quandle homomorphism and 
    $h\colon R\twoheadrightarrow Q$ is a covering homomorphism.
    Then, there exists a unique quandle homomorphism $u\colon P/N \to R$ such that $\pi = u\circ \pi_N$ and $h_N= h \circ u$.
    \[\begin{tikzcd}
        P \arrow[rr,"f" description,rounded corners,
		      to path={(\tikztostart)
			-- ([yshift=2ex]\tikztostart.north)
			-- ([yshift=2ex]\tikztotarget.north)\tikztonodes
			-- (\tikztotarget)}]
        \arrow{r}{\pi_N} \arrow{rd}[swap]{\pi} & P/N \arrow{r}{h_N} \arrow[dashed]{d}{u} & Q \\
        & R \arrow{ru}[swap]{h} & 
    \end{tikzcd}\]
\end{prop}
\begin{proof}
    Suppose that $f= h\circ \pi$, where $\pi\colon P\to R$ is a homomorphism and $h\colon R\to Q$ is covering.
    Let $x, y\in P$ satisfy $f(x) = f(y)$.
    Since $h\circ \pi(x) = f(x) = f(y) = h\circ \pi(y)$ and $h$ is covering, we have $s_{\pi(x)} =s_{\pi(y)}$.
    Therefore, the calculation
    \[
        \pi\circ s_x\circ s_y^{-1} = s_{\pi(x)}\circ s_{\pi(y)}^{-1}\circ\pi = \pi 
    \]
    implies that $\pi\colon P\to R$ factors as $\pi=u\circ \pi_N$ for a unique homomorphism $u\colon P/N \to R$ by \cref{prop:basic_property_of_quotient_by_normal_subgroup}.
    Then, observing $h_N\circ \pi_N = f = h\circ \pi = h\circ u\circ \pi_N$, we see $h_N= h \circ u$ by the surjectivity of $\pi_N$.
\end{proof}

\begin{dfn}
\label{definition: maximal covering factor of surjection}
    For a quandle surjection $f\colon P\to Q$, we call the factorization constructed in \cref{proposition: connected covering decomposition via quot of relTrans} the \emph{maximal covering factorization} of $f$.
    We write $\covfac{f}$ for the induced covering homomorphism $P/\relTrans(f) \to Q$, and
    call it the \emph{(maximal) covering factor} of $f$.
\end{dfn}

We can prove that strongly connected homomorphisms are orthogonal to \emph{pre-covering} homomorphisms.

\begin{dfn}
\label{definition: pre-covering homomorphism}
    Let $f\colon P\to Q$ be a quandle homomorphism.
    We call $f$ \emph{pre-covering} if the condition $f(x)=f(y)$ implies $s_x=s_y$ for all $x,y \in P$.
\end{dfn}

In other words, pre-covering means covering without surjectivity.

\begin{prop}
\label{prop:orthogonal_property_between_strongly_connected_and_covering}
    The class of strongly connected homomorphisms is left orthogonal to the class of pre-covering homomorphisms.
    That is, for any commutative square
    \[\begin{tikzcd}
        A \arrow{r}{u} \arrow[two heads]{d}[swap]{e} & C \arrow{d}{m} \\
        B \arrow{r}[swap]{v} & D
    \end{tikzcd}\]
    where $e$ is strongly connected and $m$ is pre-covering, there exists a unique diagonal filler $w \colon B \to C$ such that $w \circ e = u$ and $m \circ w = v$.
\end{prop}

\begin{proof}
    Since $e$ is surjective, the diagonal map $w \colon B \to C$, if it exists, must be uniquely defined by $w(b) \coloneqq u(a)$ for any $a \in e^{-1}(b)$. To prove that $w$ is well-defined, let $a_1, a_2 \in A$ such that $e(a_1) = e(a_2)$. We must show that $u(a_1) = u(a_2)$. 
    
    Because $e$ is strongly connected, $\relTrans(e)$ acts transitively on the fibers of $e$. Thus, there exists an automorphism $\varphi \in \relTrans(e)$ such that $\varphi(a_1) = a_2$. By the definition of the relative transvection group, $\varphi$ can be expressed as a product of canonical generators, say $\textstyle \varphi = \prod_{i=1}^k (s_{x_i} s_{y_i}^{-1})^{\varepsilon_i}$, where $e(x_i) = e(y_i)$ and $\varepsilon_i \in \{\pm 1\}$ for each $i$. Then, we obtain
    \[
    u(a_2) = u(\varphi(a_1)) = \left( \prod_{i=1}^k (s_{u(x_i)} s_{u(y_i)}^{-1})^{\varepsilon_i} \right) (u(a_1)).
    \]
    From the commutativity of the square, $e(x_i) = e(y_i)$ implies $v(e(x_i)) = v(e(y_i))$, which in turn implies $m(u(x_i)) = m(u(y_i))$. Because $m$ is a pre-covering homomorphism, it follows that $s_{u(x_i)} = s_{u(y_i)}$ in $\Inn(C)$.
    Consequently, we have $s_{u(x_i)} s_{u(y_i)}^{-1} =\id$ and hence
    \[
    u(a_2) = u(a_1),
    \]
    confirming that $w$ is perfectly well-defined and satisfies $w\circ e=u$.
    This map $w$ is indeed a quandle homomorphism: for $b_1,b_2\in B$ and their respective lifts $a_1,a_2\in A$,
    \begin{align*}
        w(s_{b_1}(b_2)) &= w(s_{e(a_1)}(e(a_2))) = w(e(s_{a_1}(a_2))) = u(s_{a_1}(a_2)) \\
        &= s_{u(a_1)}(u(a_2)) = s_{w(b_1)}(w(b_2)).
    \end{align*}
    
    We have $m\circ w \circ e = m\circ u = v \circ e$. The surjectivity of $e$ implies $m\circ w = v$; hence $w$ is the desired filler.
    The uniqueness also follows directly from the surjectivity of $e$.
\end{proof}

\begin{rem}
\label{rem:strongly_connected_and_covering_do_not_form_OFS}
    While the classes of strongly connected homomorphisms and covering homomorphisms satisfy the orthogonality property, they do not form an orthogonal factorization system, in contrast to the $(\text{connected}, \text{rigid})$ factorization system. 
    
    This failure can be immediately deduced from the fact that quandle (pre-)covering homomorphisms are generally not closed under composition. In any valid orthogonal factorization system, both the left and right classes must be closed under composition. The lack of this closure property for covering homomorphisms implies that not every surjective homomorphism admits a (strongly connected, covering) factorization.
\end{rem}

\subsection{Relative Inner Automorphisms modulo Transvections}
\label{subsection: Relative Inner Automorphisms modulo Transvections}

\begin{dfn}
    For a surjective quandle homomorphism $f \colon P\twoheadrightarrow Q$, we write 
\[
\relInnmodTrans_f\coloneqq \relInn(f)/\relTrans(f).
\]
\end{dfn}

The group $\relInn(f)$ consists of inner symmetries of $P$ that induce identities on the base $Q$. Among them, $\relTrans(f)$ is generated by elementary vertical transvections $s_xs_y^{-1}$ with $x$ and $y$ lying in the same fiber. Thus, the quotient $H_f$ measures the residual relative symmetry which is not generated by vertical transvections.
%For this reason, we call $H_f$ the \memo{relative holonomy group???} of $f$.

\begin{eg}
    In the absolute case, with $Q=\{\ast\}$, the group $H_f$ is given by $H_f=\Inn(P)/\Trans(P)$. It is shown in \cite[Proposition 2.1 (ii)]{hulpke_stanovsky_vojtv_2016_connected_quandles_and_transitive_groups} that this group is cyclic, generated by $s_x$ for some $x\in P$.
    In general, however, the quotient group $H_f$ need not be cyclic. 
    Indeed, if $f\colon \Conj(G) \to \Conj(H)$ is the quandle homomorphism between conjugacy quandles induced by a group homomorphism $\phi\colon G\to H$, then we shall see in \cref{proposition: properties of morphism of conjugacy quandles} that $H_f \cong Z(H)/\phi(Z(G))$, which is not necessarily cyclic.
\end{eg}

Nevertheless, as the following proposition shows,
$H_f$ is always abelian.

\begin{prop}
\label{prop:H_f_is_contained_in_the_center_of_Inn/relTrans}
Let $f\colon P \to Q$ be a surjective quandle homomorphism.
    Then, the group $H_f = \relInn(f)/\relTrans(f)$ is contained in the center of $\Inn(P)/\relTrans(f)$. In particular, $H_f$ is always abelian.
\end{prop}

\begin{proof}
    Take any $\varphi\in \relInn(f)$ and $x\in P$.
    Using the fact that $\varphi$ is a quandle homomorphism, we have $\varphi s_x \varphi^{-1} s_x^{-1} = s_{\varphi(x)} s_x^{-1}$, which is an element of $\relTrans(f)$ since $f(\varphi(x))=f(x)$.
    Hence, in the quotient group $\Inn(P)/\relTrans(f)$, we have $[\varphi] [s_x] [\varphi]^{-1} [s_x]^{-1} = 1$. Because this holds for all generators $[s_x]$ of $\Inn(P)/\relTrans(f)$, it follows that $[\varphi]$ is in the center of $\Inn(P)/\relTrans(f)$.
\end{proof}

\begin{cor}
\label{cor:central_extension_associated_with_surjective_homs}
    For a surjective quandle homomorphism $f\colon P \to Q$, there is a short exact sequence
    \[ 1 \longrightarrow \relInn(f)/\relTrans(f) \longrightarrow \Inn(P)/\relTrans(f) \longrightarrow \Inn(Q) \longrightarrow 1, \]
    which is a central group extension.
\end{cor}

If $f$ is a covering homomorphism (i.e.~$\relTrans(f)=\{\id\}$), then it recovers \cref{proposition: covering hom is central extension}.

By \cref{lem:commutator_inn_rel_trans_rel}, the following group-theoretic description of $H_f$ holds for a connected homomorphism $f$.
\begin{prop}
\label{proposition: Hf and its vanishing for connected homs}
Let $f\colon P \to Q$ be a connected homomorphism.
Then, the natural homomorphism $\relInn(f)/[\Inn(P), \relInn(f)] \twoheadrightarrow H_f$ is an isomorphism. 
In particular, $H_f=1$ if and only if $\relInn(f)$ is $\Inn(P)$-perfect.
\end{prop}

\subsection{Pullbacks for Connected and Strongly Connected Homomorphisms}
\label{subsection:pullbacks_for_connectd_and_strongly_connected_hom}

In general, connected homomorphisms are not preserved by pullbacks; the homomorphism $R_4\to R_2$, $x\mapsto (x \mod 2$), between the dihedral quandles is connected, while the pullback of it along itself is not connected (see \cref{example: pullback of connected is not necessarily connected}).
%In this subsection, we determine the obstruction and show that connectedness is preserved by rigid pullbacks. On the other hand, strongly connectedness is stable under pullback along surjections.
In this subsection, we determine the obstruction to the preservation of connectedness under pullback and show, in particular, that connectedness is preserved under pullback along rigid homomorphisms.
On the other hand, strongly connected homomorphisms are preserved under pullback along surjective homomorphisms.

Let
\begin{equation}
    \begin{tikzcd}
        P' \arrow[two heads]{r}{g'} \arrow[two heads]{d}[swap]{f'} & P \arrow[two heads]{d}{f} \\
        Q' \arrow[two heads]{r}[swap]{g} & Q \arrow[phantom]{lu}[very near end]{\lrcorner}
    \end{tikzcd}\label{diagram:pullback_square_of_surjectives}
\end{equation}
be a pullback square of surjective quandle homomorphisms.

\begin{prop}
\label{prop:base_change_exact_sequence_for_inn_rel}
    For the pullback diagram \eqref{diagram:pullback_square_of_surjectives}, there exists a exact sequence
    \[
    1 \longrightarrow \relInn(f') \xlongrightarrow{g'_*} \relInn(f) \xlongrightarrow{\delta_{f,g}} \frac{\relInn(g)}{f'_*(\relInn(g'))} \longrightarrow 1.
    \]
\end{prop}

\begin{proof}
    Consider the following two short exact sequences and homomorphisms between them:
    \[\begin{tikzcd}
        1 \arrow{r} & \relInn(g') \arrow{r} \arrow{d}{f'_*} & \Inn(P') \arrow{r}{g'_*} \arrow[two heads]{d}{f'_*} & \Inn(P) \arrow{r} \arrow[two heads]{d}{f_*} & 1 \\
        1 \arrow{r} & \relInn(g) \arrow{r} & \Inn(Q') \arrow{r}[swap]{g_*} & \Inn(Q) \arrow{r} & 1\rlap{.}
    \end{tikzcd}\]
    By applying the snake lemma to this, we obtain an exact sequence
    \[ 1 \longrightarrow \relInn(f')\cap \relInn(g') \longrightarrow \relInn(f') \xlongrightarrow{g'_*} \relInn(f) \xlongrightarrow{\delta_{f,g}} \frac{\relInn(g)}{f'_*(\relInn(g'))} \longrightarrow 1. \]
    It remains to prove that $\relInn(f')\cap \relInn(g')$ is in fact trivial.
    If we take $\varphi' \in \relInn(f')\cap \relInn(g')$, then it satisfies $f'\circ \varphi'=f'$ and $g'\circ \varphi'=g'$ by definition. Since the identity also satisfies the same equations, the universal property of pullback forces $\varphi'=\id$. Hence we have $\relInn(f')\cap \relInn(g')=1$.
\end{proof}

\begin{cor}
\label{cor:connectedness_for_pullback}
    In the pullback diagram~\eqref{diagram:pullback_square_of_surjectives}, $f'$ is connected if and only if $\Ker(\delta_{f,g})$ acts transitively on every fiber of $f$.
\end{cor}

\begin{proof}
    For each $q' \in Q'$, the restriction of $g'$ induces an isomorphism $f'^{-1}(q') \cong f^{-1}(g(q'))$.
    Under this isomorphism, the action of $\relInn(f')$ precisely agrees with the action of $\Ker(\delta_{f,g}) \subseteq \relInn(f)$. Therefore, since $g$ is surjective, we conclude that the connectedness of $f'$ is equivalent to the transitivity of $\Ker(\delta_{f,g})$ on each fiber of $f$.
\end{proof}

\begin{prop}
\label{prop:base_change_trans_rel}
    Consider the pullback diagram \eqref{diagram:pullback_square_of_surjectives} of surjective quandle homomorphisms. Then, the restriction of $g'_*$ induces an isomorphism of the relative transvection groups:
    \[
    g'_* \colon \relTrans(f') \xrightarrow{\sim} \relTrans(f).
    \]
\end{prop}

\begin{proof}
    As seen in \cref{prop:base_change_exact_sequence_for_inn_rel}, $g'_*\colon \relInn(f') \to \relInn(f)$ is injective, and it is immediate to see that it induces an injective homomorphism $g'_* \colon \relTrans(f') \hookrightarrow \relTrans(f)$.
    It remains to show that every canonical generator of $\relTrans(f)$ lifts to $\relTrans(f')$. Let $s_x s_y^{-1} \in \relTrans(f)$ be a generator, where $x, y \in P$ lie in the same fiber, i.e., $f(x) = f(y) = q$. Since $g$ is surjective, there exists an element $q' \in Q'$ such that $g(q') = q$. 
    This allows us to take the elements $x'$ and $y'$ in the pullback $P'$ satisfying $(g'(x'),f'(x'))=(x, q')$ and $(g'(y'),f'(y'))=(y, q')$. We then have
    \[
    g'_*(s_{x'}s_{y'}^{-1}) = s_{g'(x')} s_{g'(y')}^{-1} = s_x s_y ^{-1}.
    \]
    Since $f'(x') = q' = f'(y')$, the element $s_{x'} s_{y'}^{-1}$ belongs to $\relTrans(f')$. Therefore, $g'_* \colon \relTrans(f') \to \relTrans(f)$ is surjective, and hence an isomorphism.
\end{proof}

\begin{cor}
\label{cor:connected_is_stable_under_rigid_base_change}
    Let $f \colon P \twoheadrightarrow Q$ and $g \colon Q' \twoheadrightarrow Q$ be surjective quandle homomorphisms, and suppose that $H_g = 1$ (this holds if $g$ is rigid, in particular). Then the pullback $f' \colon P \times_Q Q' \twoheadrightarrow Q'$ of $f$ along $g$ is connected if and only if $f$ is connected.
\end{cor}

\begin{proof}
    Interchanging $f$ and $g$, we see from \cref{prop:base_change_trans_rel} that $f'_*$ restricts to an isomorphism $\relTrans(g') \cong \relTrans(g)$.
    In particular, we have $\relTrans(g) \subseteq f'_*(\relInn(g'))$ and hence a surjection
    \[ H_g= \dfrac{\relInn(g)}{\relTrans(g)} \twoheadlongrightarrow \dfrac{\relInn(g)}{f'_*(\relInn(g'))}. \]
    Thus, when $H_g=1$, it follows from \cref{prop:base_change_exact_sequence_for_inn_rel} that the obstruction map $\delta_{f,g}$ is trivial and thus $\Ker(\delta_{f,g}) = \relInn(f)$.
    Therefore, by \cref{cor:connectedness_for_pullback}, $f'$ is connected if and only if $f$ is connected.
\end{proof}

\begin{prop}
\label{prop:strongly_connected_is_stable_under_surjective_pullback}
    Let $f \colon P \twoheadrightarrow Q$ and $g \colon Q' \twoheadrightarrow Q$ be surjective quandle homomorphisms. Then, the pullback $f' \colon P \times_Q Q' \twoheadrightarrow Q'$ is strongly connected if and only if $f$ is strongly connected.
\end{prop}

\begin{proof}
    Let $P' = P \times_Q Q'$ be the pullback. 
    By \cref{prop:base_change_trans_rel}, the restriction of the projection map $g' \colon P' \twoheadrightarrow P$ induces a canonical isomorphism $g'_* \colon \relTrans(f') \xrightarrow{\sim} \relTrans(f)$. 
    Furthermore, for each $q' \in Q'$, the projection $g'$ restricts to an isomorphism $f'^{-1}(q') \cong f^{-1}(g(q'))$.
    Under this isomorphism, the action of $\relTrans(f')$ precisely agrees with the action of $\relTrans(f)$. Therefore, the surjectivity of $g$ shows that the strong connectedness of $f'$ is equivalent to the strong connectedness of $f$.
\end{proof}

Here, we also observe the stable property of strongly connectedness under product.

\begin{prop}
\label{prop:strongly_connected_is_stable_under_product}
    Let $f \colon P \twoheadrightarrow Q$ and $f' \colon P' \twoheadrightarrow Q'$ be strongly connected homomorphisms. Then the product $f \times f' \colon P \times P' \twoheadrightarrow Q \times Q'$ is strongly connected.
\end{prop}

\begin{proof}
    We must show that the relative transvection group $\relTrans(f \times f')$ acts transitively on every fiber of $f \times f'$. The fiber over a point $(q, q') \in Q \times Q'$ is precisely the product $f^{-1}(q) \times f'^{-1}(q')$. 
    Let $(x, x')$ and $(y, y')$ be two arbitrary elements in this fiber. 

    First, since $f$ is strongly connected, we can take an automorphism $\varphi \in \relTrans(f)$ such that $\varphi(x) = y$.
    Write $\textstyle \varphi=\prod_{i=1}^k (s_{u_i}s_{v_i}^{-1})^{\varepsilon_i}$, where $f(u_i)=f(v_i)$ and $\varepsilon_i\in \{\pm 1\}$ for each $i$.
    We then define $\Phi\in \relTrans(f\times f')$ as
    \[ \Phi= \prod_{i=1}^k (s_{(u_i,x')}s_{(v_i,x')}^{-1})^{\varepsilon_i}. \]
    Then
    \[ \Phi = \prod_{i=1}^k (s_{u_i}s_{v_i}^{-1} \times s_{x'}s_{x'}^{-1})^{\varepsilon_i} = \prod_{i=1}^k (s_{u_i}s_{v_i}^{-1})^{\varepsilon_i} \times \id =\varphi\times \id. \]
    Hence we obtain $\Phi(x,x')=(\varphi\times \id)(x,x')=(\varphi(x),x')= (y,x')$.

    Symmetrically, we can take $\Phi' \in \relTrans(f\times f')$ such that $\Phi'(y,x')=(y,y')$. Thus, $\Phi'\circ\Phi(x,x')=(y,y')$, which proves that $f \times f'$ is strongly connected.
\end{proof}

\begin{rem}
The analogous assertion for connected homomorphisms is false in general.
\end{rem}

\begin{prop}
\label{proposition: connectedness of projection from product}
    Let $P$, $Q$ be non-empty quandles.
    The projection $\pr_2\colon P\times Q \to Q$ is connected if and only if $P$ is connected.
\end{prop}

\begin{proof}
    ($\Rightarrow$): As $Q$ is non-empty, we take an element $q\in Q$. Let $x, y \in P$ be any two elements. Considering $(x,q),(y,q) \in P\times Q$, by the connectedness, we get $\Phi\in \relInn(\pr_2)$ such that $\Phi(x,q)=(y,q)$. Let $\varphi\coloneqq (\pr_1)_*(\Phi)\in \Inn(P)$, where $\pr_1$ is the first projection. Then we have $\varphi(x)=y$, which proves $P$ is connected.

    ($\Leftarrow$): When $P$ is connected, the unique homomorphism $P\to \{\ast\}$ is strongly connected. Then by \cref{prop:strongly_connected_is_stable_under_surjective_pullback}, the pullback $\pr_2\colon P\times Q \to Q$ along the surjection $Q\to \{\ast\}$ is also strongly connected, and hence connected.
\end{proof}

\section{Doubly Transitive Homomorphisms}
\label{section: doubly transitive homomorphisms}

In this section, following the studies of doubly transitive quandles by Tamaru~\cite{tamaru_2013_twopoint_homogeneous_quandles_with_prime_cardinality}, Wada~\cite{wada_2015_twopoint_homogeneous_quandles_with_cardinality_of_prime_power}, and Vendramin~\cite{vendramin_2017_doubly_transitive_groups_and_cyclic_quandles}, we discuss the surjective quandle homomorphisms whose relative inner automorphism group acts $2$-transitively on every fiber.
To see this we first introduce the relative version of doubly-transitivity.

\begin{dfn}
\label{definition: doubly transitive quandle homomorphism}
    A quandle homomorphism $f\colon P \to Q$ is called \emph{doubly transitive} if it is surjective and for each $q\in Q$, the action of $\relInn(f)$ on the fiber $F_q=f^{-1}(q)$ is $2$-transitive.
\end{dfn}

\begin{eg}
    Let $P$ be a doubly transitive quandle and $Q$ be a trivial quandle.
    Then the relative inner automorphism group of the projection $P\times Q \to Q$ acts doubly transitively on each fiber.
\end{eg}

Since doubly transitive group actions are transitive, we have the following.

\begin{lem}
    If a quandle homomorphism $f\colon P \to Q$ is doubly transitive, then it is connected.
\end{lem}

\begin{rem}
    For the case of connectedness, fiber-wise connectedness implies the connectedness of homomorphisms (\cref{proposition:strongly_connected_can_be_verified_fiberwise}).
    On the other hand, it does not hold for doubly transitive morphisms; see \cref{example: fiber-wise doubly trans and not doubly trans}.
\end{rem} 

As above, we write the restriction map $\rho_q\colon \relInn(f) \to \Aut(F_q)$, $\varphi\mapsto \varphi\rvert_{F_q}$, via which $\relInn(f)$ acts on the fiber $F_q$.

The purpose of this section is to establish the following theorem.

\begin{thm*}[= \cref{theorem:classification_of_doubly_transitive_hom_with_finite_fibers}]
%\label{theorem: main theorem for doubly trans morph}
    Let $f\colon P\to Q$ be a doubly transitive homomorphism with finite fibers.
    Then each fiber is a trivial quandle or a connected Alexander quandle of prime power cardinality.
\end{thm*}

%We will see in \cref{example: doubly trans morph and triv and connected fiber} that every quandle appearing in the above classification actually arises as a fiber of a doubly transitive homomorphism:
We will see in \cref{example: doubly trans morph and triv and connected fiber} that both cases in the above classification can actually occur as fibers of a doubly transitive homomorphism:

\begin{eg}
    Conversely, let $\FB$ be a finite field, and $L_c$ the left multiplication of $c$.
    For every $c\in\FB^\times$, there exists a doubly transitive quandle homomorphism having a fiber isomorphic to 
    \[
    \Alex(\FB, L_c), \qquad s_x(y) = (1-c)x+cy.
    \]
    In particular, the construction realizes both trivial fibers, corresponding to $c=1$, and non-trivial connected Alexander fibers, corresponding to $c\neq1$.
\end{eg}

\Cref{example: doubly trans morph and triv and connected fiber} provides examples of doubly transitive homomorphisms in which trivial fibers (when they occur) have prime power cardinality. The authors do not know whether every finite trivial quandle occurs as a fiber:

\begin{question}
    For any positive integer $n$, does there exist a doubly transitive quandle homomorphism $f\colon P\to Q$ having a fiber isomorphic to the trivial quandle of order $n$?
\end{question}

\subsection{Group-Theoretic Preparations}
\label{subsection: group theoretic preparation}

We prepare some group-theoretic results that will be used later.

\begin{lem}
\label{lemma:doubly-transitive_induces_transitive_action_of_stabilizer}
%% see [Tamaru2013, Prop.3.3]
    Let $G$ be a group acting on a set $X$ and $G_x$ be the stabilizer at any $x\in X$. Assume $\lvert X \rvert \geq 3$. If the action $G \curvearrowright X$ is doubly transitive, then $G_x\curvearrowright X\setminus \{x\}$ is transitive.
\end{lem}

\begin{proof}
    If we take $y,y'\in X\setminus\{x\}$, then double transitivity implies that there exists $g\in G$ such that $g\cdot (x,y)=(x,y')$. This means that $g\in G_x$ and $g\cdot y=y'$.
\end{proof}

\begin{lem}
\label{lemma: commute bijection is id or no fixed point}
Let $G$ be a group that acts on a set $X$ transitively, and $f \colon X\to X$ be a bijection.
Assume that for any $g\in G$ and any $x\in X$, $f(g\cdot x) = g\cdot f(x)$.
If $f$ has a fixed point, then $f=\id$.
\end{lem}
\begin{proof}
    Let $a\in X$ be a fixed point of $f$. As the action is transitive, for any $x\in X$, there exists $g\in G$ such that $x = g\cdot a$. Then, we have
    \[
    f(x) = f(g\cdot a) = g\cdot f(a) = g\cdot a =x.
    \]
    This means $f= \id$.
\end{proof}

\begin{lem}[{\cite[Theorem 1.7]{book_cameron_1999_permutation_groups}}]
\label{lemma: normal subgrp of primitive action acts transitively}
Let $G$ be a group that acts faithfully and primitively on a set $X$, and $H\mathrel{\triangleleft} G$ be a non-trivial normal subgroup. Then, the induced action $H\curvearrowright X$ is transitive.
\end{lem}
\begin{proof}
    Let $\OO\coloneqq \{Hx \mid x\in X\}$ be the orbit decomposition of the action $H\curvearrowright X$.
    Note that there is the induced action $G\curvearrowright\OO$ because
    \[
    g(Hx) = gHg^{-1}\cdot gx = H(gx)
    \]
    by the normality.
    This action is transitive and thus $|Hx| = |Hy|$ for any $x, y\in X$, and it shows that $\OO$ forms a block system of $X$ with $G$-action.
    The primitivity of $G\curvearrowright X$ shows that either $|\OO| = |X|$ or $|\OO| = 1$ holds. 
    However, the former contradicts the non-triviality of $H$. This completes the proof.
\end{proof}

\begin{thm}[Burnside; see {\cite[Theorem 7.2E]{book_dixon_mortimer_1996_permutation_groups}} or {\cite[Theorem 4.3]{cameron_1981_finite_permutation_groups_and_finite_simple_groups}}]
    \label{theorem:Burnside's_theorem}
    % Cameron1981, Theorem 4.3:
    %% Let $G$ be a finite doubly transitive group and $N$ be a minimal normal subgroup of $G$. Then $N$ is either a regular elementary abelian group or a non-regular simple group.
    % DM96, Theorem 7.2 E:
    Let $G \leq Sym(X)$ be a finite $2$-transitive group. Then $\operatorname{soc}(G)$ is either (i) primitive and simple; or (ii) regular and elementary abelian.
\end{thm}

\begin{lem}
\label{lem:faithful_and_2-transitive_action_implies_center-trivial}
    Let $G$ be a group acting faithfully and $2$-transitively on a set $X$ with $\card{X}\geq 3$.
    Then the center $Z(G)$ is trivial.
\end{lem}

\begin{proof}
    Take $g\in Z(G)$ from the center and assume that $g\neq 1$.
    Since the action of $G$ is faithful, there is an element $x\in X$ such that $y\coloneqq gx\neq x$.
    By \cref{lemma:doubly-transitive_induces_transitive_action_of_stabilizer}, the action $G_x \curvearrowright X\setminus \{x\}$ is transitive.
    If we take $z \in X\setminus \{x,y\}$, then there is $h\in G_x$ such that $hy=z$. However, since $g$ is a central element, we have $z=hy=hgx=ghx=gx=y$, which is a contradiction.
\end{proof}

We say that a group $G$ is \emph{almost simple} if there exists a non-abelian simple group $S$ such that $S \leq G \leq \Aut(S)$.

\begin{lem}
\label{lem:almost_simple_center_trivial}
    Let $G$ be a finite almost simple group acting faithfully and doubly transitively on a finite set $X$ with $\card{X}\geq 4$. Then for any $x \in X$, the center of the point stabilizer $Z(G_x)$ is trivial.
\end{lem}

\begin{proof}
    In the case of the action of $G$ being $3$-transitive, $G_x$ acts faithfully and $2$-transitively on $X \setminus \{x\}$. Since $|X \setminus \{x\}| \ge 3$, it follows from \cref{lem:faithful_and_2-transitive_action_implies_center-trivial} that $Z(G_x)=\{1\}$.

    Hence, we may assume that $G$ is $2$-transitive but not $3$-transitive. Since $G$ is an almost simple group, it contains a non-abelian simple socle. This implies that its Fitting subgroup is trivial, i.e., $F(G) = \{1\}$. 
    By applying the elegant conceptual argument given by Vendramin \cite[Proof of Theorem 3]{vendramin_2017_doubly_transitive_groups_and_cyclic_quandles}%
    \footnote{See also G.~Robinson's answer on the mathoverflow question asked by L.~Vendramin: \url{https://mathoverflow.net/q/184682}.}%
    , which fundamentally relies on $F(G) = \{1\}$, we can conclude that $Z(G_x) = \{1\}$.
    (Alternatively, the triviality of $Z(G_x)$ can be directly confirmed by a case-by-case inspection of the point stabilizers using the known classification list of almost simple and $2$-transitive groups (see \cite[Section 7.4]{cameron_1995_permutation_groups} for example).)
\end{proof}

\begin{comment}
\memo{This can be deleted, as is not refed anywhere.}

\begin{prop}[Dixon-Mortimer:1996, Theorem 4.7 A.]
\label{prop:Dixon-Mortimer:1996_Theorem_4-7-A}
Let $G$ be a finite primitive group with an abelian socle (which is necessarily regular). Then $G$ has degree $p^n$ for some prime $p$ and some $n \geq 1$. If $V$ is a vector space of dimension $n$ over the field $\mathbb{F}_p$ with $p$ elements, then there is a subgroup $K\leq \GL(V)$ acting irreducibly on $V$ and an isomorphism of $G$ onto $V \rtimes K$ in which a point stabilizer of $G$ maps onto $K$.
\end{prop}
\end{comment}

\subsection{Classification of Finite Fibers}
\label{subsection: classification of doubly transitive action on fibers}

Before proceeding with the classification, we prepare some lemmas.

\begin{lem}
\label{lem:symmetry_map_at_x_is_in_the_center_of_stabilizer}
    Let $Q$ be a quandle. Then for any $x\in Q$, $s_x$ is in the center $Z(\Aut(Q)_x)$ of the stabilizer subgroup of $\Aut(Q)$ at $x$.
\end{lem}

\begin{proof}
    Let $\varphi\colon Q \to Q$ be an automorphism of $Q$. If $\varphi\in \Aut(Q)_x$ (i.e.~$\varphi(x)=x$), then we have $\varphi\circ s_x = s_{\varphi(x)} \circ \varphi= s_x \circ \varphi$. This shows that $s_x \in Z(\Aut(Q)_x)$.
\end{proof}

The next lemma shows that it is a rare case that the actions on fibers are $d$-transitive for $d \geq 3$.

\begin{lem}[{cf.{\cite[Proposition 5]{mccarron_2012_small_homogeneous_quandles}}, \cite[Lemma 7]{vendramin_2017_doubly_transitive_groups_and_cyclic_quandles}}]
\label{lemma: no higher transitive fiber}
    Let $f\colon P\to Q$ be a surjective quandle homomorphism, $q\in Q$, and $d\in\ZZ_{\ge 3}$. 
    If $F_q=f^{-1}(q)$ is non-trivial and has at least $4$ elements, then the action $\relInn(f) \curvearrowright F_q$ is not $d$-transitive.
\end{lem}
\begin{proof}
    The same proof as in {\cite[Proposition 5]{mccarron_2012_small_homogeneous_quandles}} and \cite[Lemma 7]{vendramin_2017_doubly_transitive_groups_and_cyclic_quandles} applies to our relative case.
    Put $H=\relInn(f)$.
    It is sufficient to show that the action $H \curvearrowright F_q$ is not $3$-transitive.
    Assume that it is $3$-transitive.
    Then the stabilizer $H_x$ acts $2$-transitively on $F_q\setminus\{x\}$ for any $x\in F_q$. In particular, the stabilizer $\Aut(F_q)_x$ also does, as it contains $\rho_q(H_x)$.
    It follows from \cref{lem:faithful_and_2-transitive_action_implies_center-trivial} that the center $Z(\Aut(F_q)_x)$ is trivial. However, as seen in \cref{lem:symmetry_map_at_x_is_in_the_center_of_stabilizer}, $s_x$ belongs to $Z(\Aut(F_q)_x)$, and hence $s_x= \id_{F_q}$ for all $x\in F_q$, which contradicts the non-triviality of $F_q$.
\end{proof}

The condition that a quandle homomorphism is doubly transitive strongly restricts the structure of its fibers. 
First, we observe that each fiber must be either a trivial quandle or a connected quandle using an elementary combinatorial argument.

\begin{thm}
\label{theorem: each fiber is trivial or connected as a subquandle}
    Let $f\colon P \to Q$ be a doubly transitive quandle homomorphism and $F_q$ be a fiber of $f$ with at least $3$ elements.
    Then the subquandle $F_q$ of $P$ is either a trivial quandle or a connected quandle.
\end{thm}

\begin{proof}
Fix an arbitrary point $x\in F_q$, put $H=\rho_q(\relInn(f))$, and let $H_x\coloneqq \{\varphi\in H \mid \varphi(x) = x\}$ be the stabilizer at $x$. Since $H\curvearrowright F_q$ is $2$-transitive, by \cref{lemma:doubly-transitive_induces_transitive_action_of_stabilizer} the action $H_x\curvearrowright F_q\setminus\{x\}$ is transitive.
%(see {\cite[Proposition 3.3]{tamaru_2013_twopoint_homogeneous_quandles_with_prime_cardinality}})

Consider the map $s_x\colon F_q\setminus\{x\}\to F_q\setminus\{x\}$, which is $H_x$-equivariant because  
\(
\varphi\circ s_x(y) = s_{\varphi(x)}(\varphi(y)) =s_x\circ \varphi(y)
\)
for all $\varphi \in H_x$ and all $y\in F_q$.
Thus, \cref{lemma: commute bijection is id or no fixed point} yields two possibilities for $s_x$:
\begin{enumerate}
    \item\label{case: s_x is identity} $s_x = \id$, or
    \item\label{case: s_x has no fixed point} $s_x$ has no fixed point on $F_q\setminus\{x\}$.
\end{enumerate}

In the case of \eqref{case: s_x is identity}, the quandle structure of $F_q$ must be trivial; indeed, for any $y\in F_q$, as the action $H\curvearrowright F_q$ is transitive, we can pick $\varphi\in H$ such that $y=\varphi(x)$.
Since 
\(
\varphi \circ s_x = s_{\varphi(x)}\circ \varphi=s_y\circ \varphi
\)
and $s_x = \id$, we have
\(
s_y = \varphi s_x\varphi^{-1} = \id.
\)
Therefore, $F_q$ is a trivial quandle.

Next, let us consider the case \eqref{case: s_x has no fixed point}.
Since $H\subseteq \Aut(F_q)$ acts $2$-transitivly on $F_q$, so is the action $\Aut(F_q)\curvearrowright F_q$, and hence it is primitive. By the assumption of \eqref{case: s_x has no fixed point}, the normal subgroup $\Inn(F_q)\triangleleft \Aut(F_q)$ has the non-trivial element $s_x$.
Thus, \cref{lemma: normal subgrp of primitive action acts transitively} shows that the action $\Inn(F_q)\curvearrowright F_q$ is transitive, which means that $F_q$ is connected.
\end{proof}

While \cref{theorem: each fiber is trivial or connected as a subquandle} provides a topological dichotomy of the fibers, we can completely determine their precise algebraic structure by applying classification results of finite permutation groups. 
Specifically, using Burnside's theorem on finite $2$-transitive groups, we show that any such fiber is an Alexander quandle.

\begin{prop}
\label{proposition: sufficient cond for alex when doubly trans containing inns}
Let $Q$ be a finite quandle. Suppose that there exists a subgroup $G \le \Aut(Q)$ containing the inner automorphism group $\Inn(Q)$ such that $G$ acts $2$-transitively on $Q$.
If $G$ has a non-trivial abelian normal subgroup, then $Q$ is isomorphic to an Alexander quandle of prime power cardinality.
\end{prop}

\begin{proof}
    Since $G$ acts $2$-transitively on $Q$, it is primitive. By assumption, $G$ contains a non-trivial abelian normal subgroup. A fundamental theorem in finite group theory (see \cite[Theorem 4.7A]{book_dixon_mortimer_1996_permutation_groups}) states that a finite primitive group with an abelian normal subgroup is of affine type. Thus, the set $Q$ can be identified with a vector space $V = (\ZZ/p\ZZ)^n$ over a prime field $\mathbb{F}_p$ for some prime $p$ and integer $n \ge 1$. Furthermore, $G$ embeds into the affine general linear group $\AGL(V) = V \rtimes \GL(V)$ such that $G = V \rtimes G_0$, where the point stabilizer $G_0 \le \GL(V)$ of the zero vector $0 \in V$ acts transitively on $V \setminus \{0\}$.

    Under this identification, the translation map $T_x(v) = v + x$ for any $x \in V$ belongs to $G$. Since $G$ is a subgroup of $\Aut(Q)$ containing $\Inn(Q)$ and hence $T_x$ is a quandle automorphism of $Q$, we have $T_x \circ s_0 \circ T_x^{-1} = s_{T_x(0)} = s_x$. From this, for $x, y \in V$, we can explicitly compute $s_x(y)$ as
    \[
        s_x(y) = (T_x \circ s_0 \circ T_x^{-1})(y) = s_0(y - x) + x = (1-s_0)(x) + s_0(y).
    \]
    This representation shows that $Q$ is the Alexander quandle of the abelian group $V$ and the automorphism $s_0$.
\end{proof}

\begin{thm}
\label{theorem: fiber of doubly transitive is connected Alexander quandle}
    Let $f \colon P \to Q$ be a doubly transitive quandle homomorphism. 
    Assume that the fiber $F_q = f^{-1}(q)$ at $q\in Q$ is a finite set of cardinality at least $4$. 
    Then $F_q$ is an Alexander quandle.
\end{thm}

\begin{proof}
    Fix $q \in Q$ and let $F = F_q$. By the definition of a doubly transitive quandle homomorphism, the relative inner group $\relInn(f)$ acts $2$-transitively on $F$. 
    Let us define the subgroup $G \le \Aut(F)$ by
    \[
        G = \langle \rho_q(\relInn(f)), \Inn(F) \rangle.
    \]
    Since $G$ contains $\rho_q(\relInn(f))$, the action of $G$ on $F$ is also $2$-transitive. 
    Note that, since every element $g \in G$ is a quandle automorphism of $F$, it satisfies $g \circ s_x \circ g^{-1} = s_{g(x)}$ for all $x \in F$.

    By Burnside's theorem (\cref{theorem:Burnside's_theorem}), the socle of the finite $2$-transitive group $G$, denoted by $\operatorname{soc}(G)$, is either (i) primitive and simple, or (ii) regular and elementary abelian. We consider these two cases separately.

    \textbf{Case (i):} In this case, $\operatorname{soc}(G)$ is a non-abelian simple group. A finite group with a non-abelian simple socle is almost simple (cf.~\cite[Section 4.3]{book_dixon_mortimer_1996_permutation_groups}). 
    Since $G$ is an almost simple group acting faithfully and $2$-transitively on the finite set $F$ with $|F| \ge 4$, we can apply \cref{lem:almost_simple_center_trivial} to conclude that $Z(G_x) = \{1\}$. 
    For any $g \in G_x$, we have $g(x) = x$, which implies $g \circ s_x \circ g^{-1} = s_x$. Thus $s_x \in Z(G_x)$.
    Because $s_x \in Z(G_x)$, this forces $s_x = \id_{F}$ for all $x \in F$. This means $s_x(y) = y$ for all $x, y \in F$, so $F$ is a trivial quandle. A trivial quandle is trivially an Alexander quandle (with the automorphism being the identity map).

    \textbf{Case (ii):} In this case, $G$ contains a non-trivial abelian normal subgroup. Since $F$ is a finite quandle and $G$ is a $2$-transitive group such that $\Inn(F) \leq G \leq \Aut(F)$, we can use \cref{proposition: sufficient cond for alex when doubly trans containing inns} to conclude that $F$ is an Alexander quandle.
\end{proof}

\begin{thm}
\label{theorem:classification_of_doubly_transitive_hom_with_finite_fibers}
    Let $f\colon P\to Q$ be a doubly transitive homomorphism with finite fibers.
    Then each fiber is a trivial quandle or a connected Alexander quandle of prime power cardinality.
    More precisely, a nontrivial fiber is isomorphic to $\Alex(F_{p^m}^n, L_\lambda)$, where $p$ is a prime number, $\lambda\in\FB_{p^m}^{\times}$, and $L_\lambda$ is the multiplication of $\lambda$.
\end{thm}

\begin{proof}
    Let $q\in Q$ and consider the fiber $F_q=f^{-1}(q)$ over $q$. If $\card{F_q}=1$ or $2$, then the only possible quandle structure is trivial.

    If $\card{F_q}=3$, then there are, up to isomorphism, three possible quandle structures on a three-element set. By \cref{theorem: each fiber is trivial or connected as a subquandle}, $F_q$ is either trivial or connected. Hence, if $F_q$ is nontrivial, it is necessarily isomorphic to the dihedral quandle $R_3$, which is a connected Alexander quandle (of prime power cardinality).

    Finally, suppose that $\card{F_q} \geq 4$. By combining \cref{theorem: each fiber is trivial or connected as a subquandle}, \cref{proposition: sufficient cond for alex when doubly trans containing inns}, and \cref{theorem: fiber of doubly transitive is connected Alexander quandle}, it follows that $F_q$ is a connected Alexander quandle.
    The explicit description is provided by \cref{theorem: bonatto characterization of doubly homogeneous quandles}. 
\end{proof}

\section{Examples and Counterexamples}
\label{section: examples and counterexamples}

In this section, we construct some examples of connected or non-connected quandle homomorphisms.

\subsection{Conjugacy Quandles}
\label{subsection: example cases of conjugacy quandles}

In this subsection, we consider quandle homomorphisms obtained from group surjective homomorphisms.

Taking conjugacy quandles of groups, we have a functor $\Conj \colon \Grp\to\Quandle$, which preserves surjections.
Throughout this subsection, let $\phi\colon G \twoheadrightarrow H$ be a surjective group homomorphism, and consider $f\coloneqq \Conj(\phi)\colon \Conj(G) \twoheadrightarrow \Conj(H)$.
Write $K\coloneqq \Ker(\phi)$ and $A\coloneqq \phi^{-1}(Z(H))$, where $Z(H)$ is the center of a group $H$, i.e., $Z(H)=\{x\in H \mid xy=yx \text{ for all }y\in H\}$.
Note that by surjectivity, $\phi$ preserves central elements, so it induces a group homomorphism $Z(\phi)=\phi\rvert_{Z(G)}\colon Z(G) \to Z(H)$ and the inclusion $Z(G)\subseteq A=\phi^{-1}(Z(H))$.

Recall that $\Inn(\Conj(G)) = \Inn(G) \coloneqq \{i_g \mid g\in G\}$, where $i_g$ is given by $i_g(x)=gxg^{-1}$. Also note that $\Trans(\Conj(G))=\{ i_g i_h^{-1} = i_{gh^{-1}} \mid g,h\in G \} = \{ i_g \mid g \in G \} = \Inn(G)$. It is easy to see that the induced group homomorphism
\[ f_* = \Inn(f) \colon \Inn(G) \to \Inn(H) \]
sends $i_g$ to $i_{\phi(g)}$.

\begin{prop}
\label{proposition: properties of morphism of conjugacy quandles}
    With the notations above, the following hold:
    \begin{enumerate}
        \item\label{item: conj mor inn} $\Inn(\Conj(G)) \cong G/Z(G)$.
        \item\label{item: conj mor relinn} $\relInn(f) = \{ i_g \mid g \in G \text{ such that } \phi(g) \in Z(H)\} \cong A/Z(G)$.
        \item\label{item: conj mor reltrans} $\relTrans(f)= \{i_k \mid k\in \Ker(\phi)\} \cong K /(K\cap Z(G))$.
        \item\label{item: conj mor Hf} $H_f \cong A/KZ(G) \cong Z(H)/\phi(Z(G))$.
    \end{enumerate}
\end{prop}
\begin{proof}
        Since $i_g =\id_G$ is equivalent to $g\in Z(G)$, \eqref{item: conj mor inn} follows immediately.
        %(\cite[Theorem 2.4]{elhamdadi_macquarrie_restrepo_2012_automorphism_groups_of_quandles})

        \eqref{item: conj mor relinn}: It holds that $i_g\in \relInn(f)$ if and only if $\phi(gxg^{-1}) = \phi(x)$ for all $x\in G$, which is equivalent to saying that $\phi(g)$ is in the center $Z(H)$ of $H$, because $\phi$ is surjective. Therefore, we have 
        \[
        \relInn(f) = \{ i_g \mid g \in \phi^{-1}(Z(H))\} \cong A/Z(G),
        \]
        where the last isomorphism follows from \eqref{item: conj mor inn}.

        \eqref{item: conj mor reltrans}: First, we claim that $\relTrans(f) = \{i_k \mid k\in \Ker(\phi)\}$.
        The left-hand side is contained in the right-hand side, since $s_{g'}s_{g}^{-1}= i_{g'}i_g^{-1}=i_{g'g^{-1}}$. 
        As $i_k = s_ks_1^{-1}$, the equality holds.
        Thus, there is a surjective homomorphism $K\to \relTrans(f)$, $k \mapsto i_k$, and its kernel is 
        \[
        \{k\in K \mid i_k =\id\} = K\cap Z(G).
        \]
        Thus, we have $\relTrans(f)\cong K /(K\cap Z(G)) (\cong KZ(G)/ Z(G))$.

        \eqref{item: conj mor Hf}: We obtain the first isomorphism
        by \eqref{item: conj mor relinn} and \eqref{item: conj mor reltrans}. 
        For the second isomorphism, we identify the isomorphism 
        \[
        A/K \cong Z(H)
        \]
        which is induced by the surjection $\phi\rvert_A\colon A\to Z(H)$.
        Under this identification, the subgroup $KZ(G)/K \subseteq A/K$ corresponds to $\phi(Z(G))\subseteq Z(H)$.
        This completes the proof.
\end{proof}

\begin{prop}
\label{prop:connected-rigid-covering_property_for_Conj(phi)}
Under the notations in this subsection, the following hold.
    \begin{enumerate}
        \item\label{item: prop of eg of conj conn iff iso} $f$ is (strongly) connected if and only if $\phi$ is an isomorphism. 
        \item\label{item: prop of eg of conj rigidity} $f$ is rigid if and only if $\phi^{-1}(Z(H)) \subseteq Z(G)$.
        \item\label{item: prop of eg of conj covering} $f$ is covering if and only if $\phi$ is a central extension, i.e., $\Ker(\phi)\subseteq Z(G)$.
    \end{enumerate}
\end{prop}

\begin{proof}
    \eqref{item: prop of eg of conj conn iff iso}: When $\phi$ is an isomorphism, it is trivial that $f=\Conj(\phi)$ is (strongly) connected. Suppose $f$ is connected and consider the fiber $F_1 = \Ker\phi$ at the unit $e_H \in H$. As inner automorphisms fix the unit $e_G$, the transitivity of $\relInn(f)$ on $F_1$ implies $\Ker\phi = \{e_G\}$, and hence $\phi$ is bijective.
    \eqref{item: prop of eg of conj rigidity}: This follows from \eqref{item: conj mor relinn} of \cref{proposition: properties of morphism of conjugacy quandles}.
    \eqref{item: prop of eg of conj covering}: This holds, since $f(x)= f(y)$ if and only if $xy^{-1}\in K$.
\end{proof}

\begin{prop}
\label{prop: prop of ef of conj homogeneous}
With the notations above, $f$ is homogeneous if and only if $\phi$ is a central extension (and hence if and only if $f$ is covering).
\end{prop}

\begin{proof}
    ($\Rightarrow$): Suppose that $f$ is homogeneous. 
    We first show $\relAut(f)\cdot e_G \subseteq K\cap Z(G)$.
    Let $\alpha\in\relAut(f)$. Since $f(\alpha(e_G)) = f(e_G) = \phi(e_G) = e_H$, we have $\alpha(e_G)\in K=\Ker(\phi)$. 
    For any $g\in G$, taking $x\in G$ such that $\alpha(x)=g$, we obtain
    \[ i_g(\alpha(e_G))=s_{\alpha(x)}(\alpha(e_G)) = \alpha(s_x(e_G)) = \alpha(e_G). \]
    Thus $\alpha(e_G)\in Z(G)$, and we have $\relAut(f)\cdot e_G \subseteq K\cap Z(G)$.
    
    Since $f$ is homogeneous, $\relAut(f)$ acts transitively on the fiber $K$ of $f$ over $f(e_G)=e_H$, which means $K= \relAut(f)\cdot e_G$. Therefore, we have $K = \relAut(f)\cdot e_G \subseteq Z(G)$ and $\phi$ is a central extension.

    ($\Leftarrow$): Suppose that $K\subseteq Z(G)$. 
    For each $k\in K$, we write $L_k$ for the left multiplication map of $k$.
    For every $x,y\in G$, the centrality of $k$ yields
    \[ s_{L_k(x)}(L_k(y)) = s_{kx}(ky) = kxkyx^{-1}k^{-1} = kxyx^{-1} = L_k(s_x(y)). \]
    Therefore $L_k$ is an automorphism of $\Conj(G)$. Since $k\in K=\Ker(\phi)$, we also have
    \[ f(L_k(x)) = \phi(kx) = \phi(x) = f(x) \]
    for all $x\in G$, and hence $L_k\in\relAut(f)$.

    Let $x,y\in G$ lie in the same fiber of $f$. Then $ \phi(yx^{-1}) = \phi(y)\phi(x)^{-1} = e_H$, so $yx^{-1}\in K$. Then $L_{yx^{-1}}(x) = y$. Thus, $\relAut(f)$ acts transitively on every fiber of $f$, and hence $f$ is homogeneous.
\end{proof}

\begin{cor} 
\label{corollary: homogeneous conjugacy quandles} 
The conjugacy quandle $\Conj(G)$ is homogeneous if and only if $\Conj(G)$ is trivial, and if and only if $G$ is abelian.
\end{cor}

\begin{eg}
\label{example: non-trivial rigid hom}
    By \cref{prop:connected-rigid-covering_property_for_Conj(phi)}, for a central extension $\phi\colon G \twoheadrightarrow H$ with $Z(H)$ trivial (e.g.\ $\pi\colon SL(2, 3) \to A_4$), we see that $\relInn(f)$ is trivial, and hence $f=\Conj(\phi)$ is rigid.
\end{eg}

\begin{eg}
\label{example: ker of trans does not characterize covering}
    Let $G=D_4 = \langle r, s \mid r^4 = s^2=1, srs = r^{-1} \rangle$, the dihedral group of order $8$. 
    Note that $Z(D_4) = \langle r^2 \rangle$ and $D_4/Z(D_4)\cong \ZZ/2\ZZ\times \ZZ/2\ZZ =D_2$, the dihedral group of order $4$. 
    Consider the surjection $\phi\colon D_4\to D_4/Z(D_4) = D_2$ and $f=\Conj(\phi)\colon P\coloneqq\Conj(D_4) \to \Conj(D_2)\eqqcolon Q$. Here $Q$ is the trivial quandle with $4$ elements.
    Since $\relTrans(f) = Z(D_4)/Z(D_4)\cap Z(D_4) = \{1\}$ by \cref{proposition: properties of morphism of conjugacy quandles}, $f$ is covering by \cref{lemma: relTrans vanish iff covering}.
    On the other hand, we have 
    \begin{align*}
        \Ker(\Trans(P) \xrightarrow{f_*} \Trans(Q)=\{\ast\})&= \Trans(P)=\Inn(P) = D_4/Z(D_4) \\
        &\cong \ZZ/2\ZZ\times \ZZ/2\ZZ \neq\{1\}.
    \end{align*}
    This example shows that the triviality of the kernel of the induced homomorphism
    $\Trans(P)\to \Trans(Q)$ is not equivalent to $f$ being covering.
\end{eg}

\begin{eg}
\label{example: composition of homogeneous is not homogeneous}
Let $Q_8$ be the quaternion group and $V_4$ be the Klein four-group. 
Consider surjections $Q_8 \twoheadrightarrow Q_8/\{\pm1\}\cong V_4 \twoheadrightarrow \{*\}$ and their conjugacy quandles.
As $Z(Q_8) = \{\pm1\}$ and $V_4$ is abelian, these surjections are homogeneous.
However, as $Q_8$ is not abelian, the composition is not homogeneous.
\end{eg}

\begin{eg}
\label{example: fiber-wise homogeneous but not homogeneous}
Let us consider the group homomorphism $\sgn\colon S_3\twoheadrightarrow\ZZ/2\ZZ$ and the induced quandle homomorphism $f\coloneqq \Conj(\sgn)$. By \cref{prop: prop of ef of conj homogeneous}, $f$ is not homogeneous, but two fibers are both homogeneous.
\end{eg}

\begin{comment}
%%%% なんかあたりまえのことしか言えなかった感じなので消す．

At the end of this subsection, we describe an explicit form of the covering factor of these morphisms.

The orbit of the action of $\relTrans(f)$ of $g\in G$ is the $K$-conjugacy class $K\cdot g\coloneqq \{kgk^{-1} \mid k\in K\}$. 
The set of all $K$-conjugacy classes is denoted by
\[
G//K\coloneqq \{ K\cdot g \mid g\in G \}.
\]
\begin{lem}\label{lemma: qdl structure of K-conjugacy classes}
    The set $G//K$ becomes a quandle with the following structure
    \[
    s_{K\cdot y}(K\cdot x) \coloneqq K\cdot(yxy^{-1}).
    \]
\end{lem}
\begin{proof}
    \todo
\end{proof}

\begin{prop}\label{proposition: covering factor of morph between conjugacy quandles}
    Under the notations in this subsection, we have
    \[
        \covfac{f}\colon \Conj(G)/\relTrans(f)\cong G//K \to H.
    \]
\end{prop}
\begin{proof}
    \todo
\end{proof}
\end{comment}

\subsection{Generalized Alexander Quandles}
\label{subsection: generalized alexander quandles}

In this subsection we discuss typical types of quandle homomorphisms between generalized Alexander quandles (see \cref{example: generalized alex and alex}). In particular, we determine their connectedness.

The construction of generalized Alexander quandles gives rise to a functor $\GAlex(-) \colon \GrpAut \to \Quandle$ from the category of groups with automorphisms.
Let $G,H$ be groups and let $\sigma\in\Aut(G)$, $\tau\in\Aut(H)$ be their automorphisms.
Then for a surjective group homomorphism $\Phi\colon G\twoheadrightarrow H$ satisfying $\Phi\circ\sigma = \tau\circ\Phi$, we have a surjective quandle homomorphism $f=\GAlex(\Phi)\colon \GAlex(G, \sigma) \to \GAlex(H, \tau)$.

Note that not every quandle homomorphism between generalized Alexander quandles can be written in this form. For example, the multiplication map $L_g\colon G \to G$, $x\mapsto gx$, is a quandle homomorphism $L_g\colon \GAlex(G,\sigma) \to \GAlex(G,\sigma)$; indeed, for $x,y \in \GAlex(G,\sigma)$, we calculate
\begin{align*}
    s_{L_g(x)}(L_g(y)) &= s_{gx}(gy) = gx\cdot \sigma((gx)^{-1}(gy)) = gx\cdot \sigma(x^{-1}g^{-1}gy) \\ 
    &= gx\cdot \sigma(x^{-1}y) = L_g(s_x(y)). 
\end{align*}
However, unless $g=e_G$, $L_g$ does not preserve the unit elements, so it cannot be written as $\GAlex(\Phi)$.

\begin{prop}
\label{proposition: hom between GAlexes from group surj is homogeneous}
With notations in this subsection, $f\coloneqq \GAlex(\Phi)$ is homogeneous.
\end{prop}
\begin{proof}
Let $K\coloneqq \Ker(\Phi)$.
We show that $L(K)\coloneqq \{L_x \mid x\in K\}\subseteq \relAut(\GAlex(G, \sigma))$ acts transitively on every fiber of $f$.
Let $x,y\in G$ lie in the same fiber of $f$. 
Then 
$\Phi(yx^{-1}) = \Phi(y)\Phi(x)^{-1} = e_H,$ 
so $yx^{-1}\in K$. 
Therefore $L_{yx^{-1}}(x) = y$. 
This proves transitivity. 
\end{proof}

We begin by computing the inner automorphism group of generalized Alexander quandles. 
The result is already known; see \cite[Lemma 2.7]{bonatto_2020_principal_and_doubly_homogeneous_quandles} and \cite[Proposition 3.8]{higashitani_kurihara_2024_generalized_alexander_quandles_of_finite_groups_and_their_characterizations}. We here include the proof to reconcile differences in notation and for completeness.

It is immediate to observe that $s_e=\sigma$ and for any $x\in \GAlex(G,\sigma)$, $s_x = L_{x\sigma(x)^{-1}} \circ \sigma$ as maps.
Hence, for $g\in G$, the automorphism $L_{g\sigma(g)^{-1}} = s_x \circ s_e^{-1}$ is in fact an inner automorphism of $\GAlex(G,\sigma)$.

Let $D(G, \sigma)\coloneqq [G,\sigma]\coloneqq  \langle g\sigma(g)^{-1} \mid g\in G\rangle$ be the subgroup of $G$ generated by the elements of the form $g\sigma(g)^{-1}$. Then the injective group homomorphism $L_{(-)}\colon G \to \Aut(\GAlex(G,\sigma))$, $g\mapsto L_g$, induces $L_{(-)}\colon D(G,\sigma) \to \Inn(\GAlex(G,\sigma))$.
We can prove the following.

\begin{prop}
\label{proposition: inner of generalized alex}
    For a group $(G,\sigma)$ with an automorphism, we have the equality
    \[
    \Inn(\GAlex(G,\sigma)) = \{ L_z\circ\sigma^n \mid z\in D(G, \sigma), n\in\ZZ \} \cong D(G,\sigma) \rtimes \langle\sigma\rangle.
    \]
\end{prop}
\begin{proof}
    Put $S\coloneqq \{ L_z\circ\sigma^n \mid z\in D(G, \sigma), n\in\ZZ \}$.
    As already observed, $L_{g\sigma(g)^{-1}}$ and $\sigma$ both belong to $\Inn(\GAlex(G,\sigma))$. Hence, $S\subseteq \Inn(\GAlex(G,\sigma))$.

    Since the generators $s_x = L_{x\sigma(x)^{-1}}\circ \sigma$ of $\Inn(\GAlex(G,\sigma))$ are in $S$, it is sufficient to show that $S$ is a subgroup of $\Aut(\GAlex(G,\sigma))$.
    We remark that $D(G, \sigma)$ is $\sigma$-invariant, because
    \(
    \sigma(g\sigma(g)^{-1}) = \sigma(g)\sigma(\sigma(g))^{-1}\in D(G, \sigma). 
    \)
    For $z,w \in D(G,\sigma)$ and $n,m\in \ZZ$, we have
    \begin{align*}
        (L_z\circ\sigma^{n})\circ (L_w\circ\sigma^{m})(x) &= z\sigma^{n}(w\sigma^m(x))\\
        &= z\sigma^n(w)\sigma^{n+m}(x)\\
        &= L_{z\sigma^n(w)}\sigma^{n+m}(x)
    \end{align*}
    for all $x\in \GAlex(G,\sigma)$.
    As $x\sigma^m(y)\in D(G, \sigma)$ by the discussion above, we see that $S$ is closed under compositions.
    Also, we have
    \[
    (L_z\circ\sigma^n)^{-1}(x) = \sigma^{-n}\circ L_{z^{-1}}(x) = \sigma^{-n}(z^{-1}x)= L_{\sigma^{-n}(z^{-1})}\circ \sigma^{-n}(x)
    \]
    for all $x\in \GAlex(G,\sigma)$.
    Hence $L_z\circ\sigma^n \in S$, because $\sigma^{-n}(z^{-1}) \in D(G,\sigma)$.
    Therefore, $S\subseteq \Aut(\GAlex(G,\sigma))$ is a subgroup, which proves the claim. The last isomorphism is immediate.
\end{proof}

\begin{prop}
\label{lemma: transvection of generalized alex}
    For a group $(G,\sigma)$ with an automorphism, we have
    \[
    \Trans(\GAlex(G, \sigma)) = \{L_x \mid x\in D(G, \sigma)\} \cong D(G,\sigma).
    \]
\end{prop}
\begin{proof}
    Put $T = \{L_x \mid x\in D(G, \sigma)\}$.
    For the equality, as $L_{g\sigma(g)^{-1}} = s_g s_e^{-1}$, $T$ is contained in the left hand side.
    Since $s_x = L_{x\sigma(x)^{-1}}\circ \sigma$, we have
    \[ s_xs_y^{-1} = L_{x\sigma(x)^{-1}}\circ\sigma\circ\sigma^{-1}\circ L_{y\sigma(y)^{-1}}^{-1} 
    = L_{x\sigma(x)^{-1}\sigma(y)y^{-1}} \in T \]
    for $x,y \in \GAlex(G,\sigma)$. Hence the converse inclusion also holds. The last isomorphism is clear.
\end{proof}

Next, using the above calculation, we investigate the connectedness of the surjective homomorphism $f=\GAlex(\Phi)\colon \GAlex(G, \sigma) \twoheadrightarrow \GAlex(H, \tau)$ where $\Phi\colon (G,\sigma)\twoheadrightarrow (H,\tau)$ is surjective.
Note that $f \circ L_g = L_{\Phi(g)} \circ f$ for any $g \in G$.
One can observe that
\[ \Inn(f) \colon \Inn(\GAlex(G,\sigma)) \twoheadrightarrow \Inn(\GAlex(H,\tau)) \]
sends $L_z \circ \sigma^n$ to $L_{\Phi(z)} \circ \tau^n$. 
Let $K=\Ker(\Phi)$ be the kernel. 

\begin{prop}
\label{proposition: relative inner of morph between generalized alex}
    We have
    \[
    \relInn(f) = \{ L_z\circ\sigma^n \mid z\in D(G,\sigma)\cap K,\> n \in \ZZ \text{ such that } \tau^n = \id_H \}.
    \]
\end{prop}
\begin{proof}
    Put $U = \{ L_z\circ\sigma^n \mid z\in D(G,\sigma)\cap K,\> n \in \ZZ \text{ such that } \tau^n = \id_H \}$.
    The inclusion $\relInn(f)\supseteq U$ is clear.
    For the converse, take any $L_z\circ\sigma^n\in\Inn(\GAlex(G, \sigma))$.
    Since 
    \(
    f_*(L_z\circ\sigma^n) = L_{\Phi(z)}\circ \tau^n, 
    \)
    we have
    \[
    f_*(L_z\circ\sigma^n)(e_H) = L_{\Phi(z)}\circ \tau^n(e_H) = L_{\Phi(z)}(e_H) = \Phi(z).
    \]
    Thus, for $L_z\circ\sigma^n$ to belong to $\relInn(f)$, it is necessary that 
    $z\in K = \Ker\Phi$. Then $f_*(L_z\circ\sigma^n) = \tau^n$ implies the other condition.
\end{proof}

If we write $C\coloneqq \{ \sigma^m\in\langle \sigma \rangle \mid m\in \ZZ \text{ such that } \tau^m = \id_H \}$, \cref{proposition: relative inner of morph between generalized alex} shows
\[
    \relInn(f)\cong (D(G,\sigma)\cap K) \rtimes C.
\]

\begin{prop}
\label{proposition: connectedness for morph between generalized alexes}
    With the notations above, $f$ is connected if and only if $K= \Ker\Phi\subseteq D(G, \sigma)$.
\end{prop}
\begin{proof}
    The fiber of $e_H$ is $K$ and the orbit of $e_G$ of the action of $\relInn(f)$ is $D(G, \sigma)\cap K$.
    Thus, if $f$ is connected, then $K=K\cap D(G, \sigma)$, which is equivalent to saying that $K\subseteq D(G, \sigma)$.

    For the converse, take two elements $x,y$ in the same fiber $Kg=f^{-1}(f(g))$.
    If $K\subseteq D(G,\sigma)$, then $L_{yx^{-1}}$ is in $\relInn(f)$, as $yx^{-1}\in K$. Since $L_{yx^{-1}}(x) = yx^{-1}x = y$, this shows that $f$ is connected.
\end{proof}

Before proceeding with the calculation of a relative transvection group, we introduce a notation:
\[
N(K,\sigma)\coloneqq \langle gk\sigma(k)^{-1}g^{-1} \mid g\in G, k\in K=\Ker\Phi \rangle,
\]
which is the normal closure of $D(K,\sigma)=\langle k\sigma(k)^{-1} \mid k \in K \rangle$.

\begin{prop}
\label{proposition: relative transvection of generalized alex}
    With the notations above, we have
    \[
    \relTrans(f) = \{ L_u \mid u\in N(K,\sigma)\} \cong N(K,\sigma).
    \]
    In particular, $f$ is strongly connected if and only if $N(K, \sigma) = K$.
\end{prop}
\begin{proof}
    As seen in \cref{lemma: transvection of generalized alex}, we have the equality $s_xs_y^{-1} = L_{x\sigma(x)^{-1}\sigma(y)y^{-1}}$.
    If $x,y\in \GAlex(G, \sigma)$ are in the same fiber of $f$, then there is $k\in K$ such that $x=yk$. Then
    \begin{align*}
        x\sigma(x)^{-1}\sigma(y)y^{-1}
        &= (yk)\sigma(yk)^{-1}\sigma(y)y^{-1} \\
        &= yk\sigma(k)^{-1}\sigma(y)^{-1}\sigma(y)y^{-1} \\
        &= yk\sigma(k)^{-1}y^{-1} \in N(K,\sigma). 
    \end{align*}
    Thus we have $\relTrans(f) = \{ L_u \mid u\in N(K,\sigma)\}$. The last isomorphism is clear.
\end{proof}

\begin{cor}
The group $H_f$ is isomorphic to $(D(G, \sigma)\cap K)/N(K,\sigma) \rtimes C$.
\end{cor}

\begin{prop}
\label{prop:rigid-covering_property_for_GAlex(Phi)}
Under the notations in this subsection, the following hold.
    \begin{enumerate}
        \item\label{item: prop of eg of GAlex rigidity} $f$ is rigid if and only if both $C$ and $D(G,\sigma)\cap K$ are trivial.
        \item\label{item:necessary_condition_for_GAlex_rigidity} If $f$ is rigid, then $\Phi^{-1}(\Fix(H,\tau)) \subseteq \Fix(G,\sigma)$ (in fact, the equality holds).
        \item\label{item: prop of eg of GAlex covering} $f$ is covering if and only if $\Ker\Phi \subseteq \Fix(G,\sigma)$, i.e., $\sigma \rvert_{\Ker\Phi}$ is the identity.
    \end{enumerate}
\end{prop}

\begin{proof}
    Remark that $g \in \Fix(G,\sigma)=\{x\in G \mid \sigma(x)=x \}$ is equivalent to $g\sigma(g)^{-1}=e$.

    \eqref{item: prop of eg of GAlex rigidity}: It follows from \cref{proposition: relative inner of morph between generalized alex}.

    \eqref{item:necessary_condition_for_GAlex_rigidity}: If $f$ is rigid, then $D(G,\sigma)\cap \Ker\Phi=\{e\}$.
    To deduce the desired inclusion, we observe that $g\in \Phi^{-1}(\Fix(H,\tau))$ if and only if $\Phi(g\sigma(g)^{-1}) = \Phi(g)\tau(\Phi(g))^{-1}=e$, which is equivalent to $g\sigma(g)^{-1} \in D(G,\sigma)\cap\Ker\Phi$.
    Hence, if $g\in \Phi^{-1}(\Fix(H,\tau))$, the rigidity forces $g\sigma(g)^{-1}=e$, which means $g \in \Fix(G,\sigma)$.
    
    \eqref{item: prop of eg of GAlex covering}: By \cref{proposition: relative transvection of generalized alex}, $f$ is covering if and only if $N(K,\sigma)=\{e\}$, or equivalently, $D(K,\sigma)=\{e\}$. The latter condition can be rewritten as $\Ker\Phi\subseteq \Fix(G,\sigma)$.
\end{proof}

\begin{eg}
    If $\tau$ has an infinite order, then $C$ is trivial, and hence the relative inner automorphism group is $\relInn(f) = \{L_z\mid z\in D(G, \sigma)\cap K \}$.
    In particular, we have
    \[
    H_f \cong (D(G, \sigma)\cap K)/N(K,\sigma).
    \]
\end{eg}

At the end of this subsection, we describe an explicit form of the covering factor of these homomorphisms.

\begin{eg}\label{example: covering factor of hom between generalized alexes}
    Under the notation of this subsection, we identify the covering factor of $f=\GAlex(\Phi)\colon P\coloneqq \GAlex(G, \sigma) \twoheadrightarrow \GAlex(H, \tau)\eqqcolon Q$ with the following homomorphism:
    \[
        \covfac{f}\colon \GAlex(G/N,\overline{\sigma})\longrightarrow Q, 
    \]
    where $N\coloneqq N(K,\sigma)$.
    
    First note that $N$ is a normal subgroup of $G$ and is $\sigma$-invariant by the definition.
    Thus, there is an induced automorphism 
    \[
    \overline{\sigma}\colon G/N \to G/N\quad gN\mapsto \sigma(g)N.
    \]
    As $\relTrans(f)= \{L_u \mid u\in N\}\cong N$ by \cref{proposition: relative transvection of generalized alex}, the $\relTrans(f)$-orbit of $g\in G$ is $Ng=gN$. 
    Consequently, the map (between sets)
    \[
    \alpha\colon P/\relTrans(f) \to \GAlex(G/N, \overline{\sigma}), \qquad [g]\mapsto gN
    \]
    is well-defined and bijective.
    We claim that the map $\alpha$ is a quandle homomorphism.
    For $x, y\in G$, we have
    \[
    \alpha(s_{[x]}([y])) = x\sigma(x^{-1}y)N = (xN)\overline{\sigma}((xN)^{-1}(yN)) = s_{\alpha([x])}(\alpha([y])).
    \]
    Therefore, $\alpha$ is indeed a quandle isomorphism.

    Thus, $f$ factors as $P= \GAlex(G, \sigma)\xrightarrow[]{\pi'} \GAlex(G/N, \overline{\sigma}) \xrightarrow[]{\overline{\Phi}} \GAlex(H, \tau) = Q$, where the second homomorphism is induced by the natural surjection $\overline{\Phi}\colon G/N \to H$ and it is covering by the definition of $\relTrans(f)$.
    Under the isomorphism $\alpha$,
    \[
    \covfac{f}\cong \overline{\Phi}\colon\GAlex(G/N, \overline{\sigma}) \longrightarrow Q.
    \]
\end{eg}

\subsection{Linear Alexander Quandles}
\label{subsection: linear alexander quandles in examples section}

Recall from \cref{example: linear Alexander quandle} that a linear Alexander quandle $\Lambda_{n,a}$ is a quandle whose underlying set is $\ZZ/n\ZZ$ and whose quandle structure is given by $s_x(y)= (1-a)x + ay$, where $\gcd(a,n)=1$.
By translating \cref{proposition: inner of generalized alex} and \cref{lemma: transvection of generalized alex} in the case of linear Alexander quandles, we have:

\begin{lem}
\label{lemma:Inn_and_Trans_of_linear_Alexander}
    For a linear Alexander quandle $\Lambda_{n, a}$, we have 
    \begin{align*}
        \Inn(\Lambda_{n, a}) &= \{x\mapsto a^k x + b \mid b\in (1-a)\Lambda_{n, a}\}\text{, and} \\
        \Trans(\Lambda_{n,a}) &= \{x\mapsto x+b \mid b\in(1-a)\Lambda_{n,a}\}\cong \Image(1-a).
    \end{align*}
\end{lem}

Thus, by \cref{lemma: connectedness is also determined by trans}, we obtain an equivalent condition for the connectedness of linear Alexander quandles:

\begin{prop}[{\cite[Corollary 7.2]{hulpke_stanovsky_vojtv_2016_connected_quandles_and_transitive_groups}}]\label{proposition: connectedness of lin alex}
    A linear Alexander quandle $\Lambda_{n,a}$ is connected if and only if $1-a$ is surjective, which is equivalent to $\gcd(1-a, n)=1$.
\end{prop}

The surjective homomorphisms arising from the construction in
\cref{subsection: generalized alexander quandles}
take the following form in the linear case.
Let \(p=nq\), and consider linear Alexander quandles
\(\Lambda_{p,a}\) and \(\Lambda_{q,c}\).
Any surjective homomorphism between their underlying additive groups is of the form
\[
    \Phi_u\colon \ZZ/p\ZZ \longrightarrow \ZZ/q\ZZ,
    \qquad
    x\longmapsto ux \pmod q,
\]
where \(\gcd(u,q)=1\).
The compatibility of \(\Phi_u\) with the automorphisms
\(x\mapsto ax\) and \(x\mapsto cx\) is equivalent to
\[
    a\equiv c\pmod q.
\]
After composing \(\Phi_u\) with the quandle automorphism
\(x\mapsto u^{-1}x\) of \(\Lambda_{q,c}\), we may assume that \(u=1\).
Moreover, if \(a\equiv c\pmod q\), then
\(\Lambda_{q,c}=\Lambda_{q,a}\).
Thus, up to an isomorphism of the target, it suffices to consider the
reduction homomorphism below. We determine when this homomorphism is connected.

\begin{prop}\label{example: relinn of morph between lin alexes}
    Let $P = \Lambda_{p,a}$ and $Q = \Lambda_{q,a}$ be linear Alexander quandles where $p,q,a$ are integers such that $p = nq$ for an $n\in\ZZ_{>0}$ and $a$ is coprime to $p,q$.
    Consider the quandle homomorphism $f \colon P\to Q$ defined by $f(x) = x \pmod q$.
    Then, $f$ is connected if and only if $\gcd(1-a, p)$ divides $q$.
\end{prop}
    \begin{proof}
    Let us determine $\relInn(f)$.
    Any $\varphi\in\Inn(P)$ can be written as $\varphi(x) = a^kx+b$ for some $k \in \ZZ$ and $b\in(1-a)P$ (see \cref{lemma:Inn_and_Trans_of_linear_Alexander}).
    Thus, $\varphi\in\relInn(f)$ if and only if $f_{*}(\varphi)= \id_Q$, and if and only if 
    \begin{align*}
        \begin{cases}
            a^k \equiv 1 &\pmod q, \text{ and }\\
            %(1-a)b\equiv 
            b \equiv 0 &\pmod q.
        \end{cases}
    \end{align*}
    Thus, we have
    \begin{align}
        \relInn(f) = \{ x\mapsto a^kx+b \mid \ord_{q}(a)\text{ divides } k \text{ and } q\text{ divides } b, b\in(1-a)P\}. 
        \label{equation: relInn of Alexes}
    \end{align}
    %In such a case, the homomorphism $f$ is connected if and only if $\gcd(1-a, p)$ divides $q$.
    %Indeed, 
    The fiber $F_x$ of $x\in Q$ is
    \[
        F_x = \{ x, x+q,x+2q, \ldots, x+(n-1)q\} \subseteq \ZZ/p\ZZ, 
    \]
    and the orbit is
    \[
    \relInn(f)\cdot x = \{  x+ r\mid r\in(1-a)P\cap qP \}.
    \]
    Thus, the orbit coincides with $F_x$ if and only if $qP\subseteq (1-a)P$, which is equivalent to saying that $\gcd(1-a, p)$ divides $q$.
    This completes the proof.
\end{proof}

\begin{prop}
\label{proposition: strongly connectedness between lin alex}
With the same setting in \cref{example: relinn of morph between lin alexes}, $f$ is strongly connected if and only if $\gcd(1-a, n) = 1$.
\end{prop}
\begin{proof}
    When $f(y) = f(z)$, it can be written as $y=z+mq$ for some $m\in\ZZ$.
    As $s_z^{-1}(x) = (1-a^{-1})z + a^{-1}x$, we have 
    \begin{align*}
    s_ys_z^{-1}(x) &= (1-a)y + a((1-a^{-1})z + a^{-1}x) \\
    &= x + (1-a)(y-z)\\
    &= x+ (1-a)mq.
    \end{align*}
    Thus, we have $\relTrans(f) = \{x \mapsto x+ t \mid t\in (1-a)qP\}$, and therefore, $f$ is strongly connected if and only if $(1-a)qP = qP$, which is equivalent to the condition $\gcd(1-a,n)=1$.
\end{proof}

With these results in hand, we present several more concrete examples.

\begin{eg}
\label{example: pullback of connected is not necessarily connected}
Let $f\colon R_4\to R_2$ be the surjection between dihedral quandles.
By \cref{example: relinn of morph between lin alexes}, $f$ is a connected homomorphism.

Consider the pullback of $f$ along itself, denoted by $f'$. Then $f'$ is not connected: we can see that $f'_*\colon \Inn(R_4\times_{R_2}R_4) \twoheadrightarrow \Inn(R_4)$ is an isomorphism, and hence $\relInn(f')$ is trivial, while every fiber of $f'$ has two elements.
\end{eg}

\begin{eg}
\label{example: connected morphism whose fiber is trivial}
    There exists a connected homomorphism with a non-connected fiber.
    In the settings of \cref{example: relinn of morph between lin alexes}, 
    suppose that $q\ge 2$, $n=q$, $p=q^2$, and $a=1+q$. Note that $P$ and $Q$ are not connected, since $\gcd(1-a,q) = \gcd(1-a, p)=q\neq 1$.

    We claim that the homomorphism $f$ is connected.
    As $Q$ is the trivial quandle with $q$ points and \cref{lemma:Inn_and_Trans_of_linear_Alexander}, 
    \[
    (1-a)P = q\ZZ/q^2\ZZ = \Trans(P) \subset \Inn(P) = \relInn(f).
    \]
    Moreover, each fiber is a coset of $q\ZZ/q^2\ZZ$ and $(1-a)P$ acts on it as $x\mapsto x+kq$. Therefore, the action of $\relInn(f)$ on each fiber is transitive. 

    Every fiber of $f$ is a trivial quandle as a subquandle of $P$; indeed, for elements $x, y\in f^{-1}(c)$ in the same fiber, we have
    \[  s_x(y) = (1-a)x + ay = -qx +(1+q)y = q(y-x) + y = y \]
    in $q\ZZ/q^2\ZZ$.
    So, no fiber of $f$ is connected.
\end{eg}

\begin{eg}\label{example: eg of relInnmodTrans neq 0 by using Alexes}
    Under the settings of \cref{example: relinn of morph between lin alexes} and 
    the assumption $\gcd(1-a,q) = \gcd(1-a, p)=1$.
    Thus, we have $\relTrans(f)\cong \ZZ/n\ZZ$.
    Moreover, the quotient $\relInnmodTrans_f=\relInn(f)/\relTrans(f)$ is a cyclic group of order 
    \[
    \lvert\relInnmodTrans_f\rvert = \frac{\ord_p(a)}{\ord_q(a)}.
    \]
\end{eg}

\begin{eg}
    In \cref{example: eg of relInnmodTrans neq 0 by using Alexes}, set $p = 21$, $q=7$, $a = 2$.
    As $\displaystyle \lvert\relInnmodTrans_f\rvert=\frac{\ord_{21}(2)}{\ord_7(2)} = \frac{6}{3} =2$, 
    $\relInnmodTrans_f \cong \ZZ/2\ZZ$ is nontrivial.
    Moreover, every fiber is isomorphic to $\Lambda_{3,2}$, which is connected. 
    Since every fiber is connected, $f$ is strongly connected by \cref{proposition:strongly_connected_can_be_verified_fiberwise}. Thus, a strongly connected homomorphism may have non-trivial $H_f$. In particular, strong connectedness does not imply that $\relInn(f)$ is $\Inn(P)$-perfect.
\end{eg}

\begin{eg}\label{example: fiber-wise doubly trans and not doubly trans}
    There exists a surjective homomorphism that has doubly transitive fibers (viewed as subquandles of the domain) but that is not doubly transitive in the sense of \cref{definition: doubly transitive quandle homomorphism}.

    Let $f\colon P\to Q$ be the morphism for the case of $(q,n,a) = (3, 5, 2)$ in \cref{example: relinn of morph between lin alexes}.
    As the fiber 
    \[
    F_i = f^{-1}(i) = \{i, i+3, i+6, i+9, i+12\}
    \]
    is isomorphic to the linear Alexander quandle $\Lambda_{5, 2}$, it is a doubly transitive quandle (see \cite[Theorem 4.3]{wada_2015_twopoint_homogeneous_quandles_with_cardinality_of_prime_power}).
    We have already calculated the relative inner automorphism group:
    \[
    \relInn(f) = \{x\mapsto 4^kx+3b \mid k\in\ZZ, b\in\ZZ/5\ZZ\}.
    \]
    It acts on each fiber as $x\mapsto \pm x+b$, which is not doubly transitive; the stabilizer $(\rho_0(\relInn(f)))_0$ at $0$ is $\{\pm 1\}$ and its orbits on $(\ZZ/5\ZZ)\setminus\{0\}$ are $\{1, 4\}$ and $\{2, 3\}$. Hence, $f$ is not doubly transitive in view of \cref{lemma:doubly-transitive_induces_transitive_action_of_stabilizer}. 
\end{eg}

    In \cref{section: doubly transitive homomorphisms}, we have proved that the fibers of a doubly transitive quandle homomorphism are severely restricted: each fiber is either a trivial quandle or a connected Alexander quandle of prime-power order. The following examples show that both possibilities actually occur: the fibers can be trivial, and they can also be connected Alexander quandles.
    
\begin{eg}
\label{example: doubly trans morph and triv and connected fiber}
    Let $Q = \{1, \ldots, m\}$ be the trivial quandle with $m$ elements.
    Let $p$ be a prime number, $k$ a positive integer, and $\FB\coloneqq \FB_{p^k}$.

    Consider a quandle $P$ whose underlying set is $P\coloneqq Q\times \FB$ and whose quandle structure is given by 
    \[
        s_{(i,a)}(j, b) = (j, \lambda_ib+(1-\lambda_i)a), 
    \]
    where $\lambda_i\in\FB^{\times}$ for $i\in Q$.%
    \footnote{This quandle is an example of a quandle module over the trivial quandle $Q$.}
    Moreover, we assume that $\langle\lambda_i \mid i\in Q\rangle = \FB^{\times}$ as a group.

    We check here that $P$ is actually a quandle.
    As each $s_{(i, a)}$ preserves fibers and behaves as an Alexander quandle, $s_{(i, a)}$ is a bijection.
    The idempotency $s_{(i, a)}(i, a) = (i, a)$ follows easily.
    For $(i, a), (j, b), (l,c) \in P$, we have
    \begin{align*}
        s_{(i,a)}\circ s_{(j, b)}(l, c) &= s_{(i, a)}(l, \lambda_j c+(1-\lambda_j)b)\\
        &= (l, \lambda_i\lambda_jc+\lambda_i(1-\lambda_j)b +(1-\lambda_i)a)\\
        &=(l, \lambda_j(\lambda_ic+(1-\lambda_i)a)+(1-\lambda_j)(\lambda_ib+(1-\lambda_i)a))\\
        &=s_{(j, \lambda_ib+(1-\lambda_i)a)} (l, \lambda_ic+(1-\lambda_i)a)\\
        &=s_{s_{(i,a)}(j, b)}\circ s_{(i, a)}(l, c).
    \end{align*}

    Consider the map $f\colon P \to Q$ defined by $(i,a) \mapsto i$. It is a quandle homomorphism.
    \begin{claim*}
        The quandle homomorphism $f$ is a doubly transitive homomorphism.
    \end{claim*}
    \begin{proof}
        The fiber of $i\in Q$ is $F_i \coloneqq f^{-1}(i) \cong\Alex(\FB_{p^k},L_{\lambda_{i}})$, where $L_{\lambda_{i}}$ is the multiplication operator by $\lambda_i$.
        As $Q$ is a trivial quandle, we have $\relInn(f) = \Inn(P)$. Thus, it is enough to show that the action $\Inn(P)\curvearrowright F_i$ is $2$-transitive.
        Take any $(i, a_1), (i, a_2), (i, b_1), (i, b_2) \in F_i$ such that $a_1\neq a_2$, $b_1\neq b_2$.
        Put
        \[
        \alpha\coloneqq \frac{b_2 - b_1}{a_2-a_1}\in\FB^{\times}, \qquad \beta\coloneqq b_1 - \alpha a_1.
        \]
        As $\{\lambda_i\}_{i\in Q}$ generates the multiplicative group $\FB^{\times}$, $\alpha$ has a representation of the form 
        \[
        \alpha = \lambda_{i_1}^{\varepsilon_1}\cdots\lambda_{i_r}^{\varepsilon_r}, 
        \]
        where $\varepsilon_i\in\{1, -1\}$.
        Let $\psi\coloneqq s_{(i_1, 0)}^{\varepsilon_1}\circ\cdots\circ s_{(i_r, 0)}^{\varepsilon_r} \in \Inn(P)$. 
        Note that, for $(n,x)\in P$, $\psi(n,x) = (n,\alpha x)$ because $s_{(i, 0)}(n,x) = (n,\lambda_{i}x)$.
        
        Fix $j\in Q$ such that $\lambda_j \neq 1$; such $j$ exists because of the assumption $\langle\lambda_i \mid i\in Q\rangle = \FB^{\times}$.
        Put $\varphi\coloneqq s_{\left(j, (1-\lambda_j)^{-1}\beta\right)}\circ s_{(j, 0)}^{-1} \in \Inn(P)$.
        Then, it acts on each fiber as 
        \begin{align*}
            \varphi(n, x) &= s_{\left(j, (1-\lambda_j)^{-1}\beta\right)}(n, \lambda_j^{-1}x)\\
            &= \left(n, \lambda_j\lambda_j^{-1}x + (1-\lambda_j)(1-\lambda_j)^{-1}\beta\right)\\
            &= (n, x+\beta).
        \end{align*}
        With these notations, let $\Phi\coloneqq \varphi\circ\psi$. Then we have
        \[ \Phi(i, a_1) = (i, \alpha a_1+\beta) = (i, \alpha a_1 +b_1-\alpha a_1) = (i, b_1), \]
        and
        \[ \Phi(i, a_2) = (i, \alpha a_2+\beta) = \left(i, \frac{b_2-b_1}{a_2-a_1}a_2 + b_1 -\frac{b_2-b_1}{a_2-a_1}a_1\right) = (i, b_2). \]
        This proves that $f$ is doubly transitive.
    \end{proof}
\end{eg}

\begin{eg}
    In \cref{example: doubly trans morph and triv and connected fiber}, set $\FB=\ZZ/7\ZZ$, $m=3$, $\lambda_1 = -1$, $\lambda_2 = 2$, and $\lambda_3 = 1$.
    As $\langle -1, 2, 1\rangle = \FB_7^{\times}$, the homomorphism $f\colon P\to Q$ is doubly transitive.
    While the fibers $F_{1}=\Lambda_{7,-1}$ and $F_2=\Lambda_{7,2}$ are connected but not doubly transitive quandles as subquandles of $P$ by \cite[Theorem 4.3]{wada_2015_twopoint_homogeneous_quandles_with_cardinality_of_prime_power}, the quandle structure of $F_3=\Lambda_{7,1}$ is trivial.
\end{eg}

\bibliographystyle{amsalpha}
\bibliography{converted_bibtex_quandles_tyoshida}

\end{document}